\documentclass[11pt, a4paper, english]{amsart}

\usepackage{amsmath, amsthm, amssymb, amsfonts}
\usepackage[unicode,colorlinks=true,linkcolor=blue,urlcolor=blue]{hyperref}
\usepackage{bm}
\usepackage{dsfont}
\usepackage{color}
\usepackage{geometry}
\usepackage{listings}
\usepackage{mathrsfs, enumerate}
\usepackage[textsize=tiny]{todonotes}

\usepackage{tikz}
\usetikzlibrary{arrows,shapes,automata}
\tikzstyle{vertex}=[circle,fill=black!25,minimum size=20pt,inner sep=0pt]
\tikzstyle{selected vertex} = [vertex, fill=red!24]
\tikzstyle{edge} = [draw,thick,-]
\tikzstyle{weight} = [font=\small]
\tikzstyle{selected edge} = [draw,line width=5pt,-,red!50]
\tikzstyle{ignored edge} = [draw,line width=5pt,-,black!20]

\usepackage{graphicx}
\usepackage{subcaption}
\usepackage{lipsum}
\newtheorem{theorem}{Theorem}[section]
\newtheorem{remark}[theorem]{Remark}
\newtheorem{claim}[theorem]{Claim}

\newtheorem{proposition}[theorem]{Proposition}

\newtheorem{definition}[theorem]{Definition}
\newtheorem*{notation*}{Notation}  %%%The * is used to avoid numbering

\newcommand\mydots{\hbox to 1em{.\hss.\hss.}}

\newcommand\Nnn{\hbox to 2em{$N$\hss$-$\hss$1$}}

\newcommand\Nnk{\hbox to 2em{$N$\hss$-$\hss$k$}}

\newcommand\PX{\hbox to 4em{$\mathcal{P}(\mathcal{X})^{T+1}$}}

\newcommand\UU{\hbox to 0.8em{$\widehat{\mathcal{R}}$}}

\newcommand\EEE{\hbox to 1em{$\widehat{\mathcal{E}}_t$}}

\definecolor{ofi_c}{rgb}{0.0, 0.42, 0.24}

\newcounter{hypcount}
\newenvironment{hypenv}{\begin{enumerate}\setcounter{enumi}{\value{hypcount}}}{\setcounter{hypcount}{\value{enumi}}\end{enumerate}}

\newcounter{hypcount2}
\newenvironment{Hypenv}{\begin{enumerate}\setcounter{enumi}{\value{hypcount2}}}{\setcounter{hypcount2}{\value{enumi}}\end{enumerate}}

\newcommand{\hypref}[1]{\textbf{(A\ref*{#1})}}
\newcommand{\hyprefall}{\textbf{(A1)\,--\,(A\arabic{hypcount})}}
\newcommand{\Hypref}[1]{\textbf{(R\ref*{#1})}}
\newcommand{\Hyprefall}{\textbf{(R1)\,--\,(R\arabic{hypcount2})}}
\newcounter{dummy}

\title[Correlated equilibria for mean field games with progressive strategies]{Correlated equilibria for mean field games with progressive strategies
}

\author{Ofelia Bonesini}
\address[Ofelia Bonesini]{Department of Mathematics, Imperial College London, London SW7 1NE, UK. \newline
	\indent Dipartimento di Matematica “Tullio Levi-Civita", Università degli studi di Padova,\newline
	\indent Via Trieste 63, 35121 Padova, Italy.}
\email[Ofelia Bonesini]{bonesini@math.unipd.it/obonesin@ic.ac.uk}%

\author{Luciano Campi}
\address[Luciano Campi]{Dipartimento di Matematica “Federigo Enriques”, Università degli studi di Milano,\newline
	\indent Via Cesare Saldini 50, 20133, Milano, Italy.}
\email[Luciano Campi]{luciano.campi@unimi.it}%

\author{Markus Fischer}
\address[Markus Fischer]{Dipartimento di Matematica “Tullio Levi-Civita", Università degli studi di Padova,\newline
	\indent Via Trieste 63, 35121 Padova, Italy}
\email[Markus Fischer]{fischer@math.unipd.it}%

\date{\today}

\numberwithin{equation}{section}
\allowdisplaybreaks

\begin{document}

\begin{abstract}
	In a discrete space and time framework, we study the mean field game limit for a class of symmetric $N$-player games based on the notion of correlated equilibrium. We give a definition of correlated solution that allows to construct approximate $N$-player correlated equilibria that are robust with respect to progressive deviations. We illustrate our definition by way of an example with explicit solutions.

\end{abstract}
\maketitle
%
%\smallskip
%
\smallskip


\smallskip

\section{Introduction}

	Building on \cite{CF2022}, we consider correlated equilibria for a simple class of symmetric finite horizon $N$-player games in discrete time and their natural mean field game counterpart as the number of players $N$ goes to infinity.

	\mbox{MFGs} is the acronym %/short form 
	for \emph{mean field games} and refers to a stream of  literature in game theory extremely popular nowadays whose origins are quite recent. 
	Indeed, MFGs were introduced nearly at the same time but independently by \cite{huangetal06} and \cite{lasrylions07} in the mid 2000's.
	 %Roughly speaking/ 
	In a nutshell, MFGs are limit systems for symmetric stochastic $N$-player games with mean field interaction for $N \to \infty$.
	Thanks to the mean field interaction among the players, a kind of law of large numbers (known as propagation of chaos), one expects the empirical distribution of the players' states to converge as $N \to \infty$ to the law of some representative player.
	In the limit, the concept of Nash equilibrium translates into a two-step solution where (i) the representative player reacts optimally to the measure flow representing the distribution of the whole population, and (ii) the latter arises as aggregation of all such identical players' best responses at equilibrium.	
	The reader interested in a broad yet detailed overview on the topic from a probabilistic viewpoint is referred to the two-volume book by Carmona and Delarue \cite{carmonadelarue18}. 	
	
	The connection between MFGs and their finite-player counterpart can be established in two ways. Crucial is the choice of the type of strategies the players are allowed to play.
	On one hand, a solution to the \mbox{MFG} can be exploited in order to build approximate Nash equilibria for $N$-player games. See, e.g., \cite{carmonadelarue13,gomesetal13,huangetal06}.
	On the other hand, approximate $N$-player Nash equilibria can be shown to converge to solutions of the corresponding \mbox{MFG}, as $N\to \infty$.
	Cardaliaguet, Delarue, Lasry and Lions in \cite{CDLL15} gave an important contribution in this direction when the strategies are of closed loop type, exploiting the well-posedness of the so-called master equation, which implies uniqueness of \mbox{MFG} solutions. Later, Lacker in \cite{lacker18} was able to establish a general convergence result for  non-degenerate diffusions, which he subsequently extended to the common noise case in the joint article \cite{lackerleflem} with Le Flem.\medskip

	\emph{Correlated equilibria} were first introduced for many-player games by Robert Aumann, see \cite{aumann74,aumann87}.
	His idea can be summarised in the following way: A \emph{correlation device} or \emph{mediator} (he) picks a strategy profile according to some probability distribution which is common knowledge among the players.
	Then, according to the selected profile, he privately suggests a strategy to each player, meaning that each player only knows the recommendation provided to him by the mediator.
	A correlated equilibrium (\mbox{CE}, for short) is a probability distribution on the space of strategy profiles such that no player is willing to unilaterally deviate from the mediator's suggestion.
	We notice that, when the distribution used by the mediator to generate his recommendations has a product form, then CE reduces to the usual notion of Nash equilibrium in mixed strategies.
	Traffic lights in routing games provide an intuitive example of a mediator in everyday life, e.g. \cite[Section 13.1.4]{roughgarden16}. Other interpretations for such equilibria are available in the literature, we refer the interested reader to, e.g., \cite{barany92}.
	
 The notion of CE was originally introduced for static games with complete information and it rapidly led to a massive research activity in game theory as well as in economic theory along many directions.
	The survey \cite{forges12} provides a thorough analysis on several aspects of the more general notion of communication equilibrium within a wide range of games, such as stochastic games and games with incomplete information.
	In particular, for stochastic games we also refer to \cite{solan00,solan01,solanvieille02}. Many pleasant features of CE justify the scientific community interest towards it, for instance the fact that it may lead to higher payoffs than Nash equilibria, its lower computational complexity (see, e.g.,\cite{gilboazemel89}), and also that CE are reachable by a wide range of learning procedures (see \cite{hart05}).\smallskip
	
	CE in mean field games where first studied in \cite{CF2022}, where the authors established approximation and convergence results for a class of symmetric finite horizon games in \emph{restricted strategies}.
	After \cite{CF2022} two more papers on correlated equilibria in mean field games appeared, by Paul M\"uller and co-authors \cite{google21,google22}, whose setting is very close to ours. Indeed, they, too, consider discrete time games with finite state and action spaces. The mean field interaction is modeled in the $N$-player games via the empirical measure of players' states. Players' strategies depend only on the player's individual states in a Markovian fashion. We stress that their definition of correlated equilibrium is different from the one we give in \cite{CF2022}. In particular, it does not require any explicit consistency condition for the flow of measures, which is obtained as a consequence of their definition. Nonetheless the most recent paper \cite{google22} has an interesting discussion on how to pass from our definition in \cite{CF2022} to theirs and vice-versa. Lastly, big parts of those papers are devoted to more computational issues focusing on learning algorithms approximating the equilibria.\smallskip
	
	Here, we consider correlated equilibria for a simple class of symmetric finite horizon $N$-player games and their natural \mbox{MFG} counterpart as $N \to \infty$. In the $N$-player setting, the state variables evolve in discrete time, both state space and the set of control actions are finite. The mediator recommends restricted strategies to the players, that is, feedback strategies that depend only on time and the corresponding individual state variable. This is the same framework as in \cite{CF2022}. As opposed to that work, and also to \cite{google21,google22}, the deviating player is allowed to use (randomized) progressive strategies, that is, strategies that depend on the evolution of the entire system state  up to current time; see Remark~\ref{RemDefMFGStrategies} below. We stress that the possibility for the players to deviate by playing progressive strategies make the analysis and the proofs much more delicate than in \cite{CF2022}. Our main results can be summarized as follows:
\begin{itemize}
	
	\item[-] We extend the notion of correlated solution for a mean field game to allow for progressive deviations. Two formulations are presented, one based on closed-loop controls, the other on stochastic open-loop controls.
	
	\item[-] Starting from suitable correlated \mbox{MFG} solutions, we construct approximate $N$-player correlated equilibria that are robust against progressive deviations.
	
	\item[-] We provide an explicit example for a mean field game possessing correlated solutions against progressive deviations that have non-deterministic flows of measures and satisfy all conditions of the approximation result.
	
\end{itemize}

The rest of the paper is structured as follows. In Section~\ref{SectNotation}, we introduce the notation and state some preliminary definitions. In Section~\ref{SectNPlayer}, we describe the underlying $N$-player games and give the definition of (approximate) correlated equilibrium against progressive deviations. Section~\ref{SectMFG-closed} is dedicated to the corresponding mean field game. Correlated \mbox{MFG} solutions are first defined in feedback strategies with deviations that may directly depend on the possibly random flow of measures. In Section~\ref{SectMFG-open} we give an alternative definition of correlated \mbox{MFG} solution in stochastic open-loop strategies and establish an equivalence between the two formulations. Our main result is given in Section~\ref{SectApproximate}, where we show that suitable correlated \mbox{MFG} solutions yield approximate correlated solutions for the $N$-player game. An example of a correlated \mbox{MFG} with explicit solutions satisfying the assumptions of our approximation result is provided in Section~\ref{SectMFGExample}.  In Appendix~\ref{SectAppendix}, we collect some auxiliary results.

%%%%%%%%%%%%%%%%%%%%%%%%%%%%%%%%%%%%%%%%%%%%%%%%%%%%%%%%%%%
%%%%%%%%%%%%%%%%%%%%%%%%%%%%%%%%%%%%%%%%%%%%%%%%%%%%%%%%%%%
%%%%%%%%%%%%%%%%%%%%%%%%%%%%%%%%%%%%%%%%%%%%%%%%%%%%%%%%%%%

\section{Preliminaries and notation} \label{SectNotation}

We denote with $[\![m,M]\!]$ the set of natural numbers greater or equal to $m$ and lower or equal to $M$, namely we set $[\![m,M]\!]:= \{m, m+1, \dots, M-1,M\}$.
A given $(T+1)$-dimensional vector, $(x_0, \dots, x_T)$, will be denoted with $(x_t)_{t=0}^T$ or just by $x$ when its indices are clear from the context. Then, the $(t+1)$-dimensional vector of its first $t+1$ components is denoted with $x^{(t)}:=(x_0,x_1, \dots, x_t)$.
Similarly, for a $T$-dimensional vector, $(x_1, \dots, x_T)$, we introduce the notation $(x_t)_{t=1}^T$ (just $x$ when the context is clear), and the $t$-dimensional vector of its first $t$ components is denoted with $x^{(t)}:=(x_1, \dots, x_t)$.
Finally, let us fix a notation that is useful in the following. Let $(\Omega, \mathcal{F},\mathbb{P})$ be a complete probability space supporting  a $(\mathcal{X},\mathcal{B}(\mathcal{X}))$-valued random variable, $X$.

We consider the (discrete) time frame $[\![0,T]\!]$, with finite final time $T \in \mathbb{N}.$
The individual states  and the control actions lie in non-empty finite sets $\mathcal{X}$ and $\Gamma$, respectively. We  mostly deal with finite sets and the sets of probability measures on them. Throughout the whole paper these sets are equipped with the discrete metric and the metric $\text{dist}(\cdot,\cdot)$, respectively, making them \emph{Polish spaces}. 
The metric $\text{dist}(\cdot,\cdot)$ on the set $\mathcal{P}(E)$ of probability measures over a finite set $E$ is defined as follows. For $\mu, \tilde{\mu} \in \mathcal{P}(E)$, set
\[
	\text{dist}(\mu,\tilde{\mu}):= \frac{1}{2} \sum_{e \in E}|\mu(e)-\tilde{\mu}(e)|.
\]
Notice that this metric is compatible with the weak convergence topology and, for measures over finite sets, weak convergence is equivalent to the convergence in total variation.
The set $\mathcal{Z}=[0,1]$ is the space of noise states.
All the variables representing idiosyncratic noise are distributed according to $\nu,$  uniform distribution on  $\mathcal{Z}=[0,1]$.

The one-step individual state dynamics is given by the following system function:
\[
	\Psi \colon [\![0,T-1]\!] \times \mathcal{X} \times \mathcal{P(X)}\times \Gamma \times \mathcal{Z} \longrightarrow \mathcal{X}.
\]
The running costs are specified through a function:
\[
	f \colon [\![0,T-1]\!] \times \mathcal{X} \times \mathcal{P(X) } \times \Gamma \longrightarrow \mathbb{R}.
\]
The terminal costs are described by the following function:
\[
	F \colon  \mathcal{X} \times \mathcal{P(X) }  \longrightarrow \mathbb{R}.
\]
Consider the product space $ [\![0,T-1]\!]\times \mathcal{X}\times \mathcal{P}(\mathcal{X})^{T}$. We equip this space with the product topology with respect to the topologies defined on each space, that are, respectively, discrete topology for $ [\![0,T-1]\!]$ and $ \mathcal{X}$, since they are finite sets, and the topology of weak convergence for the space $\mathcal{P(X)}.$ Then, on the space $ [\![0,T-1]\!]\times \mathcal{X}\times \mathcal{P}(\mathcal{X})^{T}$, we consider the $\sigma$-algebra:
\[
	\begin{split}
		\mathcal{B}\big( [\![0,T-1]\!]\times \mathcal{X}\times \mathcal{P}(\mathcal{X})^{T} \big)
		& =\mathcal{B}(  [\![0,T-1]\!]) \otimes \mathcal{B}( \mathcal{X})\otimes \mathcal{B}(\mathcal{P}(\mathcal{X})^{T})\\&
		= 2^{[\![0,T-1]\!]} \otimes 2^{\mathcal X}\otimes \mathcal{B}( \mathcal{P(X)})^{T},
	\end{split}
\]
where $2^E$ denotes the power set of a finite set $E$. Notice that $\mathcal{B}( \mathcal{P(X)})$ is the Borel $\sigma$-algebra induced by the topology of weak convergence, that in our case, where the state space $\mathcal{X}$ is finite, coincides with the one induced by the metric $\text{dist}(\cdot,\cdot)$, on $\mathcal{P(X)}$.
On the finite set $\Gamma$ we consider the discrete topology and its Borel $\sigma$-algebra.

Let us define $\widehat{\mathcal{R}}$, the set of progressive feedback strategies:
\[
	\widehat{\mathcal{R}}:= \Big\{ \varphi: [\![0,T-1]\!]\times 	\mathcal{X}^T\times \mathcal{P}(\mathcal{X})^{T}\longrightarrow 	\Gamma,\quad \varphi\text{  progressively measurable}   \Big\}.
\]
As it is used several times in the following, we introduce another set of feedback strategies. It corresponds to the Markov strategies that depend only on the individual player's state, see \emph{restricted strategies} in \cite{CF2022}:
\[
	{\mathcal{R}}:= \Big\{ \varphi: [\![0,T-1]\!]\times \mathcal{X} \longrightarrow \Gamma \Big\}.
\]
This space is equipped, as all finite sets in this paper, with the discrete topology. Notice that we have the natural inclusion $\mathcal{R} \subset \UU$, and $\mathcal{R}$ is compact since it is finite.\\
Furthermore, for convenience of notation, for each $t \in [\![0,T-1]\!]$, we set
\[
	\begin{split}
		&\widehat{\mathcal{E}}_t
		:= \Big\{ \varphi:  \mathcal{X}^{t+1}\times \mathcal{P}(\mathcal{X})^{t+1} \longrightarrow \Gamma,\quad \varphi\text{  Borel-measurable}   \Big\},\\&
		\widehat{\mathcal{E}}^{(t)}
		:= \Big\{ \varphi:  [\![0,t]\!] \times \mathcal{X}^{t+1}\times \mathcal{P}(\mathcal{X})^{t+1} \longrightarrow \Gamma,\quad \varphi\text{ progressively measurable}   \Big\},
	\end{split}
\]
and the corresponding restricted quantities
\[
	\begin{split}
		&{\mathcal{E}}_t={\mathcal{E}}
		:= \Big\{ \varphi:  \mathcal{X}\longrightarrow \Gamma,\quad \varphi\text{ Borel-measurable}   \Big\},\\&
		{\mathcal{E}}^{(t)}={\mathcal{E}}^t
		:= \Big\{ \varphi: [\![0,t]\!] \times \mathcal{X} \longrightarrow \Gamma,\quad \varphi\text{  Borel-measurable}   \Big\}.	
	\end{split}
\]

When considering the $N$-player game, the set of progressively measurable feedback strategies corresponds to the following subset of $\widehat{\mathcal{R}}$
\[
	\widehat{\mathcal{R}}_N:= \Big\{ \varphi: [\![0,T-1]\!]\times 	\mathcal{X}^T\times {(\mathcal{M}_N^{\mathcal{X}})}^{T}\longrightarrow 	\Gamma,\quad \varphi\text{ progressively measurable}   \Big\},
\]
where  $\mathcal{M}_N^{\mathcal{X}}:= \{m \in \mathcal{P(X)}: \text{ for any } x \in \mathcal{X}, m(x)=\frac{k}{N}, k \in [\![0,N]\!]\}$ is the set of empirical measures of $N$-samples. Notice that the set $\widehat{\mathcal{R}}_N$ is finite. Indeed this is a consequence of the finiteness of $\mathcal{M}_N^{\mathcal{X}}$, whose cardinality is $\frac{(N+|\mathcal{X}|-1)!}{N!(|\mathcal{X}|-1)!}$. Thus, we endow this set with the discrete topology.
Analogously to what is done above, we set
\[
	\begin{split}
		&\widehat{\mathcal{E}}_{t,N}
		:= \Big\{ \varphi:  \mathcal{X}^{t+1}\times {(\mathcal{M}_N^{\mathcal{X}})}^{t+1} \longrightarrow \Gamma,\quad \varphi\text{  Borel-measurable}   \Big\},\\&
		\widehat{\mathcal{E}}^{(t)}_N
		:= \Big\{ \varphi: [\![0,t]\!] \times \mathcal{X}^{t+1}\times {(\mathcal{M}_N^{\mathcal{X}})}^{t+1} \longrightarrow \Gamma,\quad \varphi\text{ progressively measurable}   \Big\}.	
	\end{split}
\]
Finally, set
\begin{align}
	& \mathcal{D}:=\{ w: \mathcal{R} \to \mathcal{R} \},
	\qquad 	
	\widehat{\mathcal{D}}:=\{ w: \widehat{\mathcal{R}}%{\color{red}(\mathcal{R}??)}
	 \to \widehat{\mathcal{R}}\},
\end{align}
which are respectively the sets of restricted and not strategies modifications. 
Notice again that the former set is clearly finite and the latter, when restricted to the $N$-player game, is finite and denoted by
\begin{align}
	& \widehat{\mathcal{D}}_N:=\{ w: \widehat{\mathcal{R}}_N \to \widehat{\mathcal{R}}_N\}.
\end{align}

In the following we make extensive use of the concepts of  regular conditional distribution and probability kernel, for which we refer to \cite{Kallenberg}.
For all $N \in \mathbb{N}$, we define the set of flows of kernels
\[
	\begin{split}
		\mathcal K_N:=\{ & \beta=(\beta_t)_{t=0}^{T-1}: \beta_t \text{ probability kernel from } (\widehat{\mathcal{R}}_N,\mathcal B(\widehat{\mathcal{R}}_N)) \text{ to } (\widehat{\mathcal{E}}_{t,N},\mathcal B(\widehat{\mathcal{E}}_{t,N})),\\&
\text{ for all } t \in [\![0,T]\!]\}.
	\end{split}
\]
We can provide a natural interpretation for a flow of kernels $\beta \in \mathcal K_N$ in our context. 
It represents some procedure through which players in the $N$-player game select their strategies. 
Indeed, a player receives a $\widehat{\mathcal{R}}_N$-valued suggestion from the mediator at the beginning of the game and then, at each time step $t \in [\![0,T-1]\!]$, determines his $\widehat{\mathcal{E}}_{t,N}$-valued strategy as a function of the suggestion received and an additional independent randomization factor (e.g. tossing a coin).

Finally, in the following, all $\sigma$-algebras and filtrations are assumed to be completed w.r.t. $\mathbb{P}$-null sets.

%%%%%%%%%%%%%%%%%%%%%%%%%%%%%%%%%%%%%%%%%%%%%%%%%%%%%%%%%%%%%%%%%%%%%%%%%%%%%%%%%%%%%%%%%%%%%%%%%%%%%%%%%%%%%%%%%%%%%%%%%%%%%%%%%%%%%%%%%%%%%%%%%%%%%%%%%%%%%%%%%%%%%%%%%%%%%%%%%

\section{The N-player game} \label{SectNPlayer}

Consider a fixed  number of players, $N \in \mathbb{N},$   and let $\mathfrak{m}^N \in \mathcal{P(X}^N)$ represent the initial distribution of the $N$-player system. 
For any probability distribution $\gamma \in \mathcal P(\widehat{\mathcal{R}}_N)$, we define the set $\mathcal N^N_\gamma$ as
	\begin{align}
		\mathcal N_\gamma^N:=\Big\{ &\widetilde \gamma \in  \mathcal P (\widehat{\mathcal{R}}_N\times \UU_N): 
			\widetilde \gamma(d\varphi,d\psi)=\beta_0(\varphi,d\psi_{0})\ldots\beta_{T-1}(\varphi,d\psi_{T-1})\gamma(d	\varphi),\\ \nonumber
			& \text{ for some }\beta=(\beta_t)_{t=0}^{T-1}\in \mathcal K_N \Big\}.
	\end{align}
	where, for all $t \in [\![0,T-1]\!]$, $\psi_t$ is the short form for $\psi(t,\cdot,\cdot)$.
	The elements of $\mathcal N_\gamma^N$ represent the joint distribution of the mediator's suggestion and players' strategy choices. In particular, if $\widetilde \gamma\in \mathcal N_\gamma^N$, then the first marginal of $\widetilde \gamma$ equals $\gamma$.

Finally, for a probability distribution $\gamma^N \in \mathcal{P}(\widehat{\mathcal{R}}^N_N)$, we denote its $i$-th marginal by
	\[
		\gamma^N_i(\cdot) := \gamma^N(\widehat{\mathcal{R}}_N \times \dots \times \cdot \times \dots \times \widehat{\mathcal{R}}_N),
	\]
	where $\cdot$ on the right-hand side above occupies the $i$-th coordinate.
	
\begin{definition}\label{Def_Real_N}
	We call \emph{correlated suggestion} any probability distribution $\gamma^N \in \mathcal{P}(\widehat{\mathcal{R}}^N_N)$.
	Then, consider a probability distribution $\widetilde \gamma \in \mathcal N^N_{\gamma^N_i}$
and call it a \emph{strategy modification} for the $i$-th player.
	Let $(\Omega_N,{ \mathcal{F}}_N,{\mathbb{P}}_N)$ be  a complete probability space carrying  
 $\mathcal{X}$-valued random variables $(X^{1,N}_t,\dots, X^{N,N}_t)_{t=0}^T$,
 $\widehat{\mathcal{R}}_N $-valued random variables $\Phi_1, \dots, \Phi_N, \widetilde \Phi_i$, and  $\mathcal{Z}$-valued random variables $(\xi_t^{1,N},\dots,\xi_t^{N,N})_{t=1}^{T}$ and $(\vartheta_t)_{t=0}^{T-1}$ such that the following properties hold:
	\begin{itemize}
		\item[\textbf{i)}] ${\mathbb{P}}_N\circ (X^{1,N}_0,\dots,X^{N,N}_0)^{-1}=\mathfrak{m}^N$;\\
 $ {\mathbb{P}}_N\circ (\Phi_1,\dots,\Phi_N)^{-1}=\gamma^N$; 

		\item[\textbf{ii)}] $(\xi_t^{1,N},\dots,\xi_t^{N,N})_{t=1}^{T}$ are i.i.d.\ all distributed according to $\nu$;

		\item[\textbf{iii)}] $(\vartheta_t)_{t=0}^{T-1}$ are i.i.d.\ all distributed according to $\nu$;

		\item[\textbf{iv)}] $(\xi_t^{1,N},\dots,\xi_t^{N,N})_{t=1}^{T}$, $(X^{j,N}_0)_{j=1}^{N}$,  $(\vartheta_t)_{t=0}^{T-1}$, and $(\Phi_j)_{j=1}^N$ are independent;
		 
		\item[\textbf{v)}] $\mathbb P \circ (\Phi_i, \widetilde \Phi_i)^{-1}=\widetilde \gamma$ and,
for any $t \in [\![0,T-1]\!]$, $\widetilde\Phi_i(t,\cdot,\cdot)$ is  $\sigma(\Phi_i, \vartheta_t)$-measurable;

		\item[\textbf{vi)}]
		for any \hbox{$t \in [\![0,T-1]\!],$}
		
			\begin{equation}
					\begin{split}
						&X^{i,N}_{t+1}= \Psi\left(t,X^{i,N}_t,\mu^{i,N}_t,\widetilde{\Phi}_i(t,X^{i,N},\mu^{i,N}),\xi^{i,N}_{t+1}\right), \\&
						X^{j,N}_{t+1}= \Psi\left(t,X^{j,N}_t,\mu^{j,N}_t, \Phi_j(t,X^{j,N},\mu^{j,N}),\xi^{j,N}_{t+1}\right),\quad  j \neq i, \qquad  \mathbb{P}_N\text{-a.s.},
					\end{split}
			\end{equation}
			where $\mu_t ^{l,N}$ denotes the empirical measure of all $N$ players' states but the $l$-th,\\
			 i.e. \hbox{$\mu^{l,N}_t:=\frac{1}{N-1}\sum_{j=1, j\neq l}^N \delta_{X^{j,N}_t}$}, and $\mu^{l,N}:=(\mu^{l,N}_t)_{t=0}^{T} \in \PX$.
	\end{itemize}
	Any tuple $((\Omega_N, \mathcal{F}_N,\mathbb{P}_N),(\Phi_j)_{j=1}^N, (\vartheta_t)_{t=0}^{T-1},(\xi^{1,N}_t,\dots,\xi^{N,N}_t)_{t=1}^{T},\widetilde \Phi_i,(X^{1,N}_t,\dots,X^{N,N}_t)_{t=0}^{T})$ satisfying the conditions above is called a \emph{ realization of the triple $(\mathfrak{m}^N, \gamma^N,\widetilde\gamma)$ for player $i \in [\![1,N]\!]$}. 
\end{definition}
The correlated suggestion $\gamma^N$ represents the known distribution, over the product set of the players' strategies, according to which the mediator gives his recommendations to the players, while $\widetilde \gamma$ represents the strategy modification for the deviating $i$-th player, encoded as the joint distribution of the suggestion received and the strategy he is actually taking into action. The fact that, for any $t \in [\![0,T-1]\!]$, $\widetilde\Phi_i(t,\cdot,\cdot)$ is $\sigma(\Phi_i, \vartheta_t)$-measurable yields that, at any time instant $t \in [\![0,T-1]\!]$, the $i^{th}$ player can exploit an (independent) randomization device to choose the strategy that he actually implements.

\begin{remark}
	Notice that, for any $w \in \widehat{\mathcal{D}}_N$, given a sequence of suggestions $(\Phi_j)_{j=1}^N$, $\widetilde \Phi_i= w(\Phi_i)$ satisfies assumption \textbf{\emph{v)}} in Definition \ref{Def_Real_N}, with $\sigma(\widetilde\Phi_i(t,\cdot, \cdot)) \subset \sigma (\Phi_i)$ and $\mathbb{P}\circ (\Phi_i, \widetilde \Phi_i) (d \varphi, d \psi) = \delta_{w(\varphi)}(d\psi)\gamma^N_i(d\varphi) $.
\end{remark}

\begin{remark}\label{IInd_cond}
	We make the following useful remarks concerning (conditional) independence properties of a realization.
	\begin{itemize}

		\item[i)] Notice that the following inclusion of $\sigma$-algebras holds $ \sigma (\widetilde\Phi_i ) \subseteq \sigma ((\vartheta_t)_{t=0}^{T-1}, \Phi_i)$, by definition.
			Indeed, we have
			\[
				\begin{split}
					\sigma(\widetilde \Phi_i)&
					=\sigma ((\widetilde \Phi_i(t,\cdot,\cdot))_{t=0}^{T-1})
					=\bigvee_{t \in [\![0,T-1]\!]}\sigma (\widetilde \Phi_i(t,\cdot,\cdot))
					\subseteq \bigvee_{t \in [\![0,T-1]\!]}\sigma ( \Phi_i,\vartheta_t)%\\&
					%=\bigvee_{t \in [\![0,T-1]\!]} \left(\mathcal{W}\lor \sigma ( \Phi_i,\vartheta_t)\right)
					%=\mathcal{W}\lor \left(\bigvee_{t \in [\![0,T-1]\!]} \sigma ( \Phi_i,\vartheta_t)\right)\\&
					%=\mathcal{W} \lor \sigma ( \Phi_i,( \vartheta_t)_{ t=0}^{T-1})
					=\sigma( \Phi_i,( \vartheta_t)_{ t=0}^{T-1}),
				\end{split}
			\]
		The identities above hold since, for all $t \in [\![0,T-1]\!]$, $\widehat{\mathcal{E}}_{t,N}$ are equipped with discrete topology (making them Polish spaces)  so the Borel $\sigma$-algebra of the product space $\widehat{\mathcal{R}}_N$
		coincides with the product of the Borel $\sigma$-algebras of $\widehat{\mathcal{E}}_{t,N}$. Thus, the $\sigma$-algebra generated by a $\widehat{\mathcal{R}}_N$-valued r.v. coincides with the one generated by its components in $\widehat{\mathcal{E}}_{t,N}$.

		\item[ii)] Notice that the assumptions in Definition \ref{Def_Real_N}, in particular \textbf{iv)} and \textbf{v)}, imply that for a realization of $(\mathfrak m^N, \gamma^N,\widetilde \gamma)$ as above
		\[
			(\xi^{1,N}_t, \dots, \xi^{N,N}_t)_{t=1}^T, (X^{j,N}_0)_{ j =1}^N \text{ and }(\widetilde\Phi_i, (\Phi_j)_{j=1}^N)) \text{ are independent.}
		\]
		In fact, by \textbf{v)}, $\sigma (\widetilde\Phi_i, (\Phi_j)_{j=1}^N)) \subseteq \sigma (( \vartheta_t)_{ t=0}^{T-1}, (\Phi_j)_{j=1}^N)$ and the $\sigma$-algebras  $\sigma (( \vartheta_t)_{ t=0}^{T-1},$ $(\Phi_j)_{j=1}^N)$, $\sigma((\xi^{1,N}_t, \dots, \xi^{N,N}_t)_{t=1}^T)$ and $\sigma((X^{j,N}_0)_{ j =1}^N )$ are independent by \textbf{iv)}.

		\item[iii)]For a realization of $(\mathfrak m^N, \gamma^N,\widetilde \gamma)$, as above,
		\[
			(\Phi_j)_{j=1}^N \text{ and }\widetilde{\Phi}_i \text{ are conditionally independent given }\Phi_i.
		\]
		We  notice that $(\Phi_j)_{j}:=(\Phi_j)_{j=1}^N$ and $\vartheta := (\vartheta_t)_{t=0}^{T-1}$ are conditionally independent given $\Phi_i.$ 
		Indeed, given $A \in \mathcal B( \widehat{\mathcal{R}}_N^N), B \in \mathcal B (\mathcal Z^{T+1}),$ we have, $\mathbb{P}_N$-a.s.,
		\[
			\begin{split}
				\mathbb E_N [\mathbf I_A((\Phi_j)_{j}) \mathbf I_B(\vartheta) | \Phi_i]&
				=\mathbb E_N [\mathbb E_N [\mathbf I_A((\Phi_j)_{j}) \mathbf I_B(\vartheta) | (\Phi_1, \ldots, \Phi_N)] | \Phi_i]\\&
				=\mathbb E_N [\mathbf I_A((\Phi_j)_{j}) \mathbb E _N[ \mathbf I_B(\vartheta) | (\Phi_1, \ldots, \Phi_N)] | \Phi_i]\\&
				=\mathbb E_N [\mathbf I_A((\Phi_j)_{j}) \mathbb E_N [ \mathbf I_B(\vartheta) ] | \Phi_i]
				= \mathbb E_N [ \mathbf I_B(\vartheta) ]\mathbb E_N [\mathbf I_A((\Phi_j)_{j}) | \Phi_i]\\&
				= \mathbb E_N [ \mathbf I_B(\vartheta) | \Phi_i]\mathbb E_N [\mathbf I_A((\Phi_j)_{j}) | \Phi_i].
			\end{split}
		\]

		Then,  for arbitrary sets $A \in \mathcal{B}(\widehat{\mathcal{R}}_N^N),B \in \mathcal B (\UU_N)$, exploiting \textbf{iv)} and  \textbf{v)} and the conditional independence showed above, we see, $\mathbb{P}_N$-a.s.,
		\[
			\begin{split}
				\mathbb E_N [\mathbf I_A((\Phi_j)_{j}) \mathbf I_B(\widetilde \Phi_i) | \Phi_i]&
			=\mathbb E_N [\mathbb E_N [\mathbf I_A((\Phi_j)_{j}) \mathbf I_B(\widetilde \Phi_i) |\sigma( \vartheta,\Phi_i)] | \Phi_i]\\&
			=\mathbb E_N [\mathbf I_B(\widetilde \Phi_i)\mathbb E_N [\mathbf I_A((\Phi_j)_{j})  |\sigma( \vartheta,\Phi_i)] | \Phi_i]\\&
			=\mathbb E_N [\mathbf I_B(\widetilde \Phi_i)\mathbb E_N [\mathbf I_A((\Phi_j)_{j})  | \Phi_i] | \Phi_i]\\&
			=\mathbb E_N [\mathbf I_B(\widetilde \Phi_i)| \Phi_i]\mathbb E_N [\mathbf I_A((\Phi_j)_{j})  | \Phi_i] \\&
			=\mathbb E_N [\mathbf I_B(\widetilde \Phi_i)| \Phi_i]\mathbb E_N [\mathbf I_A((\Phi_j)_{j})  | \Phi_i].
			\end{split}
		\]
	\end{itemize}
\end{remark}

\begin{remark}\label{gam_tild}
	There is a strategy modification of particular interest for every correlated suggestion and every player.
	It reflects the case in which the player $i$, as all the other players, follows the suggestion he is given by the mediator.
	Exploiting the definition of realization of a certain triple, this corresponds to $\Phi_i=\widetilde{\Phi}_i,  \mathbb{P}_N$-a.s..
	In particular, let $\mathbb P_N \circ \Phi_i^{-1}=\gamma$, we have
	\[
		\begin{split}
		\mathbb{P}_N\circ(\Phi_i, \widetilde \Phi_i)^{-1}(d\varphi,d\psi)
		=\mathbb{P}_N\circ(\Phi_i, \Phi_i)^{-1}(d\varphi,d\psi)
		=\delta_{\phi}(d\psi)\gamma(d\varphi).
		\end{split}
	\]
	We denote this special strategy modification  with $\iota_\gamma \in \mathcal N_\gamma^N$ .\\
	Notice that property \textbf{v)} is obviously satisfied  in this case and, viceversa, for a realization of the triple $(m_0, \gamma^N, \iota_{\gamma^N_i})$, we have $\Phi_i=\widetilde \Phi_i,$ $\mathbb P_N$-a.s. and $\widetilde \Phi_i(t,\cdot,\cdot)$ is $\sigma(\Phi_i, \vartheta_t)$-measurable, for any $t \in [\![0,T-1]\!]$.
\end{remark}

The formalization of the concept of realization enables us to associate to the triple $(\mathfrak{m}^N, \gamma^N, \widetilde \gamma) \in \mathcal{P(X}^N)\times \mathcal{P}(\widehat{\mathcal{R}}_N^N)\times \mathcal{P}(\widehat{\mathcal{R}}_N\times \widehat{\mathcal{R}}_N)$  a cost functional  for player $i$, through the following expression:
\begin{equation}\label{def_J_i}
	J_i^N(\mathfrak{m}^N, \gamma^N, \widetilde\gamma ):= 	\mathbb{E}\left[ \sum_{t=0}^{T-1} f \left(t,X^{i,N}_t,\mu^{i,N}_t,	\widetilde{ \Phi}_i\left(t,X^{i,N},\mu^{i,N}\right)\right) +F\left(X^{i,N}_T,\mu^{i,N}_T\right) \right].
\end{equation}

By construction, the right-hand side of (\ref{def_J_i}) does not depend on the particular realization but only on $(\mathfrak{m}^N, \gamma^N,\widetilde \gamma )$. Indeed, $\widetilde \gamma \in \mathcal N^N_{\gamma^N_i}$ yields 
\[
	\widetilde \gamma(d\varphi,d\psi)=\beta^N_0(\varphi^{(0)},d\psi_0)\ldots\beta^N_T(\varphi^{(T)},d\psi_T)(\gamma^N_i)(d\varphi),
\]
for some $\beta^N=(\beta^N_t)_{t \in [\![0,T]\!]} \in \mathcal K_N$. Thus, the cost functional above is well-posed and we write
\[
	\begin{split}
		&J_i^N(\mathfrak{m}^N, \gamma^N, \widetilde \gamma )
		= \int_{\mathcal X^N} \int_{\mathcal Z^{NT}} 		\int_{\widehat{\mathcal{R}}_N^N}
		\int_{\widehat{\mathcal{E}}_{0,N}}\ldots \int_{\widehat{\mathcal{E}}_{T,N}}
		G^N(x_1,\ldots,x_N,\varphi_0,\ldots,\varphi_{T-1}, u_1,	\ldots, u_N,z_1,\ldots,z_{NT})\\
		& \qquad \beta^N_{T}(u_i,d\varphi_{T})\cdots\beta^N_0(u_i,d\varphi_{0}) \gamma^N(du_1,\ldots,du_N)\nu^{\otimes NT}		(dz_1,\ldots,dz_{NT})m_0^{\otimes N}(dx_1,\ldots,dx_{N}),
	\end{split}
\] 
for some measurable function $G^N: \mathcal X^N \times \widehat{\mathcal{E}}_{0, N} \times \ldots \times \widehat{\mathcal{E}}_{T-1, N} \times   \widehat{\mathcal{R}}_N^N\times \mathcal Z^{NT} \to \mathbb R.$
\\

Since, for each $i  \in [\![1,N]\!]$, the functional $J_i^N(\cdot)$ represents the costs that  player $i$ faces, his aim is to minimize it. As natural when dealing with several players, we deal with an equilibrium concept for optimality.
\begin{definition} \label{DefPrelimitCE}
	Let $\varepsilon \geq 0$. We call a distribution $\gamma^N \in  \mathcal{P}(\widehat{\mathcal{R}}_N^N)$ an \emph{$\varepsilon$-correlated equilibrium } with initial distribution $\mathfrak{m}^N \in \mathcal{P(X}^N)$ if we have
	\[
		J_i^N(\mathfrak{m}^N, \gamma^N, \iota_{\gamma^N_i}) \leq J_i^N(\mathfrak{m}^N,\gamma^N, \widetilde \gamma)+\varepsilon,
	\]
	for every player $i \in [\![1,N]\!]$ and every strategy modification  $\widetilde \gamma \in \mathcal N_{\gamma^N_i}$.\\
	In particular, we call $\gamma^N $ a \emph{correlated equilibrium}, denoted by CE, if $\varepsilon=0.$
\end{definition}

Definition~\ref{DefPrelimitCE} is in line with the notion of correlated equilibrium present in the literature. We stress that, here, the deviating player has access to the entire history of the system and, in addition, is allowed to use a randomization device.

%%%%%%%%%%%%%%%%%%%%%%%%%%%%%%%%%%%%%%%%%%%%%%%%%%%%%%%%%%%%%%%%%%%%%%%%%%%%%%%%%%%%%%%%%%%%%%%%%%%%%%%%%%%%%%%%%%%%%%%%%%%%%%%%%%%%%%%%%%%%%%%%%%%%%%%%%%%%%%%%%%%%%%%%%%%%%%%%%
\section{The mean field game}\label{SectMFG-closed}

Let $\mathfrak{m}_0 \in \mathcal{P(X)}$ be the initial distribution of our mean field system. 
In this model there is only one representative player in the mean field game because of the symmetry in the $N$-player game.

\begin{definition}\label{mf_real}
	Let $\rho \in \mathcal{P}\big(\mathcal{R}\times  \mathcal{P(X)}^{T+1}\big)$ and call it a \emph{correlated suggestion}.
	Call \emph{strategy modification} a function $w \in \widehat{\mathcal{D}}$.
	Then, let $( \Omega,{ \mathcal{F}},{\mathbb{P}})$ be  a  probability space supporting  $\mathcal{X}$-valued process $(X_t)_{t=0}^T$, an  $\mathcal{R}$-valued random variable $\Phi$, a $\PX$-valued random variable $\mu$ and  $\mathcal{Z}$-valued random variables $(\xi_t)_{t=1}^T$,  such that the following properties hold:
	\begin{itemize}
		\item[\textbf{i)}] ${\mathbb{P}}\circ X_0^{-1}=\mathfrak{m}_0$;

		\item[\textbf{ii)}] ${\mathbb{P}}\circ (\Phi,(\mu_t)_{t=0}^{T})^{-1}=\rho$;

		\item[\textbf{iii)}] $(\xi_t)_{t=1}^T$ are i.i.d.\ all distributed according to $\nu;$
	
		\item[\textbf{iv)}] $(\xi_t)_{t=1}^T$, $X_0$ and $(\Phi,  (\mu_t)_{t=0}^T)$ are independent;

		\item[\textbf{v)}] the evolution of $(X_t)_{t=0}^T$ follows this dynamics: for any $t \in [\![0,T-1]\!]$,	
		\begin{equation} \label{EqMFGDynamics}
			\begin{split}
				&X_{t+1}=\Psi\left(t,X_t,\mu_t, w\circ\Phi(t,X,\mu),\xi_{t+1}\right), \qquad \mathbb{P}\text{-a.s.}.
			\end{split}
		\end{equation}
	\end{itemize}
	We call any tuple $\big((\Omega, \mathcal{F},\mathbb{P}),\Phi,(\mu_t)_{t=0}^T,X_0,(\xi_t)_{t=1}^T, w, (X_t)_{t=0}^T \big)$ satisfying the conditions above a \emph{realization of the triple $(\mathfrak{m}_0, \rho, \widetilde \rho)$}.
\end{definition}

The strategy modification $w$ represents how the representative player decides to deviate from the suggestion he was given.
Notice that, contrary to the $N$-player game where the $i^{th}$ player can exploit a randomization device when selecting the strategy to put in action, the choice here is a deterministic functional of the suggestion, $\Phi$, provided by the mediator.

\begin{remark}
	As for the $N$-player game, we can characterize the form of a realization for  the case in which the representative player follows the suggestion provided to him.\\
	This is the case when the function $w$ is just the identity. Indeed, we have $w\circ\phi(t,x^{(t)},m^{(t)})=\phi(t,x_t)$, for each $t \in [\![0,T-1]\!]$ and $\phi \in \mathcal{R}$. We call this special modification $\iota$.
\end{remark}

The player in the mean field game  faces costs associated to the triple $(\mathfrak{m}_0, \rho, w ) \in \mathcal{P(X)}\times  \mathcal{P}\big(\mathcal{R}  \times  \mathcal{P(X)}^{T+1}\big) \times \widehat{\mathcal{D}}$ that are given by
\begin{equation}
	J(\mathfrak{m}_0, \rho, w)
	:= \mathbb{E}\left[ \sum_{t=0}^{T-1}f\left(t,X_{t},\mu_t,w \circ\Phi(t,X^{(t)},\mu^{(t)})\right) +F\left(X_{T},\mu_T\right) \right].
\end{equation}

As noticed for the $N$-player game, we highlight that the cost functional above is well defined  since the right-hand side does not depend on the realization considered but only on $(\mathfrak{m}_0, \rho, w)$, and we may write
\begin{align}\label{Eq_rewrite_J_functional}
	&J(\mathfrak{m}_0, \rho, w)
		= \int_{\mathcal X} \int_{\mathcal Z^{T}} 	\int_{\mathcal{R}\times \mathcal P(\mathcal X)^{T+1}} G_w(x,\phi,z,m) \rho(d\phi,dm)\nu^{\otimes T}(dz )m_0(dx),
\end{align}
for some function $G_w: \mathcal X   \times {\mathcal{R}}\times \mathcal Z^{T}\times \PX \to \mathbb R.$

\begin{definition}\label{Def_CMFG_closed}
	We say that $\rho \in  \mathcal{P}\big(\mathcal{R} \times \mathcal{ P(X)}^{T+1}\big)$ is a \emph{correlated solution for the mean field game} with initial distribution $\mathfrak{m}_0 \in \mathcal{P(X)}$, if the following two conditions hold:
	\begin{itemize}
		\item[\textbf{(Opt)}] For each  strategy modification  $w \in  \widehat{\mathcal{D}} $,
\[ J(\mathfrak{m}_0,{\rho}, \iota) \leq J(\mathfrak{m}_0, \rho, w).\]

		\item[\textbf{(Con)}] For any realization of  $(\mathfrak{m}_0,  {\rho}, \iota)$, namely $\big((\Omega,\mathcal{F},\mathbb{P}\big),\Phi,(\mu_t)_{t=0}^T, X_0, (\xi_t)_{t=1}^{T},\iota, (X_t)_{t=0}^T )$, setting $\mathcal{F}^\mu:=\sigma\big((\mu_t)_{t=0}^T\big)$, we have
		\[
\mu_t(\cdot)=\mathbb{P}(X_{t}\in \cdot \text{ } | \mathcal{F}^\mu ), \quad t \in [\![0,T]\!].
		\]
	\end{itemize}
\end{definition}

The first condition above is called \emph{optimality condition}, the second is called \emph{consistency condition}.

\begin{remark} \label{RemDefMFGStrategies}
A correlated solution according to Definition~\ref{Def_CMFG_closed} is an element of $\mathcal{P}\big(\mathcal{R} \times \mathcal{ P(X)}^{T+1}\big)$. The mediator thus suggests to play strategies that depend only on time and the representative player's current state (Markov open-loop or \emph{restricted strategies} as in \cite{CF2022}). By the optimality condition, following the mediator's recommendations in those restricted strategies has to be optimal against progressive deviations, that is, strategies that may depend on the entire history of state and flow of measures up to current time. More precisely, if the representative player decides to deviate, then she chooses a strategy modification $w$ (not equal to the identity on $\mathcal{R}$) that takes a (restricted) strategy recommended by the mediator and transforms it into a progressive feedback strategy, which is then applied to generate the state dynamics; see Eq.~\eqref{EqMFGDynamics}.
\end{remark}

\begin{remark} 
In the consistency condition of Definition~\ref{Def_CMFG_closed}, we take conditional distribution with respect to $\mathcal{F}^\mu$, the $\sigma$-algebra generated by the entire flow of measures $\mu$ (up to terminal time $T$). This implies the generally weaker condition
		\begin{equation} \label{EqWeakConsist}
	\mu_t(\cdot)=\mathbb{P}(X_{t}\in \cdot \text{ } | \mathcal{F}^\mu_t ), \quad t \in [\![0,T]\!],
		\end{equation}
where $\mathcal{F}^\mu_t:=\sigma\big((\mu_s)_{s=0}^t\big)$ is the $\sigma$-algebra generated by the flow of measures $\mu$ up to time $t$. The intuition behind conditioning on the entire flow of measures is the following. In choosing a correlated equilibrium, the mediator wants to induce a certain behavior of the population. That behavior is represented by the flow of measures $\mu$. In equilibrium, the representative player accepts the mediator's recommendations. But those recommendations are potentially correlated with the flow of measures up to terminal time. As a consequence, the player's state $X_{t}$ at any intermediate time $t$ can be correlated with the flow of measures $\mu$ also at future times. In order to reproduce the population behavior given by $\mu$, the representative player's state must therefore satisfy the consistency condition according to \emph{\textbf{(Con)}}, not just \eqref{EqWeakConsist}. For further discussion also see Remark~4.2 in \cite{CF2022}.
\end{remark}

%%%%%%%%%%%%%%%%%%%%%%%%%%%%%%%%%%%%%%%%%%%%%%%%%%%%%%%%%%%
%%%%%%%%%%%%%%%%%%%%%%%%%%%%%%%%%%%%%%%%%%%%%%%%%%%%%%%%%%%
%%%%%%%%%%%%%%%%%%%%%%%%%%%%%%%%%%%%%%%%%%%%%%%%%%%%%%%%%%%

\section{The mean field game in open-loop strategies}\label{SectMFG-open}

%%%%%%%%%%%%%%%%%%%%%%%%%%%%%%%%%%%%%%%%%%%%%%%%%%%%%%%%%%%
%%%%%%%%%%%%%%%%%%%%%%%%%%%%%%%%%%%%%%%%%%%%%%%%%%%%%%%%%%%
%%%%%%%%%%%%%%%%%%%%%%%%%%%%%%%%%%%%%%%%%%%%%%%%%%%%%%%%%%%

Now, we formalize an alternative structure for the mean field game, extending the class of admissible control policies.
We then prove in Section \ref{Subsect_val_of_MFG} that, under a mild assumption on the form of the correlated solution $\rho$, the value of the MFG remains the same.

\subsection{The definition of the MFG in open-loop strategies}

Let $\mathfrak{m}_0 \in \mathcal{P(X)}$ be the initial distribution of the mean field system. 

\begin{definition}\label{mf_real_relaxed}
	Let $\rho \in \mathcal{P}\big(\mathcal{R}\times  \mathcal{P(X)}^{T+1}\big)$. 
	A tuple $((\Omega, \mathcal{F} ,\mathbb{P}),$ $\{\mathcal{G}_t\}_{t=0}^{T-1}$, $\Phi, (\mu_t)_{t=0}^T, X_0, (\xi_t)_{t=1}^T, (u_t)_{t=0}^{T-1})$ is said to be an \emph{open-loop control policy (open-loop strategy)} if  $( \Omega, \mathcal{F}, {\mathbb{P}})$ is a complete probability space supporting  $\mathcal{X}$-valued random variables $X_{t}, t \in [\![0,T]\!],$ an  $ \mathcal{R}$-valued random variable $\Phi$, a $\PX$-valued random variable $\mu$, $\mathcal{Z}$-valued random variables $(\xi_t)_{t=1}^T$ and  $\Gamma$-valued random variables $u_t, t \in [\![0,T-1]\!]$, and $\{\mathcal{G}_t\}_{t=0}^{T-1}$ is a complete filtration such that
	\begin{itemize}
		\item[\textbf{i)}] ${\mathbb{P}}\circ (X_0)^{-1}=\mathfrak{m}_0$;

		\item[\textbf{ii)}] ${\mathbb{P}}\circ (\Phi,(\mu_t)_{t=0}^T)^{-1}=\rho$;

		\item[\textbf{iii)}] $(\xi_t)_{t=1}^T$ are i.i.d. all distributed according $\nu;$
	
		\item[\textbf{iv)}] $(\xi_t)_{t=1}^T$, $X_0$ and $(\Phi,  (\mu_t)_{t=0}^T)$ are independent;
		
		\item[\textbf{iv')}] for each $t \in [\![0,T-1]\!]$:
		\begin{itemize}
			\item $\xi_t$ is  $\mathcal{G}_t$-measurable and $\xi_{t+k}$, $k=1,\dots,T-t$, are jointly independent of $\mathcal{G}_t$,
			
			\item $\mathcal{G}_t=\mathcal{H}_t \lor \sigma(\mu^{(t)})\lor \sigma(\Phi)\lor \sigma(X_0)$, with  $\mathcal{H}_t$ independent of $\sigma(\Phi,(\mu_t)_{t=0}^T ,X_0)$,
			
			\item $u_t$ is $\mathcal{G}_t$-measurable;
		\end{itemize}		 
		\item[\textbf{v)}] 
		for all $t \in [\![0,T-1]\!]$,	
		\begin{equation}
			\begin{split}
				&X_{t+1}=\Psi\left(t,X_{t},\mu_t, u_t,\xi_{t+1}\right), \qquad \mathbb{P}\text{-a.s.}.
			\end{split}
		\end{equation}
	\end{itemize}
	
	We denote with $\mathcal{A}$ the set of all open-loop control policies and, with a slight abuse of notation, in the following we write just $(u_t)_{t=0}^{T-1} \in \mathcal{A}$ for $((\Omega, \mathcal{F},\mathbb{P}),$ $\{\mathcal{G}_t\}_{t=0}^{T-1}$, $\Phi,(\mu_t)_{t=0}^T,X_0,(\xi_t)_{t=1}^T, (u_t)_{t=0}^{T-1})\in \mathcal A$.
	\\
	We call any tuple $\big((\Omega, \mathcal{F},\mathbb{P}),\{\mathcal{G}_t\}_{t=0}^{T-1},\Phi, (\mu_t)_{t=0}^T, X_0,(\xi_t)_{t=1}^T, (u_t)_{t=0}^{T-1} , (X_t)_{t=0}^T \big)$ as above a \emph{ realization of the triple $(\mathfrak{m}_0, \rho, (u_t)_{t=0}^{T-1})$}.
\end{definition}

\begin{remark}
	Notice that this new setting includes the previous one. Indeed, setting $u_t=w\circ\Phi(t,X^{(t)},\mu^{(t)})$, $t \in [\![0,T-1]\!]$, the recursive structure of the problem yields that $u_t$ is $\mathcal{G}_t$-measurable with $\mathcal{G}_t= \sigma(X_0)\lor\sigma(\Phi)\lor \sigma(\mu^{(t)})\lor \sigma (\xi^{(t)})$, that is $\mathcal{H}_t= {\sigma}(\xi^{(t)})$, and thus all the conditions in \textbf{iv')} hold.
	In particular, the closed-loop strategy $\iota$, corresponding to the case in which the $i^{th}$-player follows the mediator's suggestion induces the open-loop admissible strategy 
	\begin{equation}
		u^{\iota}_t
		:=\iota \circ \Phi(t, X^{(t)},\mu^{(t)})
		=\Phi(t,X_t), \qquad \Phi \in \mathcal{R}.
	\end{equation}
\end{remark}

In this case the costs associated to the triple $(\mathfrak{m}_0, \rho, (u_t)_{t=0}^{T-1} ) \in \mathcal{P(X)}\times  \mathcal{P}\big(\mathcal{R}  \times  \mathcal{P(X)}^{T+1}\big) \times \mathcal A$ are given by
\begin{equation}
	\widehat J(\mathfrak{m}_0, \rho, (u_t)_{t=0}^{T-1})
	:= \mathbb{E}\left[ \sum_{t=0}^{T-1}f\left(t,X_{t},\mu_t,u_t\right) +F\left(X_T,\mu_T\right) \right].
\end{equation}
In this definition of the costs, there is a little abuse of notation. Indeed, $(u_t)_{t=0}^{T-1} \in \mathcal{A}$ stands for $((\Omega, \mathcal{F},\mathbb{P}),$ $\{\mathcal{G}_t\}_{t=0}^{T-1}$, $\Phi,(\mu_t)_{t=0}^T,X_0,(\xi_t)_{t=1}^T, (u_t)_{t=0}^{T-1}) \in \mathcal{A}$.

\begin{definition}\label{Def_CMFG_open}
	We say that $\rho \in  \mathcal{P}\big(\mathcal{R} \times \mathcal{ P(X)}^{T+1}\big)$ is an \emph{open-loop correlated solution for the mean field game} with initial distribution $\mathfrak{m}_0 \in \mathcal{P(X)}$, if the following two conditions hold:
	\begin{itemize}
		\item[\textbf{(Opt)}] For each  strategy modification  $(u_t)_{t=0}^{T-1} \in \mathcal A $,
\[ \widehat J(\mathfrak{m}_0, \rho, (u^\iota_t)_{t=0}^{T-1}) \leq \widehat J(\mathfrak{m}_0, \rho, (u_t)_{t=0}^{T-1}).\]

		\item[\textbf{(Con)}] For any realization of  $(\mathfrak{m}_0,  {\rho}, (u^\iota_t)_{t=0}^{T-1})$, namely $((\Omega, \mathcal{F},\{\mathcal{G}_t\}_{t=0}^{T-1},\mathbb{P})$, $\Phi,(\mu_t)_{t=0}^T$,  $X_0,(\xi_t)_{t=1}^T$, $(u^\iota_t)_{t=0}^{T-1}$, $(X_t)_{t=0}^{T})$, setting $\mathcal{F}^\mu:=\sigma\big((\mu_t)_{t=0}^T\big)$, we have	
		\[
\mu_t(\cdot)=\mathbb{P}(X_{t}\in \cdot \text{ } | \mathcal{F}^\mu ), \quad t \in [\![0,T]\!].
		\]
	\end{itemize}
\end{definition}

%%%%%%%%%%%%%%%%%%%%%%%%%%%%%%%%%%%%%%%%%%%%%%%%%%%%%%%%%%%
%%%%%%%%%%%%%%%%%%%%%%%%%%%%%%%%%%%%%%%%%%%%%%%%%%%%%%%%%%%
%%%%%%%%%%%%%%%%%%%%%%%%%%%%%%%%%%%%%%%%%%%%%%%%%%%%%%%%%%%

\subsection{The optimal value of the objective functional in the MFG}\label{Subsect_val_of_MFG}
This section is devoted to proving that the value of the objective functional at equilibrium in the limit game remains the same if we enlarge the set of admissible strategies to include open-loop controls with the information structure given in Definition~\ref{mf_real_relaxed}.

We start by showing that, under suitable technical assumptions needed to guarantee the well-posedness of all the conditional expectations involved, a conditional \emph{Dynamic Programming Principle} holds for MFG solutions in the sense of Definition~\ref{Def_CMFG_closed}. Then, we prove by backward induction in time that the value of the MFG in closed-loop strategies is the same as the one in open-loop strategies and that, therefore, a closed-loop solution according to Definition~\ref{Def_CMFG_closed} is also an open-loop solution according to Definition~\ref{Def_CMFG_open}.
	
Our first assumption requires the state dynamics to be non-degenerate; more precisely:
	\begin{hypenv}
	
		\item \label{HypPsiTransitions} For any $t \in [\![0,T-1]\!]$, any $m \in \mathcal{P(X)}$, any $x,y \in \mathcal{X}$ and any $u \in \Gamma$,
			\begin{align*}
				\mathbb{P}(\Psi(t,x,m,u,Z)=y)>0,
			\end{align*}
			where $Z$ is a random variable with distribution $\nu$.
	\end{hypenv}
	
In addition, we make a finiteness assumption on the structure of the correlated solution. To this end, let $\rho$ be a solution of the MFG starting at $m_0$ according to Definition~\ref{Def_CMFG_closed}. Consider a realization $\big((\Omega, \mathcal{F},\mathbb{P}),\Phi,(\mu_t)_{t=0}^T,X_0,(\xi_t)_{t=1}^T, w, (X_t)_{t=0}^T \big)$ of $(m_0, \rho, w)$ according to Definition~\ref{mf_real}. Given the fact that $\mathcal{R}$ is finite and limiting our analysis to the functions $\varphi \in \mathcal{R}$ such that $\mathbb{P}(\Phi=\varphi)> 0$, the induced conditional probability $\mathbb{P}_{\varphi}(\cdot):=\mathbb{P}(\cdot|\Phi=\varphi)$ is well-defined. The finiteness assumption on $\rho$ is now:

	\begin{Hypenv}
		\item \label{HypSolutionFlows} If $(\Phi,(\mu_t)_{t=0}^T)$ is distributed according to $\rho$, then there exists, for any choice of $\varphi \in \mathcal{R}$ such that $\mathbb{P}(\Phi=\varphi)>0$, a subset $\mathcal{P}_\varphi \subset \mathcal{P(X)}^{T+1}$ of finite cardinality such that $\mathbb{P}_\varphi(\mu^{(T)} \in \mathcal{P}_\varphi)=1$ and, for any $m \in  \mathcal{P}_\varphi$, $\mathbb{P}_\varphi(\mu^{(T)} = m )>0$. 	 
	\end{Hypenv}
		
	\begin{remark}
		The assumptions above are used to ensure the well-posedness of conditional probabilities of the form $\mathbb{P}_\varphi(\cdot|\mu^{(t)} = m^{(t)} , X^{(t)}=x^{(t)})$, for any $m^{(t)} \in \mathcal{P}_\varphi^{(t)}$, any $x^{(t)}\in \mathcal{X}^{t+1}$, where $\mathcal{P}_\varphi^{(t)}:= \pi_{\mathcal{P(X)}^{t+1}}(\mathcal{P}_\varphi)= \{ m \in \mathcal{P(X)}^{t+1} \text{ s.t.  there exists }l \in \mathcal{P(X)}^{T-t} \text{ s.t. }(m,l)\in \mathcal{P}_\varphi\} $. Indeed, for this to hold it is enough to check that $\mathbb{P}_\varphi(\mu^{(T)}=m, X^{(T)}=x^{(T)}) >0$, for any $x^{(T)} \in \mathcal{X}^{T+1}$ and any $m \in  \mathcal{P}_\varphi$. First, exploiting  disintegration we write
		\begin{align}\label{Eq_wp_cond_p_1}
			\mathbb{P}_\varphi(\mu^{(T)} =m, X^{(T)}=x^{(T)})
			= \mathbb{P}_\varphi(X^{(T)}=x^{(T)} | \mu^{(T)}=m) \cdot \mathbb{P}_\varphi(\mu^{(T)}=m),
		\end{align}
		where the second term in the product on the right is clearly strictly positive by Assumption \Hypref{HypSolutionFlows}.
		Then, another round of disintegration yields
		\begin{align}\label{Eq_wp_cond_p_2}
			\mathbb{P}_\varphi& (X^{(T)}=x^{(T)} | \mu^{(T)} =m) 
			\nonumber \\
			& = \mathbb{P}_\varphi(X_0=x_0 | \mu^{(T)} =m) \prod_{t=0}^{T-1}\mathbb{P}_\varphi(X_{t+1}=x_{t+1} | \mu^{(T)} =m ,X^{(t)}=x^{(t)})\\ \nonumber
			& = m_0(\{x_{0}\}) \prod_{t=0}^{T-1}\mathbb{P}_\varphi(X_{t+1}=x_{t+1} | \mu^{(T)} =m ,X^{(t)}=x^{(t)}).
		\end{align}
		 Now, exploiting the iterative dynamics of the state in the game, we have that, for any fixed $t \in [\![0,T-1]\!]$,
		\begin{align}\label{Eq_wp_cond_p_3}
			& \mathbb{P}_\varphi (X_{t+1}=x_{t+1} | \mu^{(T)} =m, X^{(t)}=x^{(t)})\nonumber \\
			& = \mathbb{P}_\varphi(\Psi(t,x_{t},m_t,u_t,\xi_{t+1})=x_{t+1} | \mu^{(T)} =m, X^{(t)}=x^{(t)})\\ \nonumber
			& = \sum_{\gamma \in \Gamma}\mathbb{P}_\varphi(\Psi(t,x_{t},m_t,\gamma,\xi_{t+1})=x_{t+1})\mathbb{P}_\varphi(u_t=\gamma | \mu^{(T)}  =m, X^{(t)}=x^{(t)}) %\\ \nonumber
			%& \geq \sum_{\gamma \in \Gamma}\varepsilon_\gamma \mathbb{P}_\varphi(u_t=\gamma | \mu^{(T)} =m, X^{(t)}=x^{(t)})
			%\geq \min_{\gamma \in \Gamma}\varepsilon_\gamma 
			>0.
		\end{align}		 
		Hence, putting together Equations \eqref{Eq_wp_cond_p_1}, \eqref{Eq_wp_cond_p_2} and \eqref{Eq_wp_cond_p_3}, we get 
		\begin{align}
			\mathbb{P}_\varphi(\mu^{(T)}  =m, X^{(T)}=x^{(T)})
			\geq \mathbb{P}_\varphi(\mu^{(T)} =m) m_0(\{x_{0}\}) \prod_{t=0}^{T-1}\mathbb{P}_\varphi(X_{t+1}=x_{t+1} | \mu^{(T)} =m ,X^{(t)}=x^{(t)}) 
			>0.
		\end{align}
		Finally notice that the very same proof can be carried out replacing $u_t$ with $w\circ\varphi(t,x^{(t)},m^{(t)})$ and so the result holds, in particular, for the MFG in Definition \ref{mf_real}.
	\end{remark}
Let $\rho$ be a solution of the MFG starting at $m_0$ and satisfying \Hypref{HypSolutionFlows}. Consider a realization $\big((\Omega, \mathcal{F},\mathbb{P}),\Phi,(\mu_t)_{t=0}^T,X_0,(\xi_t)_{t=1}^T, w, (X_t)_{t=0}^T \big)$ of $(m_0, \rho, w)$. Set $\rho_2(\cdot|\varphi) = \mathbb{P}(\mu \in \cdot| \Phi=\varphi)$. Such a realization then has the following properties, conditionally on the event $\{\Phi=\varphi\}$:
	\begin{itemize}
		\item[\textbf{i)}$_\varphi$] $\mathbb{P}_{\varphi}\circ (X_0)^{-1}=\mathfrak{m}_0$;
	
		\item[\textbf{ii)$_\varphi$}] $\mathbb{P}_{\varphi}\circ (\Phi,(\mu_t)_{t=0}^T)^{-1}=\mathbb{P}_{\varphi}\circ (\varphi,(\mu_t)_{t=0}^T)^{-1}=\delta_{\varphi}\otimes \rho_2(\cdot|\varphi)$;
	
		\item[\textbf{iii)$_\varphi$}] $(\xi_t)_{t=1}^T$ are i.i.d.\ all distributed according to $\mathbb{P}_{\varphi}\cdot(\xi_t)^{-1}=\nu$;
	
		\item[\textbf{iv)$_\varphi$}] $(\xi_t)_{t=1}^T$, $X_0$ and $(\mu_t)_{t=0}^T$ are independent w.r.t.\ $\mathbb{P}_{\varphi}$;
	
		\item[\textbf{v)$_\varphi$}] 
		for all $t \in [\![0,T-1]\!],$
		\begin{equation}
			\begin{split}
				&X_{t+1}=\Psi\left(t,X_{t},\mu_t,w\circ\varphi(t,X^{(t)},\mu^{(t)}),\xi_{t+1}\right), \qquad \mathbb{P}_{\varphi}\text{-a.s.}.
			\end{split}
		\end{equation}
	\end{itemize}
	Notice that properties \textbf{i)$_\varphi$, ii)$_\varphi$, iii)$_\varphi$, iv)$_\varphi$ } and \textbf{v)$_\varphi$} are a consequence of the corresponding properties in the unconditional setting and the independence in property \textbf{iv)}.
		
	Hence, the (conditional) costs associated to the triple $(\mathfrak{m}_0, \rho, w ) \in \mathcal{P(X)}\times  \mathcal{P}\big(\mathcal{R}  \times  \mathcal{P(X)}^{T+1}\big) \times \widehat{\mathcal{D}}$ are
	\begin{equation}
		J_{\varphi}(\mathfrak{m}_0, \rho, w)
		:= \mathbb{E}_{\varphi}\left[ \sum_{t=0}^{T-1}f\left(t,X_{t},\mu_t,w\circ\varphi(t,X^{(t)},\mu^{(t)})\right) +F\left(X_T,\mu_T\right) \right],
	\end{equation}
	where $\mathbb{E}_{\varphi}[\cdot]:=\mathbb{E}[\cdot|\Phi=\varphi]$.
	\\

	We set
	\begin{align*}
		& J_{\varphi}(t,x^{(t)},m^{(t)},w)=\mathbb{E}_{\varphi}\left[\sum_{s=t}^{T-1}f(s,X_s,\mu_s,w\circ\varphi(s,X^{(s)},\mu^{(s)}))+ F(X_T,\mu_T)| X^{(t)}=x^{(t)}, \mu^{(t)}=m^{(t)}\right].
	\end{align*}
	and thus, in particular,
	\begin{align*}
		& J_{\varphi}(T,x^{(T)},m^{(T)},w)=F(x_T,m_T).
	\end{align*}
	
	Notice that, for any fixed $t \in [\![0,T-1]\!]$, $X^{(t)}$ and $(\mu_s)_{s=t+1}^T$ are $\mathbb{P}_\varphi$-conditionally independent given $\mu^{(t)}$. 
	Indeed, consider a fixed $t \in [\![0,T-1]\!]$ and let $x^{(t)} \in \mathcal{X}^{t+1}$ and $m^{[t+1]}:=(m_s)_{s=t+1}^T \in \mathcal{P}_\varphi^{[T-t]}$, with $\mathcal{P}_\varphi^{[T-t]}:= \pi_{\mathcal{P(X)}^{T-t}}(\mathcal{P}_\varphi)= \{ m \in \mathcal{P(X)}^{T-t} \text{ s.t.  there exists }l \in \mathcal{P(X)}^{t+1} \text{ s.t. }(l,m)\in \mathcal{P}_\varphi\} $. 
	Exploiting tower property and measurability, we have 
	\begin{align*}
		\mathbb{P}_{\varphi}  (X^{(t)}=x^{(t)}, (\mu_s)_{s=t+1}^T=m^{[t+1]}|\mu^{(t)})\nonumber &
		= \mathbb{E}_{\varphi}[\mathbb{E}_{\varphi}[\mathbf{1}_{\{x^{(t)}\}}(X^{(t)})\mathbf{1}_{\{m^{[t+1]}\}}((\mu_s)_{s=t+1}^T)|X_0,\mu^{(t)}, \xi^{(t)}]|\mu^{(t)}]\\
		& = \mathbb{E}_{\varphi}[\mathbf{1}_{\{x^{(t)}\}}(X^{(t)})\mathbb{E}_{\varphi}[\mathbf{1}_{\{m^{[t+1]}\}}((\mu_s)_{s=t+1}^T)|X_0,\mu^{(t)}, \xi^{(t)}]|\mu^{(t)}]\\
		& = \mathbb{E}_{\varphi}[\mathbf{1}_{\{x^{(t)}\}}(X^{(t)})\mathbb{E}_{\varphi}[\mathbf{1}_{\{m^{[t+1]}\}}((\mu_s)_{s=t+1}^T)|\mu^{(t)}]|\mu^{(t)}]\\
		& = \mathbb{E}_{\varphi}[\mathbf{1}_{\{m^{[t+1]}\}}((\mu_s)_{s=t+1}^T)|\mu^{(t)}]\mathbb{E}_{\varphi}[\mathbf{1}_{\{x^{(t)}\}}(X^{(t)})|\mu^{(t)}]\\
		& = \mathbb{P}_{\varphi}(X^{(t)}=x^{(t)}|\mu^{(t)})\mathbb{P}_\varphi((\mu_s)_{s=t+1}^T=m^{[t+1]}|\mu^{(t)}).
	\end{align*}
	As a consequence of the conditional independence stated above, $ J_{\varphi}(t,x^{(t)},m^{(t)},w)= J_{\varphi}(t,x^{(t)},m^{(t)},\widetilde w)$ if $w\circ \varphi(u, \cdot)=\widetilde w\circ \varphi(u, \cdot)$, for $u \geq t$.
	Indeed, take $w, \widetilde{w} \in \widehat{\mathcal{D}}$ such that $w\circ \varphi(u, \cdot)=\widetilde w\circ \varphi(u, \cdot)$, for $u \geq t$. We have	
	\begin{align}\label{Eq_J_ind_past_strat}
		&J_{\varphi}(t,x^{(t)},m^{(t)},\widetilde w)
		 = \mathbb{E}_{\varphi}\left[\sum_{s=t}^{T-1}f(s,X_s,\mu_s,\widetilde w\circ \varphi(s,X^{(s)},\mu^{(s)}))+ F(X_T,\mu_T)| X^{(t)}=x^{(t)}, \mu^{(t)}=m^{(t)}\right]\\ \nonumber
		& =: \mathbb{E}_{\varphi}\left[G_t(x^{(t)},m^{(t)},(\mu_s)_{s=t+1}^T, (\xi_s)_{s=t+1}^T,(\widetilde w\circ \varphi(s,\cdot))_{s=t}^T)| X^{(t)}=x^{(t)}, \mu^{(t)}=m^{(t)}\right]\\ \nonumber
		& = \int_{\mathcal{Z}^{T-t}}\sum_{m^{[t+1]} \in \mathcal{P}_\varphi^{T-t}} G_t(x^{(t)},m^{(t)},m^{[t+1]},(z_s)_{s=t+1}^{T},(\widetilde w\circ \varphi(s,\cdot))_{s=t}^T) \mathbb{P}_\varphi((\mu_s)_{s=t+1}^T=m^{[t+1]} |\mu^{(t)})\nu^{\otimes (T-t-1)}(dz)\\ \nonumber
		& = \int_{\mathcal{Z}^{T-t}}\sum_{m^{[t+1]} \in \mathcal{P}_\varphi^{T-t}} G_t(x^{(t)},m^{(t)},m^{[t+1]},(z_s)_{s=t+1}^{T},(w\circ \varphi(s,\cdot))_{s=t}^T) \mathbb{P}_\varphi((\mu_s)_{s=t+1}^T=m^{[t+1]} |\mu^{(t)})\nu^{\otimes (T-t-1)}(dz)\\ \nonumber
		& = J_{\varphi}(t,x^{(t)},m^{(t)}, w),
	\end{align}
	where we have used the notation $d z = dz_{t+1},\dots,dz_{T}$ and in the third identity we have exploited the fact that, for any fixed $t \in [\![0,T-1]\!]$, $X^{(t)}$ and $(\mu_s)_{s=t+1}^T$ are $\mathbb{P}_\varphi$-conditionally independent given $\mu^{(t)}$.
	\\
	Then, we write $ J_{\varphi}(t,x^{(t)},m^{(t)},(w_s)_{s=t}^T)= J_{\varphi}(t,x^{(t)},m^{(t)},w)$.
	Thus, the optimal value function is defined as
	\begin{align*}
		 V_{\varphi}& (t,x^{(t)},m^{(t)})\\
		& =\inf_{w_t \in \widehat{\mathcal{R}}_t}\mathbb{E}_{\varphi}\left[\sum_{s=t}^{T-1}f(s,X_s,\mu_s,w_t(s,X^{(s)},\mu^{(s)}))+ F(X_T,\mu_T)| X^{(t)}=x^{(t)}, \mu^{(t)}=m^{(t)}\right],
	\end{align*}
	where $\widehat{\mathcal{R}}_t:= \{ w:[\![t, T]\!] \times \mathcal{X}^T\times \mathcal{P}_{\varphi}\to \Gamma, \text{ progressively measurable}\}$.
	
Our aim, now, is to show that, even in this non-Markovian setting, the following DPP holds.
	
	\begin{proposition}\label{Prop_cond_DPP_CMFG}
		For any $t \in [\![0,T-1]\!]$,
		\begin{align*}
			V_{\varphi}&(t,x^{(t)},m^{(t)})\\&
			=  \inf_{\gamma\in \Gamma}\mathbb{E}_{\varphi}	\bigg[f(t,x_t,m_t,\gamma)
			 +V_{\varphi}\left(t,(x^{(t)},\Psi(t,x_t,m_t,\gamma,\xi_{t+1})),(m^{(t)},\mu_{t+1})\right)\bigg|X^{(t)}=x^{(t)},\mu^{(t)}=m^{(t)}\bigg].
		\end{align*}
	\end{proposition}
	
	\begin{proof}
		By construction and measurability properties, it holds
		\begin{align*}
			 V_{\varphi}&(t,x^{(t)},m^{(t)})\\
			&=\inf_{w_t \in \widehat{\mathcal{R}}_t}\mathbb{E}_{\varphi}\Bigg[\sum_{s=t}^{T-1}f(s,X_s,\mu_s,w_t(s,X^{(s)},\mu^{(s)})) + F(X_T,\mu_T)\Big| X^{(t)}=x^{(t)}, \mu^{(t)}=m^{(t)}\Bigg]\\
			& =\inf_{w_t \in \widehat{\mathcal{R}}_t}\mathbb{E}_{\varphi}\Bigg[f(t,x_t,m_t,w_t(t,x^{(t)},m^{(t)}))+\sum_{s=t+1}^{T-1}f(s,X_s,\mu_s,w_t(s,X^{(s)},\mu^{(s)}))\\
			& \qquad + F(X_T,\mu_T)\Big| X^{(t)}=x^{(t)}, \mu^{(t)}=m^{(t)}\Bigg]\\
			%%%%%%%%%%%%%%%%%%%%%%%%%%%%%%%%%%
			& =\inf_{w_t \in \widehat{\mathcal{R}}_t}\Bigg\{f(t,x_t,m_t,w_t(t,x^{(t)},m^{(t)}))\\
			& \quad +\sum_{(y,l) \in \mathcal{X}\times \mathcal{P}_\varphi}\mathbb{P}_\varphi(\Psi(t,x_t,m_t,w_t(t,x^{(t)},m^{(t)}),\xi_{t+1})=y,\mu_{t+1}=l| X^{(t)}=x^{(t)}, \mu^{(t)}=m^{(t)}) \cdot\\
			& \qquad \cdot \mathbb{E}_{\varphi}\Big[\sum_{s=t+1}^{T-1}f(s,X_s,\mu_s,w_t(s,X^{(s)},\mu^{(s)}))+ F(X_T,\mu_T)\Big| X^{(t+1)}=(x^{(t)},y), \mu^{(t+1)}=(m^{(t)},l)\Big]	
			\Bigg\}
		\end{align*}
		Now, exploiting the fact that $\mathcal{P}_\varphi$ is finite (to exchange the $\inf$ and the summation) and the conditional independence property shown above (together with the consequent identity in Equation \eqref{Eq_J_ind_past_strat}), we have
		\begin{align*}
			 V_{\varphi}&(t,x^{(t)},m^{(t)})\\
			& =\inf_{\gamma \in \Gamma}\inf_{w_{t+1} \in \widehat{\mathcal{R}}_{t+1}}\Bigg\{f(t,x_t,m_t,\gamma)\\
			& \quad +\sum_{(y,l) \in \mathcal{X\times P_\varphi}}\mathbb{P}_\varphi (\Psi(t,x_t,m_t,\gamma,\xi_{t+1})=y,\mu_{t+1}=l| X^{(t)}=x^{(t)}, \mu^{(t)}=m^{(t)}) \cdot\\
			& \qquad \cdot \mathbb{E}_{\varphi}\Big[\sum_{s=t+1}^{T-1}f(s,X_s,\mu_s,w_s(X^{(s)},\mu^{(s)})) + F(X_T,\mu_T)\Big| X^{(t+1)}=(x^{(t)},y), \mu^{(t+1)}=(m^{(t)},l)\Big]	
		\Bigg\}\\
			%%%%%%%%%%%%%%%%%%%%%%%%%%%%%%%%%%%%%%%%%%
			& =\inf_{\gamma \in \Gamma}\Bigg\{f(t,x_t,m_t,\gamma)\\
			& \quad +\sum_{(y,l) \in \mathcal{X\times P_\varphi}}\inf_{w_{t+1} \in \widehat{\mathcal{R}}_{t+1}}\Bigg\{\mathbb{P}_\varphi (\Psi(t,x_t,m_t,\gamma,\xi_{t+1})=y,\mu_{t+1}=l| X^{(t)}=x^{(t)}, \mu^{(t)}=m^{(t)})\cdot\\
			& \qquad \cdot 
			\mathbb{E}_{\varphi}\Big[\sum_{s=t+1}^{T-1}f(s,X_s,\mu_s,w_s(X^{(s)},\mu^{(s)})) + F(X_T,\mu_T)\Big| X^{(t+1)}=(x^{(t)},y), \mu^{(t+1)}=(m^{(t)},l)\Big]
			\Bigg\} \Bigg\}\\
			%%%%%%%%%%%%%%%%%%%%%%%%%%%%%%%%%%%%%%%%%%%%%%%%%
			& =\inf_{\gamma \in \Gamma}\Bigg\{f(t,x_t,m_t,\gamma) 
			+\sum_{(y,l) \in \mathcal{X\times P_\varphi}} \Bigg\{ V_\varphi(t+1,(x^{(t+1)},y),(m^{(t)},l)) \cdot\\
			& \qquad \cdot \mathbb{P}_\varphi (\Psi(t,x_t,m_t,\gamma,\xi_{t+1})=y,\mu_{t+1}=l| X^{(t)}=x^{(t)}, \mu^{(t)}=m^{(t)}) \Bigg\} \Bigg\}\\
			%%%%%%%%%%%%%%%%%%%%%%%%%%%%%%%%%%%%%%%%%%%%%%%
			& = \inf_{\gamma\in \Gamma}\Bigg\{\mathbb{E}_{\varphi}\Big[f(t,x_t,m_t,\gamma)
			+ V_{\varphi}(t,(x^{(t)},\Psi(t,x_t,m_t,\gamma,\xi_{t+1})),(m^{(t)},\mu_{t+1}))|X^{(t)}=x^{(t)},\mu^{(t)}=m^{(t)}\Big]\Bigg\}.
		\end{align*}
	\end{proof}
	Thus, we have shown the DPP and we can proceed with the second step.

\begin{proposition}\label{Prop_equiv_CMFGs}
	Assume \hypref{HypPsiTransitions}. Let  $\rho \in \mathcal{P}(\mathcal{R}\times \PX)$ be a correlated solution of the MFG in closed-loop strategies starting at $\mathfrak{m}_0$ according to Definition  \ref{Def_CMFG_closed}.
	If $\rho$ satisfies \Hypref{HypSolutionFlows}, then $\rho$ is a solution for the mean field game in open-loop strategies, as in Definition \ref{Def_CMFG_open}, too. 
	In particular,  for any $\varphi \in \mathcal{R}$, $\widehat V_{\varphi}(t,x^{(t)},m^{(t)})$ and $ V_{\varphi}(t,x^{(t)},m^{(t)})$ coincide.
\end{proposition}

\begin{remark}
	Notice that, since the consistency conditions in Definitions \ref{Def_CMFG_closed} and \ref{Def_CMFG_open} are the same and the set of closed-loop strategies is included in the set of open-loop strategies, a solution of the correlated MFG in open-loop strategies, $\rho \in \mathcal{P}(\mathcal{R}\times \PX)$, is automatically a solution for the corresponding game in closed-loop strategies.
\end{remark}

\begin{proof}	
		
	We have already discussed the form of the objective functional for the MFG in closed-loop controls when showing the DPP.
	Regarding the relaxed MFG, conditionally on the suggestion received by the representative player (that is on the event $\{\Phi=\varphi\}$), a realization of $(m_0, \rho, (u_t)_{t=0}^{T-1})$, i.e. a tuple $\big((\Omega, \mathcal{F},\{\mathcal{G}_t\}_{t=0}^{T-1},\mathbb{P}),\Phi, (\mu_t)_{t=0}^T$, $X_0,(\xi_t)_{t=1}^T, (u_t)_{t=0}^{T-1}, (X_t)_{t=0}^T \big)$, satisfies the following:
	\begin{itemize}
		\item[\textbf{i)$_{\varphi}$}] ${\mathbb{P}}_{\varphi}\circ (X_0)^{-1}=\mathfrak{m}_0$;
	
		\item[\textbf{ii)$_{\varphi}$}] ${\mathbb{P}}_{\varphi}\circ (\Phi,(\mu_t)_{t=0}^T)^{-1}=\delta_{\varphi}\otimes\rho_2(\cdot|\varphi)$;
	
		\item[\textbf{iii)$_{\varphi}$}] $(\xi_t)_{t=1}^T$ are i.i.d. all distributed according $\mathbb{P}_{\varphi}\circ(\xi_t)^{-1}=\nu;$
	
		\item[\textbf{iv)$_{\varphi}$}] $(\xi_t)_{t=1}^T$, $X_0$ and $(\mu_t)_{t=0}^T$ are independent w.r.t. $\mathbb{P}_{\varphi}$;
		
		\item[\textbf{iv')$_{\varphi}$}] For each $t \in [\![0,T-1]\!]$,
		\begin{itemize}
			\item $\xi_t$ is  $\mathcal{G}_t$-measurable and $\xi_{t+k}$, $k=1,\dots,T-t$, are jointly independent of $\mathcal{G}_t$ w.r.t. $\mathbb{P}_{\varphi}$,
			
			\item $\mathcal{G}_t=\mathcal{H}_t \lor \sigma(\mu^{(t)})\lor \sigma(\Phi)\lor \sigma(X_0)$, with  $\mathcal{H}_t$ independent of $\sigma((\mu_t)_{t=0}^T,X_0,\Phi)$ w.r.t. $\mathbb{P}_{\varphi}$,
			
			\item $u_t$ is $\mathcal{G}_t$-measurable,	
		\end{itemize}		 
		\item[\textbf{v')$_{\varphi}$}] 
		for any $t \in [\![0,T-1]\!],$
		\begin{equation}
			\begin{split}
				&X_{t+1}=\Psi\left(t,X_{t},\mu_t,u_t,\xi_{t+1}\right), \qquad \mathbb{P}_\varphi\text{-a.s.}.
			\end{split}
		\end{equation}
	\end{itemize}
	
	Let's quickly review how we check the properties in \textbf{iv')$_{\varphi}$}. Notice that the other properties are trivial. Let us first recall that measurability properties concern $\sigma$-algebras and not the specific probability measure on them, hence we have to exhibit proofs only for the independence properties. For any arbitrary fixed $t \in [\![0,T-1]\!]$, we have
		\begin{itemize}
			\item $(\xi_{t+k})_{k=1}^{T-t}$, are jointly independent of $\mathcal{G}_t$ w.r.t. $\mathbb{P}_{\varphi}$. Indeed, let $A \in \mathcal{G}_t$ and $(B_k)_{k=1}^{T-t} \in \mathcal{B}(\mathcal{Z})$,  exploiting the tower property, the fact that $\sigma(\Phi) \subset \mathcal{G}_t$ and the fact that $\mathcal{G}_t$ and $(\xi_{t+k})_{k=1}^{T-t}$ are independent w.r.t. $\mathbb{P}$, we obtain
			\begin{align*}
				&\mathbb{P}_{\varphi}(A\cap \{\xi_{t+1} \in B_1\} \cap \dots \cap \{\xi_{T} \in B_{T-t}\})
				= \frac{\mathbb{E}[\mathbf{1}_{\{\varphi\}}(\Phi)\mathbf{1}_A\mathbf{1}_{B_1}(\xi_{t+1}) \dots \mathbf{1}_{B_{T-t}}(\xi_{T})]}{\mathbb{P}(\Phi=\varphi)}\\
				& = \frac{\mathbb{E}[\mathbb{E}[\mathbf{1}_{\{\varphi\}}(\Phi)\mathbf{1}_A\mathbf{1}_{B_1}(\xi_{t+1}) \dots \mathbf{1}_{B_{T-t}}(\xi_{T})|\mathcal{G}_t]]}{\mathbb{P}(\Phi=\varphi)}
				 = \frac{\mathbb{E}[\mathbf{1}_{\{\varphi\}}(\Phi)\mathbf{1}_A\mathbb{E}[\mathbf{1}_{B_1}(\xi_{t+1}) \dots \mathbf{1}_{B_{T-t}}(\xi_{T})|\mathcal{G}_t]]}{\mathbb{P}(\Phi=\varphi)}\\
				& = \frac{\mathbb{E}[\mathbf{1}_{\{\varphi\}}(\Phi)\mathbf{1}_A\prod_{k=1}^{T-t}\mathbb{P}(\xi_{t+k} \in B_{k})]}{\mathbb{P}(\Phi=\varphi)} 
				= \prod_{k=1}^{T-t}\mathbb{P}(\xi_{t+k} \in B_{k}) \mathbb{P}(A|\Phi=\varphi).
			\end{align*}	
			\item $\mathcal{G}_t=\mathcal{H}_t \lor \sigma(\mu^{(t)})\lor \sigma(\Phi)\lor \sigma(X_0)$, with  $\mathcal{H}_t$ independent of $\sigma(\mu,X_0,\Phi)$ w.r.t. $\mathbb{P}_{\varphi}$. By assumption, $\mathcal{H}_t$, $\sigma(X_0)$ and $\sigma(\Phi,\mu)$ are independent w.r.t. $\mathbb{P}$. Take $A \in \mathcal{H}_t$, $B \in \sigma(\mu)$ and $C \in \sigma(X_0)$, we get
			\begin{align*}
				\mathbb{P}_{\varphi}(A\cap B \cap C)
				& = \frac{\mathbb{E}[\mathbf{1}_{\{\varphi\}}(\Phi)\mathbf{1}_A\mathbf{1}_{B}\mathbf{1}_{C}]}{\mathbb{P}(\Phi=\varphi)}
				= \frac{\mathbb{E}[\mathbb{E}[\mathbf{1}_{\{\varphi\}}(\Phi)\mathbf{1}_A\mathbf{1}_{B}\mathbf{1}_{C}|\Phi,\mu]]}{\mathbb{P}(\Phi=\varphi)}
				= \frac{\mathbb{E}[\mathbf{1}_{\{\varphi\}}(\Phi) \mathbf{1}_{B}\mathbb{E}[\mathbf{1}_A\mathbf{1}_{C}|\Phi,\mu]]}{\mathbb{P}(\Phi=\varphi)}\\
				& = \frac{\mathbb{E}[\mathbf{1}_{\{\varphi\}}(\Phi) \mathbf{1}_{B}\mathbb{E}[\mathbf{1}_A\mathbf{1}_{C}]]}{\mathbb{P}(\Phi=\varphi)}
				= \frac{\mathbb{P}(A)\mathbb{P}(C)\mathbb{E}[\mathbf{1}_{\{\varphi\}}(\Phi) \mathbf{1}_{B}]}{\mathbb{P}(\Phi=\varphi)}
				= \mathbb{P}_{\varphi}(A)\mathbb{P}_{\varphi}(C)\mathbb{P}_{\varphi}(B).
			\end{align*}
		\end{itemize}

	The (conditional) costs associated to the triple $(\mathfrak{m}_0, \rho, (u_t)_{t=0}^{T-1} ) \in \mathcal{P(X)}\times  \mathcal{P}\big(\mathcal{R}  \times  \mathcal{P(X)}^{T+1}\big) \times \mathcal A$ are
	\begin{equation}
		\widehat J_{\varphi}(\mathfrak{m}_0, \rho, (u_t)_{t=0}^{T-1})
		:= \mathbb{E}_{\varphi}\left[ \sum_{t=0}^{T-1}f\left(t,X_{t},\mu_t,u_t)\right) +F\left(X_T,\mu_T\right) \right].
	\end{equation}
	Then, 
	\begin{align*}
		& \widehat V_{\varphi}(t,x^{(t)},m^{(t)})=\inf_{(u_t)_{t=0}^{T-1}  \in \mathcal A} \mathbb{E}_{\varphi}\left[\sum_{s=t}^{T-1}f(s,X_s,\mu_s,u_s)+  F(X_T,\mu_T)| X^{(t)}=x^{(t)}, \mu^{(t)}=m^{(t)}\right],
	\end{align*}
	and so, in particular, at the terminal time $T \in \mathbb{N}$, we have
	\begin{align*}
		& \widehat V_{\varphi}(T,x^{(T)},m^{(T)})=F(X_T,m_T).
	\end{align*}
	We want to prove that $\widehat{V}_{\varphi}=V_{\varphi}$.
	One side of the inequality is straightforward. Indeed, closed-loop controls as in Definition \ref{mf_real} induce admissible open-loop controls in the sense of Definition \ref{mf_real_relaxed}, through
		\begin{align*}
		u_t:=w\circ\varphi(t, X^{(t)},\mu^{(t)}), \qquad t \in [\![0,T-1]\!].
		\end{align*}		 
		Thus, it holds $\widehat{V}_{\varphi} \leq V_{\varphi}$.
		We show that 	$\widehat{V}_{\varphi} \geq V_{\varphi}$, by backward induction on $t$. 
		We have $\widehat V_{\varphi}(T,x^{(T)},m^{(T)})=F(X_T,m_T)=V_{\varphi}(T,x^{(T)},m^{(T)})$. 
		Now, as an induction hypothesis, assume that $\widehat V_{\varphi}(t+1,x^{(t+1)},m^{(t+1)})= V_\varphi(t+1,x^{(t+1)},m^{(t+1)})$. 
		To prove that $\widehat V_{\varphi}(t,x^{(t)},m^{(t)})= V_{\varphi}(t,x^{(t)},m^{(t)})$, it is enough to check that $\widehat J_{\varphi}(t,x^{(t)},m^{(t)},(u_t)_{t=0}^{T-1})\geq  V_{\varphi}(t,x^{(t)},m^{(t)})$, for any admissible sequence of controls $u \in \mathcal{A}$.
		Exploiting the definitions and induction hypothesis, we see 	
		\begin{align*}
			&\widehat J_{\varphi}(t,x^{(t)},m^{(t)},(u_t)_{t=0}^{T-1})
			= \mathbb{E}_{\varphi}\left[\sum_{s=t}^{T-1}f(s,X_s,\mu_s,u_s)+  F(X_T,\mu_T) \bigg| X^{(t)}=x^{(t)}, \mu^{(t)}=m^{(t)}\right]\\
			%%%%%%%%%%%%%%%%%%%%%%%%%%%%%%%%%%%%%%%%%%%%%%%
			&= \mathbb{E}_{\varphi}\left[f(t,x_t,m_t,u_t)\Big| X^{(t)}=x^{(t)}, \mu^{(t)}=m^{(t)}\right]\\
			& \quad +\int_{\mathcal{X}\times\mathcal{P(X)}}\mathbb{P}_{\varphi}(X_{t+1}=X_{t+1}, \mu_{t+1}=m_{t+1}| X^{(t)}=x^{(t)}, \mu^{(t)}=m^{(t)})\cdot\\
			& \quad \cdot 	\mathbb{E}_{\varphi}\left[\sum_{s=t+1}^{T-1}f(s,X_s,\mu_s,u_s)+  F(X_T,\mu_T)\bigg| X^{(t+1)}=x^{(t+1)}, \mu^{(t+1)}=m^{(t+1)}\right] \\
			%%%%%%%%%%%%%%%%%%%%%%%%%%%%%%%%%%%%%%%%%%%%%%%%%%
			& = \mathbb{E}_{\varphi}\left[f(t,x_t,m_t,u_t)\Big| X^{(t)}=x^{(t)}, \mu^{(t)}=m^{(t)}\right]
			+\int_{\mathcal{X}\times\mathcal{P(X)}}\widehat J_\varphi(t+1,x^{(t+1)},m^{(t+1)},u)\cdot\\
			& \quad \cdot \mathbb{P}_{\varphi}(X_{t+1}=x_{t+1}, \mu_{t+1}=m_{t+1} | X^{(t)}=x^{(t)}, \mu^{(t)}=m^{(t)})\\
			&\geq \mathbb{E}_{\varphi}\left[f(t,x_t,m_t,u_t)\Big| X^{(t)}=x^{(t)}, \mu^{(t)}=m^{(t)}\right] +\int_{\mathcal{X}\times\mathcal{P(X)}} V_\varphi(t+1,x^{(t+1)},m^{(t+1)})\cdot\\
			& \quad \cdot \mathbb{P}_{\varphi}(X_{t+1}=x_{t+1}, \mu_{t+1}=m_{t+1}| X^{(t)}=x^{(t)}, \mu^{(t)}=m^{(t)}).
		\end{align*}
		Now, exploiting, in sequence, the fact that $\xi_{t+1}$ is jointly independent of $X^{(t)}, u_t$ and $\mu^{(t+1)}$, the tower property, the fact that $u_t$ is $\mathcal{G}_t$-measurable, the fact that $\mu_{t+1}$ and $\mathcal{G}_t$ are $\mathbb{P}_\varphi$-conditionally independent given $\mu^{(t)}$ and the measurability properties of conditional expectations, we obtain, for any $A \in \mathcal{B(P(X))}$, $B \in \mathcal{B}(\mathcal{Z})$ and $C \in \mathcal{B}(\Gamma)$, 
		\begin{align}\label{Eq_vars_cond_ind}
			& \mathbb{P}_{\varphi}(\mu_{t+1} \in A, \xi_{t+1} \in B, u_t \in C | X^{(t)}=x^{(t)}, \mu^{(t)}=m^{(t)})\\ \nonumber
			& = \mathbb{P}_{\varphi}( \xi_{t+1} \in B| X^{(t)}=x^{(t)}, \mu^{(t)}=m^{(t)})\mathbb{P}_{\varphi}(\mu_{t+1} \in A, u_t \in C | X^{(t)}=x^{(t)}, \mu^{(t)}=m^{(t)})\\ \nonumber
			& = \mathbb{P}_{\varphi}( \xi_{t+1} \in B)\frac{\mathbb{E}[(\mathbf{1}_{A}(\mu_{t+1})\mathbf{1}_{C}(u_{t}))(\mathbf{1}_{\{x^{(t)}\}}(X^{(t)})\mathbf{1}_{\{m^{(t)}\}}(\mu^{(t)})\mathbf{1}_{\{\varphi\}}(\Phi))]}{\mathbb{P}(X^{(t)}=x^{(t)}, \mu^{(t)}=m^{(t)}, \Phi=\varphi)}\\ \nonumber
			& = \mathbb{P}_{\varphi}( \xi_{t+1} \in B)\frac{\mathbb{E}[\mathbb{E}[\mathbf{1}_{A}(\mu_{t+1})|\mathcal{G}_t]] \mathbf{1}_{C}(u_{t})\mathbf{1}_{\{x^{(t)}\}}(X^{(t)})\mathbf{1}_{\{m^{(t)}\}}(\mu^{(t)})\mathbf{1}_{\{\varphi\}}(\Phi) }{\mathbb{P}(X^{(t)}=x^{(t)}, \mu^{(t)}=m^{(t)}, \Phi=\varphi)}\\ \nonumber
			& = \mathbb{P}_{\varphi}( \xi_{t+1} \in B)\frac{\mathbb{E}[\mathbb{E}[\mathbf{1}_{A}(\mu_{t+1})|\Phi, \mu^{(t)}]\mathbf{1}_{C}(u_{t})\mathbf{1}_{\{x^{(t)}\}}(X^{(t)})\mathbf{1}_{\{m^{(t)}\}}(\mu^{(t)})\mathbf{1}_{\{\varphi\}}(\Phi) ]}{\mathbb{P}(X^{(t)}=x^{(t)}, \mu^{(t)}=m^{(t)}, \Phi=\varphi)}\\ \nonumber
			& = \mathbb{P}_{\varphi}( \xi_{t+1} \in B)\mathbb{E}[\mathbf{1}_C(u_t) \mathbb{E}[\mathbf{1}_A(\mu_{t+1})| \Phi, \mu^{(t)} ] | X^{(t)}=x^{(t)}, \mu^{(t)}=m^{(t)}, \Phi=\varphi]\\ \nonumber
			& = \mathbb{P}_{\varphi}( \xi_{t+1} \in B)\mathbb{E}[\mathbf{1}_A(\mu_{t+1})| \Phi=\varphi, \mu^{(t)}=m^{(t)} ] \mathbb{E}[\mathbf{1}_C(u_t)  | X^{(t)}=x^{(t)}, \mu^{(t)}=m^{(t)}, \Phi=\varphi]\\ \nonumber
			& =\mathbb{P}_{\varphi}( \xi_{t+1} \in B) \mathbb{P}_{\varphi}(\mu_{t+1} \in A|\mu^{(t)}=m^{(t)})\mathbb{P}_{\varphi}(u_t \in C | X^{(t)}=x^{(t)}, \mu^{(t)}=m^{(t)}) \\ \nonumber
			& =\nu(B)\mathbb{P}_{\varphi}(\mu_{t+1} \in A|\mu^{(t)}=m^{(t)})\lambda_t(C),
		\end{align}
		where $\lambda_t(C):=\mathbb{P}_{\varphi}(u_t \in C | X^{(t)}=x^{(t)}, \mu^{(t)}=m^{(t)})$.
		Then, exploiting the iterative dynamics of the state and Equation \eqref{Eq_vars_cond_ind}, we have
		\begin{align*}
			& \widehat J_{\varphi}(t,x^{(t)},m^{(t)},(u_t)_{t=0}^{T-1})
			\geq \mathbb{E}_{\varphi}\Big[f(t,x,m_t,u_t) \\
			& \qquad +  V_{\varphi}(t+1,(x^{(t)},\Psi(t,x_t,m_t,u_t,\xi_{t+1})),(m^{(t)},\mu_{t+1}))| X^{(t)}=x^{(t)}, \mu^{(t)}=m^{(t)}\Big]\\
			& = \mathbb{E}_{\varphi}\Big[f(t,x,m_t,u_t)
			 + \int_{\mathcal Z}  V_{\varphi}(t+1,(x^{(t)},\Psi(t,x_t,m_t,u_t,z)),(m^{(t)},\mu_{t+1}))\nu(dz)| X^{(t)}=x^{(t)}, \mu^{(t)}=m^{(t)}\Big]\\
			& = \int_{\Gamma} \Bigg\{ f(t,x,m_t,\gamma)
			 + \int_{\mathcal Z}\mathbb{E}_{\varphi}\Big[ V_{\varphi}(t+1,(x^{(t)},\Psi(t,x_t,m_t,\gamma,z)),(m^{(t)},\mu_{t+1}))| \mu^{(t)}=m^{(t)}\Big] \nu(dz) \Bigg\} \lambda_t(d\gamma)\\
			&\geq \inf_{\gamma \in \Gamma}\Big\{f(t,x,m_t,\gamma))
			 + \int_{\mathcal{Z}}\mathbb{E}_{\varphi}\left[  V_{\varphi}(t+1,(x^{(t)},\Psi(t,x_t,m_t,\gamma,z)),(m^{(t)},\mu_{t+1}))| \mu^{(t)}=m^{(t)}\right]\nu(dz) \Big\}\\
			& = V_{\varphi}(t,x^{(t)},m^{(t)}),
		\end{align*}
		where the last identity follows from the DPP in Proposition \ref{Prop_cond_DPP_CMFG}.
	Finally, to conclude it is sufficient to integrate with respect to $\rho_1(d \varphi)$.
	Hence, we have shown that under Assumptions \hypref{HypPsiTransitions} and \Hypref{HypSolutionFlows} the optimal value of the two mean field games is the same.

\end{proof}
%%%%%%%%%%%%%%%%%%%%%%%%%%%%%%%%%%%%%%%%%%%%%%%%%%%%%%%%%%%
%%%%%%%%%%%%%%%%%%%%%%%%%%%%%%%%%%%%%%%%%%%%%%%%%%%%%%%%%%%
%%%%%%%%%%%%%%%%%%%%%%%%%%%%%%%%%%%%%%%%%%%%%%%%%%%%%%%%%%%

\section{Approximate N-player Correlated Equilibria} \label{SectApproximate}

%%%%%%%%%%%%%%%%%%%%%%%%%%%%%%%%%%%%%%%%%%%%%%%%%%%%%%%%%%%
%%%%%%%%%%%%%%%%%%%%%%%%%%%%%%%%%%%%%%%%%%%%%%%%%%%%%%%%%%%
%%%%%%%%%%%%%%%%%%%%%%%%%%%%%%%%%%%%%%%%%%%%%%%%%%%%%%%%%%%
Here, we show how to construct approximate $N$-player correlated equilibria starting from a suitable solution of the MFG. We make the following additional assumptions on dynamics and costs:
\begin{hypenv}

	\item \label{HypPsiContinuity} \emph{Continuity} of $\Psi\colon [\![0,T-1]\!]\times \mathcal{X}\times \Gamma \times \mathcal{Z} \to \mathcal{X}:$  
	
	\begin{itemize}
		\item[1)] For every $(t,x,\gamma) \in [\![0,T-1]\!] \times \mathcal{X}\times \Gamma $  and for  all $m, \widetilde{m} \in \mathcal{P(X)}$,
		\[
			\nu\left( \left\{ z\colon \Psi(t,x,m,\gamma,z)\neq \Psi(t,x,\widetilde{m},\gamma,z)\right\} \right)\leq w(\text{dist}(m, \widetilde{m})),
		\]
		where $w\colon [0, +\infty) \to [0,1]$ is some non-decreasing function with $\lim_{s \to 0^+}w(s) = 0$.
		%{\color{ofi_c} [Should we say that this is a modulus of continuity?]}
		
		\item[2)] For any $t \in [\![0,T-1]\!]$, $\Psi(t,\cdot)$ is $\tau \otimes \nu$-almost everywhere continuous, for every $\tau \in \mathcal{P(X \times  P(X)} \times \Gamma)$.
	\end{itemize}
	
	\item \label{HypCosts} The functions $f$ and $ F$, running cost and terminal cost,  are Lipschitz continuous with the same Lipschitz constant $L$.
\end{hypenv}

For an illustration of the continuity assumption \hypref{HypPsiContinuity} on the dynamics, see Remark~6.1 in \cite{CF2022}. Assumption~\hypref{HypCosts} is rather standard; in our finite setting, it is a true restriction only with respect to the measure argument of $f$ and $F$. 

The correlated suggestion $\rho$ we start with must satisfy, in addition to \Hypref{HypSolutionFlows}, the following condition on its information structure:

\begin{Hypenv}
	\item \label{HypSolutionCondInd}
	If $(\Phi,(\mu_t)_{t=0}^T)$ is distributed according to $\rho$, then there exist $\alpha_t:[0,1]\times \mathcal{P}(X)^{t+2}\to\mathcal{E}$, $t \in [\![0,T-1]\!]$, Borel-measurable functions and a uniformly distributed random variable $Z\overset{d}{\sim} \nu$, independent of $\mu$ s.t. $\Phi(t, \cdot):=\alpha_t(Z, \mu^{(t+1)})(\cdot)$, for all $t \in [\![0,T-1]\!]$. 
\end{Hypenv}

\begin{remark}\label{Rem_dec_rho}
	If $\rho$ satisfies \Hypref{HypSolutionCondInd}, then it admits a decomposition of the form
	\begin{align*}
		\rho(C_0\times \dots \times C_{T-1}\times B)&
		=\int_{B}\rho_1(C_0\times \dots \times C_{T-1} | m)\rho_2(dm)\\
		& =\int_{B}\int_{\mathcal{Z}}\otimes_{t=0}^{T-1} \delta_{\alpha_t(z,m^{(t+1)})}(C_t)\nu(dz)\rho_2(dm),
	\end{align*}
	for any $C_t\in \mathcal{B}(\mathcal{E}),$ $t \in [\![0,T-1]\!]$ and  $ B \in \otimes^{T+1}\mathcal{B}(\mathcal{P}(\mathcal{X}))$ and where, for $t \in [\![0,T-1]\!]$, $\alpha_t: [0,1]\times \mathcal{P(X)}^{t+2} \to \mathcal{E}$ are Borel functions.
	
	Finally let us notice that, if $(\Phi,(\mu_t)_{t=0}^T)$ is distributed according to $\rho$ that satisfies \textbf{(R2)}, then, for each $t \in [\![0,T-1]\!]$, $\Phi^{(t)}$ and $\mu$ are conditionally independent given $\mu^{(t+1)}$.
	The example presented in Section \ref{SectMFGExample} seems to suggest that the two conditions are equivalent.	
\end{remark}

%%%%%%%%%%%%%%%%%%%%%%%%%%%%%%%%%%%%%%%%%%%%%%%%%%%%%%%%%%%
%%%%%%%%%%%%%%%%%%%%%%%%%%%%%%%%%%%%%%%%%%%%%%%%%%%%%%%%%%%
%%%%%%%%%%%%%%%%%%%%%%%%%%%%%%%%%%%%%%%%%%%%%%%%%%%%%%%%%%%

\begin{theorem}\label{Thm_approx_eq}

	Let $m_0 \in \mathcal{P(X)}$, and suppose \hyprefall\ hold.
	Let $\rho \in \mathcal{P}(\mathcal{{R}}\times \mathcal{P}(\mathcal{X})^{T+1})$ be a correlated solution of the mean field game starting at $m_0$ and satisfying \Hyprefall.
	For $N \in \mathbb{N}$, define $\gamma^N \in \mathcal{P}(\mathcal{R}^N)$ by
	\[ 
		\gamma^N(C_1 \times \dots \times C_N):= \int_{\mathcal{P(X)}^{T+1}}\prod_{j=1}^N \rho_1(C_j | m)\rho_2(dm).
	\]
	Then, for all $N \in \mathbb N$, $\gamma^N $ is an $\varepsilon_N$-correlated equilibrium for the $N$-player game with initial distribution $m_0^{\otimes N}$ and the sequence $\{ \varepsilon_N\}_{N \in \mathbb{N}} \subseteq [0, +\infty)$ is such that $\lim_{N \to \infty}\varepsilon_N = 0.$
\end{theorem} 

\begin{remark}\label{Rem_dec_rho_bis}
	Let  $(Z_j)_{j=1}^N$ be i.i.d. r.v.s distributed according to $\nu$, also independent of $\mu$, and define, for $j \in [\![1,N]\!]$, $\Phi_j$ through $\Phi_j(t, \cdot):=\alpha_t(Z_j, \mu^{(t+1)})(\cdot)$, $t \in [\![0,T-1]\!]$, with $\alpha$ as in \textbf{(R2)}. 
	Then we have $\mathbb{P}\circ(\Phi_1, \dots,\Phi_N)^{-1}=\gamma^N:=\int_{\mathcal{P(X)}^{T+1}}\prod_{j=1}^N \rho_1(\cdot| m)\rho_2(dm)$.
\end{remark}

\begin{proof}
	We prove the result only for strategy modifications  of the first player. 
	Then, the general result is a consequence of the symmetry in the problem. 
	With a small abuse of notation, we simply write $\iota$ for $\iota_\gamma$, when its clear from the context the distribution that it refers to. We also use this same symbol $\iota$ for both the $N$-player game and the mean field game.
	Consider the correlated suggestion $\gamma^N \in \mathcal{P}(\mathcal{R}^N)$ defined in the statement of the theorem. 
	For each $N \in \mathbb{N}$, $\gamma^N$ is  an $\varepsilon_N$-correlated equilibrium for the initial distribution $m_0^{\otimes N}$, once the sequence $\{\varepsilon_N\}_{N \in \mathbb{N}}$ is defined as
	\[
		\varepsilon_N
		:=J_1^N(m_0^{\otimes N}, {\gamma}^N,\iota )
-\inf_{\widetilde \beta^N \in \mathcal N_{\gamma^N_1}}J_1^N(m_0^{\otimes N}, \gamma^N, \widetilde \beta^N), 		\qquad \text{ for all } N \in \mathbb{N}.
	\]
	By definition of infimum, it is possible to find a sequence of strategy modifications, $\{ \widetilde \gamma ^N \}_{N \in \mathbb{N}} \subseteq \mathcal N_{\gamma^N_1}$, such that
	\begin{equation}\label{u_N}
		J_1^N(m_0^{\otimes N}, \gamma^N, \widetilde \gamma^N)\leq \inf_{\widetilde \beta^N \in \mathcal N_{\gamma^N_1}}J_1^N(m_0^{\otimes N},\gamma^N, \widetilde \beta ^N)+\frac{1}{N}, \quad N \in \mathbb{N}.
	\end{equation}
	Thence, to complete the proof of the theorem, so showing that $\varepsilon_N \overset{N \to \infty}{\longrightarrow} 0$, it suffices to prove the following:
	\begin{equation}\label{1)}
		\lim_{N\to \infty} J_1^N(m_0^{\otimes N}, \gamma^N, \iota_{\gamma^N_1})
		=J(m_0, \rho, \iota),
	\end{equation}
	\begin{equation}\label{2)}
		\liminf_{N\to \infty} J_1^N(m_0^{\otimes N},\gamma^N, \widetilde \gamma ^N ) \geq J(m_0,  \rho, \iota).
	\end{equation}

	\renewcommand\qedsymbol{}	
	
 	\begin{proof}[Proof of \eqref{1)}]
		First of all, let us notice that the following equation holds
		\begin{equation}\label{dis_JN}
			J_1^N(m_0^{\otimes N}, {\gamma}^N, \iota_{\gamma^N_1} )
			= \int_{\mathcal{P}(\mathcal{X})^{T+1}} 		J_1^N(m_0^{\otimes N}, {\gamma}^N_m,\iota_{\gamma_{m,1}^N}) \rho_2 (dm) ,
		\end{equation}
		where, for each $N \in \mathbb{N}$ and for each $m \in \mathcal{P}(\mathcal{X})^{T+1}$, $\gamma_m^N:= \otimes ^N \rho_1(\cdot| m)$. 
		In fact, it holds
		$
		\gamma^N_1
		=\gamma^N \circ \pi_1^{-1}
		=\left(\int_{\mathcal{P}(\mathcal{X})^{T+1}} \rho_1(\cdot|m)^{\otimes N}\rho_2(dm) \right) \circ \pi_1^{-1}
		=\int_{\mathcal{P}(\mathcal{X})^{T+1}}  \rho_1(\cdot|m)\rho_2(dm),
		$ and 
		$
		\gamma^N_{m,1}
		=\gamma^N_m \circ \pi_1^{-1}
		= \rho_1(\cdot|m)^{\otimes N} \circ \pi_1^{-1}
		=  \rho_1(\cdot|m).
		$
		\\
		Indeed, thanks to the particular structure of the cost functional and the fact that $\iota_{\gamma^N_1}(d\phi, du)=\delta_{u}(d\phi)\gamma^N_1(du)$,
		we write
		\[
			\begin{split}
				J_1^N&(m_0^{\otimes N}, {\gamma}^N, \iota_{\gamma^N_1})\\&
				%1
				=\int_{\mathcal{X}^N}\int_{\mathcal{Z}^{NT}}\int_{ {\mathcal{R}}^{N}} \int_{ {\mathcal R}}
G^N(x,\phi,(u_j)_{j=2}^N,z)
 \delta_{u_1}(d\phi){\gamma}^N(du_1,\ldots,du_N)\nu^{\otimes NT}(dz)m_0^{\otimes N}(dx) \\&
				%2
				=\int_{\mathcal{X}^N}\int_{\mathcal{Z}^{NT}}\int_{ {\mathcal{R}}^{N}} \int_{ {\mathcal R}}
G^N(x,(u_j)_{j=1}^N,z)
 {\gamma}^N(du_1,\ldots,du_N)\nu^{\otimes NT}(dz)m_0^{\otimes N}(dx) \\&
				%3
				%\]
				%\[
				%\begin{split}
				%3
				=\int_{\mathcal{X}^N}\int_{\mathcal{Z}^{NT}}\int_{\mathcal{P(X)}^{T+1}}\int_{ {\mathcal{R}}^{N}} 
G^N(x,(u_j)_{j=1}^N,z)
{\gamma}^N_m(du_1,\ldots,du_N)\rho_2(dm)\nu^{\otimes NT}(dz)m_0^{\otimes N}(dx) \\&
				%4
				=\int_{\mathcal{P(X)}^{T+1}} 
\Bigg(\int_{\mathcal{X}^N}\int_{\mathcal{Z}^{NT}}\int_{{\mathcal{R}}^{N}} 
G^N(x,(u_j)_{j=1}^N,z)
{\gamma}^N_m(du_1,\ldots,du_N)\nu^{\otimes NT}(dz)m_0^{\otimes N}(dx) \Bigg) \rho_2(dm)\\&
				%5
				=\int_{\mathcal{P(X)}^{T+1}} 
\Bigg(\int_{\mathcal{X}^N}\int_{\mathcal{Z}^{NT}}\int_{ {\mathcal{R}}^{N+1}}
G^N(x,\phi,(u_j)_{j=2}^N,z)
\delta_{u_1}(d\phi){\gamma}^N_m(du)\nu^{\otimes NT}(dz)m_0^{\otimes N}(dx)\Bigg)\rho_2(dm)\\&
				%6
				= \int_{\mathcal{P}(\mathcal{X})^{T+1}} 
J_1^N(m_0^{\otimes N},{\gamma}^N_m, \iota_{\gamma^N_{m,1}}) \rho_2 (dm).
			\end{split}
		\]
		Notice that, here we have implicitly exploited the conditional independence and independence properties proved in Remark \ref{IInd_cond}, points  ii) and iii).
		Indeed, assume $\mathbb{P}_N\circ (\widetilde \Phi_1,\Phi_1 , \ldots,\Phi_N)^{-1}=\lambda$ and denote with $\lambda_{i}$ the measure projected on it $i^{th}$ component(s), that is $\lambda_{i}:= \lambda \circ (\pi_i)^{-1}$. 
		Let $A, B_1, \ldots, B_N \in \mathcal{B}(\widehat{\mathcal{R}}_N)$. 
		Exploiting Remark  \ref{IInd_cond} (iii), we get
		\[
			\begin{split}
				\mathbb P_N & (\widetilde \Phi_1 \in A,\Phi_1 \in B_1, \ldots,\Phi_N \in B_N)\\
				& = \int_{ \widehat{\mathcal{R}}_N}\boldsymbol{1}_{B_1}(\phi_1)\int_{\widehat{\mathcal{R}}_N \times  {\widehat{\mathcal{R}}_N}^{N-1}}\boldsymbol{1}_{A}(d\psi)\prod_{j=2}^N \boldsymbol{1}_{B_j}(\phi_j)\lambda_{1,3, \dots, N+1}(d\psi,d\phi_2, \ldots,d\phi_N | \phi_1)\lambda_2(d\phi_1)\\&
				= \int_{ \widehat{\mathcal{R}}_N} \boldsymbol{1}_{B_1}(\phi_1)\int_{ \widehat{\mathcal{R}}_N}\boldsymbol{1}_{A}(d\psi) \lambda_1(d\psi | \phi_1) \int_{ {\widehat{\mathcal{R}}_N}^{N-1}}\prod_{j=2}^N \boldsymbol{1}_{B_j}(\phi_j)\lambda_{3,\ldots,N+1}(d\phi_2, \ldots,d\phi_N | \phi_1)\lambda_2(d\phi_1)\\&
				= \int_{ \widehat{\mathcal{R}}_N \times  {\widehat{\mathcal{R}}_N}^{N}}\boldsymbol{1}_{A}(d\psi) \prod_{j=1}^N \boldsymbol{1}_{B_j}(\phi_j)\lambda_1(d\psi | \phi_1)\lambda_{3,\ldots,N+1}(d\phi_2, \ldots,d\phi_N | \phi_1)\lambda_2(d\phi_1)\\&
				= \int_{\widehat{\mathcal{R}}_N \times  {\widehat{\mathcal{R}}_N}^{N}}\boldsymbol{1}_{A}(d\psi) \prod_{j=1}^N \boldsymbol{1}_{B_j}(\phi_j)\delta_{\phi_1}(d\psi)\gamma^N(d\phi_1, \ldots ,d\phi_N),
			\end{split}
		\]
		where the last step is a consequence of the fact that $\mathbb{P}_N\circ (\Phi_1, \ldots, \Phi_N)^{-1}=\gamma^N$ and $\mathbb{P}_N\circ (\widetilde \Phi_1, \Phi_1)^{-1}=\iota_{\gamma^N_i}$.
		Thus, we have $\lambda(d\psi,d\phi_1,\ldots,d\phi_N)=\delta_{\phi_1}(d\psi)\gamma^N(d\phi_1, \ldots ,d\phi_N)$.
		\\ 

		Furthermore, for the mean field  game, we have
		\begin{equation}\label{dis_J}
			J(m_0,{\rho}, \iota)
			= \int_{\mathcal{P}(\mathcal{X})^{T+1}} J(m_0,  \rho_1(\cdot|m) \otimes \delta_m, \iota ) \rho_2 (dm),
		\end{equation}
		where $\rho_1(\cdot|m)  $ as in the statement of the theorem. In fact, with similar computations as above for $J_1^N$, we get
		\[
			\begin{split}
				J(m_0, {\rho}, \iota)
				%=\int_{\mathcal{X}}\int_{\mathcal{Z}^{T}}\int_{{\mathcal{R}}\times \widehat{\mathcal{R}}\times  \mathcal P (\mathcal X)^{T+1}} G(x,\psi(\cdot,\cdot,m),z,m)\iota_{\rho}(d\psi,d \phi,dm)\nu^{\otimes T}(dz)m_0(dx) \\&
				%=\int_{\mathcal{X}}\int_{\mathcal{Z}^{T}}\int_{  \widehat{\mathcal{R}}\times \widehat{\mathcal{R}}\times  \mathcal P (\mathcal X)^{T+1}} G(x,\psi(\cdot,\cdot,m),z,m)\delta_\phi(d\psi)\rho(d \phi,dm)\nu^{\otimes T}(dz)m_0(dx) \\&
				%1
				&=\int_{\mathcal{X}}\int_{\mathcal{Z}^{T}}\int_{{\mathcal{R}}\times  \mathcal P (\mathcal X)^{T+1}} G_\iota(x,\phi,z,m)
\rho(d \phi,dm)\nu^{\otimes T}(dz)m_0(dx) \\&
				%2
				=\int_{\mathcal{X}}\int_{\mathcal{Z}^{T}}\int_{ \mathcal P (\mathcal X)^{T+1}} \int_{ {\mathcal{R}}}
G_\iota(x,\phi,z,m)\rho_1(d \phi|m)\rho_2(dm)\nu^{\otimes T}(dz)m_0(dx) \\&
				%3
				=\int_{ \mathcal P (\mathcal X)^{T+1}} \int_{\mathcal{X}}\int_{\mathcal{Z}^{T}}  \int_{ {\mathcal{R}}}
G_\iota(x,\phi,z,m)
\rho_1(d \phi|m)\nu^{\otimes T}(dz)m_0(dx) \rho_2(dm)\\&
				%=\int_{ \mathcal P (\mathcal X)^{T+1}}\int_{\mathcal{X}}\int_{\mathcal{Z}^{T}}  \int_{\widehat{\mathcal{R}}\times \widehat{\mathcal{R}}\times \mathcal P (\mathcal X)^{T+1}}G(x,(\psi(\cdot,\cdot,l),z,l)\delta_{\phi}(d\psi)\rho_1(d \phi|l)\delta_m(dl)\nu^{\otimes T}(dz)m_0(dx) \rho_2(dm)\\&
				%=\int_{ \mathcal P (\mathcal X)^{T+1}}\int_{\mathcal{X}}\int_{\mathcal{Z}^{T}}\int_{\widehat{\mathcal{R}}\times \widehat{\mathcal{R}}\times \mathcal P (\mathcal X)^{T+1}}G(x,(\psi(\cdot,\cdot,l),z,l)\iota_{\rho_1(\cdot|m)\otimes \delta_m}(d\psi,d\phi, dm)\nu^{\otimes T}(dz)m_0(dx) \rho_2(dm)\\&
				%4
				=\int_{ \mathcal P (\mathcal X)^{T+1}} J(m_0,\rho_1(\cdot|m)\otimes \delta_m, \iota ) \rho_2(dm),
			\end{split}
		\]
		and this ends the proof of the identity. 

		In the proof of (\ref{1)}), that is the case in which all the players follow the mediator's suggestion, computations simplify considerably.
		Indeed, since the recommendation $\gamma^N$ belongs to $\mathcal{P}(\mathcal{R}^N)$, for any $N \in \mathbb{N}$, we can proceed as in the proof of \cite[Theorem 5.1 and Theorem 6.1]{CF2022}, that is through the following three steps:
		\begin{itemize}
			\item[{1.}] We show that, for any fixed $m \in \PX$, there exists a  subsequence of indeces such that
			\[
				\lim_{k \to \infty} J_1^{N_k}(m_0^{\otimes N_k},\gamma_m^{N_k},\iota)=J(m_0,\rho_m, \iota_{\rho_m}),
			\]
			for some $\rho_m \in \mathcal P(\mathcal{R}\times \PX)$, with $\gamma_m^N=  \rho_1(\cdot|m)^{\otimes N}.$

			\item[{2.}] We prove a result of chaos propagation that enables us to deduce that, in the limit, for all $m \in \PX$, we have
			\[
				\mathbb P_m \circ (X^m_t,\mu^m_t)^{-1}=\widehat{m}^m_t \otimes \delta_{\widehat{m}^m_t }, \quad \text{ for all } t \in [\![0, T]\!],
			\]
			for some $\widehat{m}^m_t \in \mathcal{P(X)}$.

			\item[{3.}] We show that, for $\rho_2$-almost every $m \in \PX$, $(\widehat{m}^m_t)_{t=0}^T= (m_t)_{t=0}^T$, independently of the convergent subsequence considered,  and conclude by integrating in $(m_t)_{t=0}^T \in \PX$ w.r.t. $\rho_2(dm)$.
		\end{itemize}

		\subsubsection*{Step 1}
		Fix a flow of measure $m \in \PX $. 
		Consider the sequence of triples  $\{(m_0^{\otimes N},{\gamma}_m^N, \iota)\}_{N \in \mathbb{N}}$.
		For each $N \in \mathbb{N}$, consider  the  tuple $((\Omega_{N,m}, \mathcal{F}_{N,m},\mathbb{P}_{N,m}),(\Phi_j^{N,m})_{j=1}^N,(\vartheta^{N,m}_t)_{t=0}^{T-1} ,(\xi^{1,N,m}_t,\dots,\xi_t^{N,N,m})_{t=1}^T$, $\widetilde{\Phi}_1^{N,m}$,
		$(X_t^{1,N,m},\dots,X_t^{N,N,m})_{t=0}^T)$, a realization of $(m_0^{\otimes N},{\gamma}_m^N, \iota)$. Since we are proving  (\ref{1)}), w.l.o.g. we assume that  $\Phi_1^{N,m}=\widetilde \Phi_1^{N,m}$, $\mathbb P_{N,m}$-a.s.
		Set, for any $N \in \mathbb{N}$,
		\[
			\eta^N_m:=\mathbb{P}_{N,m} \circ (\Phi_1^{N,m}, (\mu_t^{1,N,m})_{t=0}^T,(\xi_t^{1,N,m})_{t=1}^T,\widetilde \Phi_1^{N,m}, (X_t^{1,N,m})_{t=0}^T)^{-1}.
		\]
		Since, for any $N \in \mathbb{N}$, $\eta^N_m$ belongs to the compact set $\mathcal{P}(\mathcal{R}\times \PX \times \mathcal{Z}^T \times \mathcal{R} \times \mathcal{X}^{T+1})$, the sequence $\{ \eta^N_m \}_{N \in \mathbb{N}}$ admits a convergent subsequence, $\{ \eta^{N_k}_m \}_{k \in \mathbb{N}}$, with limit $\eta_m$.
		On a suitable probability space, $(\Omega_m, \mathcal{F}_m, \mathbb{P}_m)$, we consider a $\mathcal{R}\times \PX \times \mathcal{Z}^T \times \mathcal{R} \times \mathcal{X}^{T+1}$- valued random vector,  $(\Phi^m, (\mu_t^m)_{t=0}^T, (\xi_t^m)_{t=1}^T, \widetilde  \Phi^m, (X^m_t)_{t=0}^T),$  such that 
		\begin{equation}
			\eta_m
			:=\mathbb{P}_m \circ (\Phi^m, (\mu_t^m)_{t=0}^T, (\xi_t^m)_{t=1}^T, \widetilde  \Phi^m, (X^m_t)_{t=0}^T)^{-1}
		\end{equation}
		and set
		\begin{equation}
			\rho_m
			:=\mathbb{P}_m \circ (\Phi^m, (\mu^m_t)_{t=0}^T)^{-1} \in \mathcal{P}(\mathcal{R} \times \PX),
		\end{equation}
		\begin{equation}
			\beta_m
			:=\mathbb{P}_m \circ (\widetilde \Phi^m, \Phi^m, (\mu^m_t)_{t=0}^T)^{-1} \in \mathcal{P}(\mathcal{R} \times\mathcal{R}\times \PX).
		\end{equation}
		Then, the  \emph{limit variables}, $(\Phi^m, (\mu_t^m)_{t=0}^T, (\xi_t^m)_{t=1}^T, \widetilde  \Phi^m, (X^m_t)_{t=0}^T)$, satisfy the following properties:
		\begin{itemize}
			\item[\textbf{i)}] By the continuous mapping theorem and the fact that, by hypothesis, $X_0^{1,N,m}\overset{d}{\sim}m_0$, for all $N \in \mathbb{N}$, we get
			\[
				\mathbb{P}_m \circ (X^m_0)^{-1}=m_0.
			\]

			\item[\textbf{ii)}] $\rho_m=\mathbb{P}_m \circ (\Phi^m, (\mu^m_t)_{t=0}^T)^{-1} \in \mathcal{P}(\mathcal{R} \times \PX),$ by definition.

			\item[\textbf{iii)}] As a consequence of the independence of the variables $(\xi_t^{1,N,m})_{t=1}^T, X_0^{1,N,m}$,  and $(\Phi_1^{N,m},\widetilde \Phi_1^{N,m})$ and the fact that they jointly converge  in distribution, together with the continuous mapping theorem and the fact that $\xi_t^{1,N,m}\overset{d}{\sim} \nu$, $N \in \mathbb{N}$,  $t \in [\![1,T]\!]$, we have
			\[
				\xi^m_t\overset{d}{\sim} \nu, \qquad \text{ for any }t \in [\![1,T]\!].
			\]

			\item[\textbf{iv)}] Since, for any $N \in \mathbb{N}$, $\Phi_1^{N,m}=\widetilde\Phi_1^{N,m}$, $\mathbb{P}_{N,m}$-a.s., we get $\Phi^m=\widetilde \Phi^m$, $\mathbb{P}_m$-a.s. Then,  we have $\widetilde\Phi^m(t,x^{(t)},m^{(t)})=\Phi^m(t,x^{(t)},m^{(t)})=\iota_t(\Phi^m,x^{(t)},m^{(t)})$.\\
			Furthermore, since $(\Phi_j^{N_k,m})_{j =1}^{N_k}$, $\widetilde \Phi_1^{N_k,m}$, as well as $\Phi^m$ and $\widetilde \Phi^m$, are $\mathcal{R}$-valued variables, reasoning as in the Step 3 of the proof of \cite[Theorem 5.1]{CF2022}, we get that 
$(\xi^m_t)_{t=1}^T$, $X^m_0$ and $(\Phi^m,  (\mu^m_t)_{t=0}^T)$ are independent.

			\item[\textbf{v)}] Furthermore, proceeding as in the Step 3. of the proof of \cite[Theorem 5.1]{CF2022}, it is possible to prove that $(X^m_t)_{t=0}^T$ follows the dynamics:
  for any $t \in [\![0,T-1]\!],$
			\begin{equation}\label{din_Xm}
					\begin{split}
						X^m_{t+1}&=\Psi\left(t,X^m_t,\Phi^m\left(t,X^{m,(t)},\mu^{m,(t)}\right),\xi^m_{t+1}\right), \qquad \mathbb{P}_m\text{-a.s.}\\&
						=\Psi\left(t,X^m_t,\Phi^m\left(t,X^m_t \right),\xi^m_{t+1}\right).
					\end{split}
				\end{equation}
			
		\end{itemize}
		
		These features correspond to properties \textbf{ i)-v)} in Definition \ref{mf_real}.
		We have  proved  that the tuple $((\Omega_m, \mathcal{F}_m, \mathbb{P}_m),\Phi^m, (\mu_t^m)_{t=0}^T, X^m_0, (\xi_t^m)_{t=1}^T, \widetilde  \Phi^m, (X^m_t)_{t=0}^T)$ is a realization of the triple $(m_0, \rho_m, \iota)$. 
		\\
		Furthermore, since, for any $N \in \mathbb N$, $\mathbb P_{N,m} \circ (\Phi_1^{N,m})^{-1}=\rho_1(\cdot|m)$, we get $\mathbb P_m \circ (\Phi^m)^{-1}=\rho_1(\cdot|m)$.\\
		Furthermore, we have 
		\begin{equation}\label{limit_id}
			\lim_{k \to \infty}J_1^{N_k}(m_0^{\otimes N_k},{\gamma}_m^{N_k},\iota)=J(m_0, {\rho}_m, \iota).
		\end{equation}
		Equation (\ref{limit_id}) follows from the joint convergence in distribution of the variables that form a realization together with hypothesis \hypref{HypCosts} and the dominated convergence theorem. Notice that, here, the fact that $\Phi_1^{N,m}$, as well as $\Phi_m$, are $\mathcal R$-valued is crucial.		
		
		\subsubsection*{Step 2}
		The symmetry and independence among the players in the prelimit game  enable us to prove a result of chaos propagation for the convergent subsequence associated to $(m_0^{\otimes N},\gamma_m^N, \iota)$.
		We are not going to show this property directly but exploiting an equivalent characterization of  propagation of chaos, namely  \cite[Theorem 4.2]{Gottlieb}(see Theorem \ref{Gott}, in the Appendix).
		\\
		In fact, we can work iteratively to show that chaos propagates from $t$ to $t+1$ for each $t \in [\![0, T-1]\!].$ 
		This fact implies
		\begin{equation}\label{chaos_p} 
			\mathbb{P}_m\circ (X^m_t, \mu^m_t)^{-1}= \widehat{m}^m_t \otimes \delta_{\widehat{m}^m_t}, \qquad t \in[\![0, T]\!].
		\end{equation}
		for some deterministic flow of measures $(\widehat{m}^m_t)_{t=0}^T \in \mathcal{P(X)}^{T+1}$, with $\widehat{m}^m_0={m}_0$.
		\\
		We show  that propagation of chaos holds, for our specific structure, for the first time step. 
		This same reasoning can be immediately extended to the other time steps implying our thesis.  
		We notice that this is possible only because the variables $\{\Phi^{N,m}_j\}_{j=1}^N$ take values \hbox{in $\mathcal{R}$}.
		For a detailed proof see Appendix \ref{SectAppendix}.	
		\\	
		
		Reframing the result (\ref{chaos_p}) of chaos propagation  in the dynamics described in Equation (\ref{din_Xm}), we get, $\mathbb{P}_m$-a.s.,
		\begin{equation}\label{eqq}
			\left \{ 
				\begin{array}{ll}
					X^m_{t+1}
					=\Psi\left(t,X^m_t,\widehat m^m_t,\Phi^m\left(t,X^m_t\right),\xi^m_{t+1}\right),
					\\
					\mathbb{P}_m \circ (X^m_t)^{-1}
					= \widehat{m}^m_t, \qquad t \in[\![0, T]\!].
				\end{array}
			\right.
		\end{equation}
		Notice that, in the variable $\Phi^m$, we are omitting the dependence on the measure $\widehat{m}$. We are allowed to do this because this variable takes values in $\mathcal{R}$, being distributed according to $\rho_1(\cdot|m)$.
		The system in (\ref{eqq}) has a unique solution.
		It is a consequence of the iterative definition of the process $(X^m_t)_{t=0}^T$ and of properties  \textbf{i)-v)} of the limit realization. 
		Thence, the flow of measures $(\widehat m^m_t)_{t=0}^T \in \PX$, corresponding to this system, is uniquely determined for each $(m_t)_{t=0}^T \in \PX.$

		\subsubsection*{Step 3}

		Now, our aim is to prove that $(\widehat m^m_t)_{t=0}^T = (m_t)_{t=0}^T$, for $\rho_2$-almost all $(m_t)_{t=0}^T \in \PX$.
		Let $\rho$ be the correlated solution for the mean field game starting at $m_0$, as  in the statement of the theorem, and consider a realization of $(m_0,{\rho}, \iota)$, i.e.  $\big((\Omega^*, \mathcal{F}^*,\mathbb{P}^*),\Phi^*,(\mu^*_t)_{t=0}^T,X^*_0,(\xi^*_t)_{t=1}^T, \iota, (X_t^*)_{t=0}^T \big)$. 
		By definition of realization, such a tuple satisfies properties \textbf{i)-v)} in Definition \ref{mf_real}. 
		In particular, without loss of generality, we set
		\begin{itemize}
			\item[\textbf{i)}] $\mathbb{P}^*\circ (X^*_0)^{-1}= m_0;$

			\item[\textbf{ii)}] $\mathbb{P}^*\circ (\Phi^*, (\mu^*_t)_{t=0}^T)^{-1}=\rho$;

			\item[\textbf{iii)}] $(\xi^*_t)_{t=1}^T$ i.i.d. with $\xi^*_t\overset{d}{\sim}\nu$;
	
			\item[\textbf{iv)}] $(\xi^*_t)_{t=1}^T, X^*_0$ and $(\Phi^*, (\mu^*_t)_{t=0}^T)$ independent;
 
			\item[\textbf{iv')}] $ \iota(\Phi^*)=\Phi^*,$ $ \mathbb{P}^*$-a.s.;

			\item[\textbf{v)}] for all $t \in [\![0,T-1]\!]$,
			\[
				X^*_{t+1}
				=\Psi\left(t,X^*_{t},\mu^*_{t},\iota \circ\Phi^*\left(t,X^*,\mu^*\right),\xi^*_{t+1}\right)
				=\Psi\left(t,X^*_{t},\mu^*_{t},\Phi^*\left(t,X^*_t\right),\xi^*_{t+1}\right), \quad \mathbb{P}^*\text{-a.s.}
			\]
		\end{itemize}
		The fact that $\rho$ is a correlated solution for the mean field game (consistency condition) and the definition of $\rho_1(\cdot|m)$ imply, respectively, that, for $\rho_2$-almost all $m \in \PX$, we have:
		\begin{itemize}
			\item $\mathbb{P}^*\left( X^*_{t}\in \cdot \big|(\mu^*_t)_{t=0}^T= (m_t)_{t=0}^T\right)=m_t,\qquad t \in [\![0,T]\!];$

			\item $\mathbb{P}^*\left( \Phi^*\in \cdot \big|(\mu^*_t)_{t=0}^T= (m_t)_{t=0}^T\right)=\rho_1(\cdot|m)$;
		
		\end{itemize}
		Then, setting $\mathbb{Q}^m(\cdot)=\mathbb{P}^*(\cdot|(\mu^*_t)_{t=0}^T= (m_t)_{t=0}^T)$, we get:
		\begin{itemize}
			\item $\mathbb{Q}^m \circ (X^*_{t})^{-1}=m_t,\qquad  t \in [\![0,T]\!];$

			\item $\mathbb{Q}^m \circ(\Phi^*)^{-1}=\rho_1(\cdot|m);$

			\item $\mathbb{Q}^m \circ (\xi^*_t)^{-1}=\nu,\qquad  t \in [\![1,T]\!];$ 
		\end{itemize}
	where the last item is a consequence of the fact that $(\xi^*_t)_{t=1}^T$ is jointly independent of $(\mu^*_t)_{t=0}^T$, by property \textbf{iv)} above.
	Hence, for $\rho_2$-almost all $(m_t)_{t=0}^T \in \PX$, $\mathbb{Q}^m$-almost surely, for any $t \in [\![0,T-1]\!],$ we have
		\begin{equation}\label{eqqq}
			\left \{ 
				\begin{array}{ll}
					X^*_{t+1}
					=\Psi\left(t,X^*_{t}, m_t, \Phi^*\left(t,X^*_{t}\right),\xi^*_{t+1}\right)
					\\
					\mathbb{Q}^m \circ ({X}^*_t)^{-1}
					= m_t, \qquad t \in[\![0, T]\!].
				\end{array}
			\right.
		\end{equation}  
		This means that the tuple $\big((\Omega^*, \mathcal{F}^*,\mathbb{Q}^m),\Phi^*,(\mu^*_t)_{t=0}^T,X^*_0,(\xi^*_t)_{t=1}^T, (X_t^*)_{t=0}^T \big)$ is a solution for the system (\ref{eqq}).
		Finally, exploiting the uniqueness of solution for this  			system, we obtain the following identities,	that hold  for $\rho_2$-almost all $(m_t)_{t=0}^T \in \PX$:
		\[
			(\widehat m_t^m)_{t=0}^T=(m_t)_{t=0}^T, 
			\qquad 
			\rho_m=\mathbb{P}_m\circ (\Phi^m, \mu_m)^{-1}= \mathbb Q^m \circ (\Phi^*, m)^{-1}=\rho_1(\cdot|m) \otimes \delta_m.
		\]
		Notice that the second equation is a consequence of the fact that $\mathbb{P}_m \circ (\Phi^m)^{-1}=\rho_1(\cdot|m).$
		In particular, we rewrite the equation in \eqref{limit_id} as
		\begin{equation}
			\lim_{k \to \infty}J_1^{N_k}(m_0^{\otimes N_k},  \gamma_m^{N_k}, \iota)
			=J(m_0, \rho_1(\cdot|m) \otimes \delta_m, \iota).
		\end{equation}
		Notice that the limit above does not depend on the subsequence considered and so we can deduce that the whole sequence converges to this limit.	
		Now, an application of the dominated convergence theorem, together with the identities  (\ref{dis_JN}) and (\ref{dis_J}), yields
		\[
			\begin{split}
				\lim_{N\to \infty}J_1^N(m_0^{\otimes N}, \gamma^{N}, \iota)&
				=\lim_{N\to \infty}\int_{\PX}J_1^N(m_0^{\otimes N}, \gamma_m^{N},\iota)\rho_2(dm)\\&
				=\int_{\PX}\lim_{N\to \infty} J_1^N(m_0^{\otimes N}, \gamma_m^{N},\iota)\rho_2(dm)\\&
				=\int_{\PX}J(m_0,  \rho_1(\cdot|m) \otimes \delta_m,\iota )\rho_2(dm)
				=J(m_0, \rho, \iota).
			\end{split}
		\]
		This ends the proof of (\ref{1)}).
	
	\end{proof}
	
%%%%%%%%%%%%%%%%%%%%%%%%%%%%%%%%%%%%%%%%%%%%%%%%%%%%%%%%%%%
%%%%%%%%%%%%%%%%%%%%%%%%%%%%%%%%%%%%%%%%%%%%%%%%%%%%%%%%%%%
%%%%%%%%%%%%%%%%%%%%%%%%%%%%%%%%%%%%%%%%%%%%%%%%%%%%%%%%%%%

	\begin{proof}[Proof of \eqref{2)}]

		Consider the minimizing sequence of  strategy modifications $\{\widetilde \gamma^N\}_{N \in \mathbb N}\subseteq \mathcal N_{\gamma^N_1} $, defined in \eqref{u_N}. Now, set 
		\begin{equation}
		\begin{split}
		\overline{\overline{\gamma}}_N(d\varphi_1,\dots,d\varphi_N,dm)
		:&=\left(\bigotimes_{j=1}^{N} \rho_1(d\varphi_j|m) \right) \rho_2(dm)\\
		&=\left(\bigotimes_{j=1}^{N} \left( \int_{\mathcal{Z}}\bigotimes_{t=0}^{T-1}\delta_{\alpha_t(z_j,m^{(t+1)})}(d\varphi_j(t, \cdot))\nu(dz_j)\right) \right) \rho_2(dm)\\
		&=\gamma_m^N(d\varphi_1,\dots,d\varphi_N)\rho_2(dm) \in \mathcal{P}(\mathcal{R}^N\times \PX ),
		\end{split}
		\end{equation}
		where $\alpha_t$ has been chosen according to \textbf{(R2)}, see Remark \ref{Rem_dec_rho} and \ref{Rem_dec_rho_bis}.
		The peculiar form of the starting MFG solution $\rho$ that satisfies assumption \Hypref{HypSolutionCondInd}, and the consequent form of the correlated suggestion in the $N$-player game, will be crucial to give an interpretation to any $N$-player game realization in the mean-field sense.
		Now, we want to build a sequence  of realizations of 
 $\{(m_0^{\otimes N},\gamma^N, \widetilde \gamma^N)\}_{N \in \mathbb{N}}$.
 		For a fixed $N \in \mathbb{N}$, let $( \Omega_N,\mathcal{F}_N,\mathbb{P}_N)$ be  a complete probability space.
		On this probability space, we set:
		\begin{itemize}

			\item[\textbf{i)}]  $(X^j_0)_{j=1}^N$,  $\mathcal{X}$-valued random variables  i.i.d. according to $m_0$;
			\\
			$(\Phi_j)_{j=1}^N,$  $\mathcal R$-valued random variables, such that
			\[
			\Phi_j(t,\cdot):=\alpha_t(Z_j,\mu^{(t+1)}), \qquad j=1,\dots, N, \quad t =0,\dots, T-1,
			\]
			with $(Z_j)_{j=1}^N$ i.i.d. $\overset{d}{\sim}\nu$ and independent of $\mu\overset{d}{\sim}\rho_2$.\\
			In particular, this implies that $\mathbb P_N \circ ((\Phi_j)_{j=1}^N, \mu)^{-1}(d\phi_1,\ldots,d\phi_N,dm)
			=\overline{\overline{\gamma}}_N(d\phi_1,\ldots,d\phi_N,dm)$ and so that $\mathbb P_N\circ ((\Phi_j)_{j=1}^N)^{-1}(d\phi_1,\ldots,d\phi_N)
			=\gamma_N(d\phi_1,\ldots,d\phi_N)$;

			\item[\textbf{ii)}]  $(\xi^1_t,\dots,\xi^N_t)_{t=0}^T$,    $\mathcal{Z}$-valued random variables  i.i.d. all distributed according to $\nu;$

			\item[\textbf{iii)}]  $(\vartheta_t)_{t=0}^T$, $\mathcal{Z}$-valued random variables  i.i.d. all distributed according to $\nu;$

			\item[\textbf{iv)}]  $(\xi^1_t,\dots,\xi^N_t)_{t=0}^T$, $(X^j_0)_{j=1}^N$, $((\mu_t)_{t=0}^T, (Z_j)_{j=1}^N)$ and $(\vartheta_t)_{t=0}^{T-1}$ are independent;

			\item[\textbf{v)}]  $\widetilde \Upsilon_1^N$, $\UU$-valued random variable s.t. $\widetilde \Upsilon_1^N(t,\cdot)= w_t^N(\vartheta_t,\Phi_1)(\cdot)$, with $w_t^N: [0,1]\times \mathcal{R} \to \widehat{\mathcal{E}}_{t,N}$ Borel function, for any $t \in [\![0,T-1]\!]$, and $\mathbb{P}_N\circ (\Phi_1,\widetilde \Upsilon_1^N)^{-1}=\widetilde \gamma_N$\footnote{The existence of these Borel functions is a consequence of the measurability condition in \textbf{v)} in Definition \ref{Def_Real_N} and of Doob's Lemma (see \cite[Lemma 1.13]{Kallenberg}).}.
		\end{itemize}
		We set the following dynamics for the $\mathcal X$-valued processes, $(X^{j,N}_t)_{t=0}^T$, $j\in [\![1,N]\!]$, for $t \in [\![0,T-1]\!]$,
		\begin{equation}\label{Eq_X_j}
			\begin{split}
				X^{j,N}_{t+1}= \Psi\left(t,X^{j,N}_t, \mu^{j,N}_t,\Phi_j(t,X^j_t),\xi^j_{t+1}\right),\quad \mathbb{P}_N\text{-a.s.},
			\end{split}
		\end{equation}
		where, for all $t \in [\![0,T]\!]$ and $j\in [\![1,N]\!]$, $\mu^{j,N}_t:=\frac{1}{N-1}\sum_{k\neq j,k =1}^N \delta_{X^{k,N}_t}$ and $\mu^{j,N}:=(\mu^{j,N}_t)_{t=0}^T$.
		This corresponds to the case where all the players stick to the suggestion given by the mediator.
		\\
		Then, we define another sequence of processes, $(\widetilde X^{j,N}_t)_{t=0}^T$, $j\in [\![1,N]\!]$, setting $\widetilde X^{j,N}_0 : =X^j_0$ and, for $t \in [\![0,T-1]\!]$,
		\begin{align}\label{Eq_tilde_X_j}
		& \widetilde X^{j,N}_{t+1}= \Psi\left(t, \widetilde X^{j,N}_t, \widetilde\mu^{j,N}_t,\Phi_j(t,\widetilde X^{j,N}_t),\xi^j_{t+1}\right), \\ \nonumber
		& \widetilde X^{1,N}_{t+1}= \Psi\left(t,\widetilde X^{1,N}_t, \widetilde \mu^{1,N}_t, \widetilde \Upsilon^{1,N}(t,\widetilde X^{1,N}, \widetilde \mu^{1,N}),\xi^1_{t+1}\right), \quad \mathbb{P}_N\text{-a.s.},
		\end{align}
		where, for all $t \in [\![0,T]\!]$ and $j\in [\![1,N]\!]$, $\widetilde \mu^{j,N}_t:=\frac{1}{N-1}\sum_{k\neq j,k =1}^N \delta_{\widetilde X^{k,N}_t}$ and $\widetilde \mu^{j,N}:=(\widetilde \mu^{j,N}_t)_{t=0}^T$.
		This represents the case in which only the first player is deviating from the suggestion according to the minimizing sequence of strategy modifications introduced in the beginning of the proof.
		Hence, $(({\Omega}_N, {\mathcal F}_N,  {\mathbb P}_N),(\Phi_j)_{j=1}^N$,  $(\vartheta_t)_{t=0}^{T-1},  (\xi^1_t,\dots,\xi^N_t)_{t=1}^T, \Phi_1,(X^{1,N}_t,\dots,X^{N,N}_t)_{t=0}^T)$ and  $(({\Omega}_N, {\mathcal F}_N,  {\mathbb P}_N)$, $(\Phi_j)_{j=1}^N$, $ (\vartheta_t)_{t=0}^{T-1},  (\xi^1_t,\dots,\xi^N_t)_{t=1}^T$, $\widetilde \Upsilon_1^N,(\widetilde X^{1,N}_t,\dots,\widetilde X^{N,N}_t)_{t=0}^T)$  are, respectively, a realization of the triple $(m_0^{\otimes N}, \gamma^N, \iota)$ and  $(m_0^{\otimes N}, \gamma^N, \widetilde \gamma^N )$ for the first player.
		Indeed, we notice that, with this construction, also   condition \textbf{v)} in Definition \ref{Def_Real_N} is satisfied.
		\\
		Then, we define another sequence of processes, $( \widehat X^{j,N}_t)_{t=0}^T$, $j\in [\![2,N]\!]$, setting $\widehat X^{j,N}_0 : =X^j_0$ and, for $t \in [\![0,T-1]\!]$,
		\begin{align}\label{Eq_hat_X_j}
		& \widehat X^{j,N}_{t+1}= \Psi\left(t, \widehat X^{j,N}_t, \widehat\mu^{1,N}_t,\Phi_j(t,\widehat X^{j,N}_t),\xi^j_{t+1}\right),\quad \mathbb{P}_N\text{-a.s.},
		\end{align}
		where, for all $t \in [\![0,T]\!]$, $\widehat \mu^{1,N}_t:=\frac{1}{N-1}\sum_{k =2}^N \delta_{\widehat X^{k,N}_t}$ and $\widehat \mu^{1,N}:=(\widehat \mu^{1,N}_t)_{t=0}^T$.
		These processes describe the evolution of the system excluding the first player.
		\\
		Finally, we define the process, $( \overline X^{1,N}_t)_{t=0}^T$, setting $\overline X^{1,N}_0 : =X^1_0$ and, for $t \in [\![0,T-1]\!]$,
		\begin{align}\label{Eq_bar_X_j}
		& \overline X^{1,N}_{t+1}= \Psi\left(t, \overline X^{1,N}_t, \mu_t,\widetilde \Upsilon^{1,N}(t,\widetilde X^{1,N}, \widetilde \mu^{1,N}),\xi^1_{t+1}\right),\quad \mathbb{P}_N\text{-a.s.}.
		\end{align}
		This last one is an auxiliary process whose utility will be made clear in the following.
		From now on, for simplicity of notation, for $t \in [\![0,T-1]\!]$, we write $\widetilde u^{1,N}_t$ for $\widetilde \Upsilon^{1,N}(t,\widetilde X^{1,N}, \widetilde \mu^{1,N})$.
		\\
		First of all, we focus on $(({\Omega}_N, {\mathcal F}_N,  {\mathbb P}_N),\Phi_1, (\mu_t)_{t=0}^T, (\mu^{1,N}_t)_{t=0}^T, (\vartheta_t)_{t=0}^{T-1}, (\xi^{1}_t)_{t=1}^T,(X_t^{1,N})_{t=0}^T)$. 
		For all $t \in [\![0,T]\!]$, an application of the tower property yields 
		\[
		\mathbb{E}_{N}[\text{dist}(\mu^{1,N}_t,\mu_t)] 
		=\int_{\mathcal{P}(\mathcal{X})^{T+1}}\mathbb{E}_{N}[\text{dist}(\mu^{1,N}_t,\mu_t)|\mu=m] \rho_2(dm).
		\]
		Conditionally on the event $\{(\mu_t)_{t=0}^T=(m_t)_{t=0}^T\}$, we have already seen that $(\mu^{1,N}_t)_{t=0}^T$ converges weakly to $(m_t)_{t=0}^T$, as $N$ goes to infinity. 
		Since $(m_t)_{t=0}^T \in \PX$ is deterministic, the convergence result above holds in probability, that is, for any fixed $\epsilon > 0$, $\mathbb{P}_N^m(\text{dist}(\mu_t^{1,N},\mu_t)>\epsilon)\overset{N \to \infty }{\longrightarrow}0$. Then, we have
			\[
			\begin{split}
			\mathbb{E}_N^m[\text{dist}(\mu_t^{1,N},\mu_t)]&
			\leq \mathbb{P}_N^m(\text{dist}(\mu_t^{1,N},\mu_t)>\epsilon) + \epsilon\mathbb{P}_N^m(\text{dist}(\mu_t^{1,N},\mu_t)\leq\epsilon)\\&
			\leq \mathbb{P}_N^m(\text{dist}(\mu_t^{1,N},\mu_t)>\epsilon) + \epsilon \overset{N \to \infty }{\longrightarrow} \epsilon,
			\end{split}			
			\]
			and we obtain by the arbitrariness of $\epsilon >0$ that $\mathbb{E}_N^m[\text{dist}(\mu_t^{1,N},\mu_t)]\overset{N \to \infty }{\longrightarrow}0$, for any $t \in [\![0,T]\!]$.
			Finally, by disintegration, an application of the dominated convergence theorem yields
			\begin{equation*}
				\lim_{N\to \infty}\mathbb{E}_{N}[\text{dist}(\mu^{1,N}_t,\mu_t)] 
				=\int_{\mathcal{P}(\mathcal{X})^{T+1}}\lim_{N\to \infty}\mathbb{E}_{N}[\text{dist}(\mu^{1,N}_t,\mu_t)|\mu=m] \rho_2(dm)
				=0, \text{ for all }t \in [\![0,T]\!],
			\end{equation*}	
			and consequently
			\begin{equation}\label{Eq_conv_mu_N}
				\lim_{N\to \infty}\mathbb{E}_{N}[\text{dist}_T(\mu^{1,N},\mu)]=0.
			\end{equation}
			Now, we  prove the following claim.
			\begin{claim}\label{Claim_conv_hat_mu_1N}
			For any $\lambda=(\lambda_t)_{t=0}^T \in  \{\widetilde \mu_t^{j,N}, j\in [\![1,N]\!] \} \cup \{\mu_t^{j,N}, j\in [\![1,N]\!] \}$
			\begin{equation}\label{conv_mu_N_1}
				\lim_{N \to \infty}\mathbb{E}\left[ \text{\emph{dist}}_T(\lambda,\widehat{\mu}^{1,N})\right]=0.
			\end{equation}
			\end{claim}
			
			\begin{proof}[Proof of Claim \ref{Claim_conv_hat_mu_1N}]
			We prove the claim for $\lambda = \widetilde\mu^{j,N}$, the proof for $\lambda =\mu^{j,N}$ being similar.
			Since by definition of $\text{dist}_T$, we have
			\[
				\mathbb{E}\left[ \text{dist}_T(\widetilde{\mu}^{j,N},\widehat{\mu}^{1,N})\right]
				=\mathbb{E}\left[ \sum_{t=0}^T\text{dist}(\widetilde{\mu}^{j,N}_t,\widehat{\mu}^{1,N}_t)\right]
				= \sum_{t=0}^T\mathbb{E}\left[ \text{dist}(\widetilde{\mu}^{j,N}_t,\widehat{\mu}^{1,N}_t)\right],
			\]
			it suffices to prove that, for any $j \in [\!]1,N]\!]$, and any $t \in [\!]0,T]\!]$,
			\begin{equation}\label{conv_mu_N}
				\lim_{N \to \infty}\mathbb{E}\left[ \text{dist}(\widetilde{\mu}^{j,N}_t,\widehat{\mu}^{1,N}_t)\right]=0.
			\end{equation}
			We notice that the definition of the distance dist together with the upper bound for empirical measures in (2.1) in \cite{CF2022} implies that,  for all $j \in [\!]1,N]\!]$, $t \in [\!]0,T]\!]$,
			\begin{equation}\label{stima_hat}
				\mathbb E \left[ \text{dist}(\widehat{\mu}^{1,N}_t,\widetilde{\mu}^{j,N}_t) \right]
\leq \frac{1}{N-1}+ \frac{1}{N-1}\sum_{l=2}^N \mathbb{P}\left(\widetilde{X}^{l,N}_t\neq \widehat{X}^{l,N}_t\right).
			\end{equation}
			In fact, for $j=1$, we have
		 	\[
				\mathbb E \left[ \text{dist}(\widehat{\mu}^{1,N}_t,\widetilde{\mu}^{j,N}_t) \right]
				\overset{(2.1)}{\leq} \mathbb E \left[ \frac{1}{N-1}\sum_{l=2}^N \boldsymbol{1}_{\widetilde{X}^{l,N}_t\neq \widehat{X}^{l,N}_t} \right]
				=\frac{1}{N-1}\sum_{l=2}^N \mathbb{P}\left(\widetilde{X}^{l,N}_t\neq \widehat{X}^{l,N}_t\right).
			\]
			Whereas, for $j \in [\!]2,N]\!]$, we get
 			\[
				\begin{split}
\mathbb E \left[ \text{dist}(\widehat{\mu}^{1,N}_t,\widetilde{\mu}^{j,N}_t) \right]&
				\overset{(2.1)}{\leq} \mathbb E \left[ \frac{1}{N-1}\sum_{l=2,l\neq j}^N \boldsymbol{1}_{\widetilde{X}^{l,N}_t\neq \widehat{X}^{l,N}_t}+\frac{1}{N-1}\boldsymbol{1}_{\widetilde{X}^{1,N}_t\neq \widehat{X}^{j,N}_t} \right]
				\\&
				\leq \mathbb E \left[ \frac{1}{N-1}\sum_{l=2}^N \boldsymbol{1}_{\widetilde{X}^{l,N}_t\neq \widehat{X}^{l,N}_t}+\frac{1}{N-1} \right]
				\\&
				=\frac{1}{N-1}\sum_{l=2}^N \mathbb{P}			\left(\widetilde{X}^{l,N}_t\neq \widehat{X}^{l,N}_t\right)+\frac{1}{N-1}.
				\end{split}
			\]
			Furthermore, we prove that, for all  $t \in [\!]0,T]\!]$, we have the following convergence, as \hbox{$N \to\infty$},
			\begin{equation}\label{tilde_X}
				\lim_{N\to \infty}\frac{1}{N-1} \sum_{j=2}^N \mathbb{P}\left( \widetilde{X}^{j,N}_t\neq \widehat{X}^{j,N}_t  \right)=0.
			\end{equation} 
			In fact, (\ref{tilde_X}), with $t=0$, follows from the fact that, for all $N \in \mathbb N$,
			\[
				\sum_{j=2}^N \mathbb{P}\left( \widetilde{X}^{j,N}_0\neq \widehat{X}^{j,N}_0  \right)=0.
			\]
			which is a consequence of the fact that, by construction, for all $N \in \mathbb N$, $j \in [\![1,N]\!]$,  
			$\widetilde{X}^{j,N}_0={X}^j_0$, $\mathbb{P}_N$-a.s. and for all $N \in \mathbb N$, $j \in [\![2,N]\!]$,  
			$\widehat{X}^{j,N}_0={X}^j_0$, $\mathbb{P}_N$-a.s..
			Then,  we prove  (\ref{tilde_X}) for a generic time, reasoning by induction. 
			Let us assume that  (\ref{tilde_X}) holds for $t$, for all $j \in [\!]2,N]\!]$,  we have
			\[
				\begin{split}
					 \mathbb{P}\left( \widetilde{X}^{j,N}_{t+1}\neq \widehat{X}^{j,N}_{t+1}  \right)&
					%=\mathbb{E}\left[ \boldsymbol{1}_{\widetilde{X}^{j,N}_{t+1}\neq \widehat{X}^{j,N}_{t+1}} \right]\\&
					= \mathbb{P}\left( \widetilde{X}^{j,N}_{t+1}\neq \widehat{X}^{j,N}_{t+1}, \widetilde{X}^{j,N}_t\neq \widehat{X}^{j,N}_t\right)
					+ \mathbb{P}\left( \widetilde{X}^{j,N}_{t+1}	\neq \widehat{X}^{j,N}_{t+1}, \widetilde{X}^{j,N}_t= \widehat{X}^{j,N}_t  \right)
					\\&
					\leq  \mathbb{P}\left( \widetilde{X}^{j,N}_t\neq \widehat{X}^{j,N}_t  \right)
+ \mathbb{P}\left( \widetilde{X}^{j,N}_{t+1}\neq \widehat{X}^{j,N}_{t+1}, \widetilde{X}^{j,N}_t= \widehat{X}^{j,N}_t  \right)
					=: \star,
				\end{split}
			\]
			where we have exploited disintegration. 
			Using the iterative definition of the processes $(\widetilde{X}^{j,N}_t)_{t =0}^T$ and $(\widehat{X}^{j,N}_t)_{t =0}^T$ through $\Psi$ and the fact that $\Phi_j$, by construction, takes values in $\mathcal{R}$ we get
			\begin{align*}
					\star
					&= \mathbb{P}\bigg(  \Psi\left(t,\widetilde{X}^{j,N}_t,\widetilde{\mu}^{j,N}_t, \Phi_j(t,\widetilde{X}^{j,N}_t),\xi^j_{t+1}\right)\neq \Psi\left(t,\widehat{X}^{j,N}_t,\widehat{\mu}^{1,N}_t, \Phi_j(t,\widehat{X}^{j,N}_t),\xi^j_{t+1}\right), \widetilde{X}^{j,N}_t= \widehat{X}^{j,N}_t\bigg)\\
					&\quad +  \mathbb{P}\left( \widetilde{X}^{j,N}_t\neq \widehat{X}^{j,N}_t  \right)
					\\
					&= \mathbb{P}\bigg(  \Psi\left(t,\widehat{X}^{j,N}_t,\widetilde{\mu}^{j,N}_t, \Phi_j(t,\widehat{X}^{j,N}_t),\xi^j_{t+1}\right)\neq \Psi\left(t,\widehat{X}^{j,N}_t,\widehat{\mu}^{1,N}_t, \Phi_j(t,\widehat{X}^{j,N}_t),\xi^j_{t+1}\right), \widetilde{X}^{j,N}_t= \widehat{X}^{j,N}_t\bigg) \\
					&\quad +  \mathbb{P}\left( \widetilde{X}^{j,N}_t\neq \widehat{X}^{j,N}_t  \right)\\
					&=  \mathbb{P}\left( \widetilde{X}^{j,N}_t\neq \widehat{X}^{j,N}_t  \right)+\mathbb{P}\bigg(  \Psi\left(t,\widehat{X}^{j,N}_t,\widetilde{\mu}^{j,N}_t, \Phi_j(t,\widehat{X}^{j,N}_t),\xi^j_{t+1}\right)\neq \Psi\left(t,\widehat{X}^{j,N}_t,\widehat{\mu}^{1,N}_t, \Phi_j(t,\widehat{X}^{j,N}_t),\xi^j_{t+1}\right) \bigg)
			\end{align*}
			Then, an application of Fubini's Theorem, together with the independence properties of 
			\\
			\hbox{$(\xi^j_{t+1})_{j =2}^N$}, yields
			\begin{equation}\label{to_be_con}
				\begin{split}
					\mathbb{P}\left( \widetilde{X}^{j,N}_{t+1}\neq \widehat{X}^{j,N}_{t+1}  \right)&
					=  \mathbb{P}\left( \widetilde{X}^{j,N}_t\neq \widehat{X}^{j,N}_t  \right)+ \mathbb{E}\left[ \int_{\mathcal{Z}}\boldsymbol{1}_{\Psi\left(t,\widehat{X}^{j,N}_t,\widetilde{\mu}^{j,N}_t, \Phi_j(t,\widehat{X}^{j,N}_t), z\right)\neq \Psi\left(t,\widehat{X}^{j,N}_t,\widehat{\mu}^{1,N}_t, \Phi_j(t,\widehat{X}^{j,N}_t), z\right)} \nu(dz)\right]\\&
\leq   \mathbb{P}\left( \widetilde{X}^{j,N}_t\neq \widehat{X}^{j,N}_t  \right)
+ \mathbb{E}\left[w(\text{dist}(\widetilde{\mu}^{j,N}_t,\widehat{\mu}^{1,N}_t)\right],
				\end{split}
			\end{equation}
			where the inequality in the last row follows from Assumption \hypref{HypPsiContinuity}~1).
			\\
			Now, notice that
			\begin{equation}\label{dis_dist_mu}
				\lim_{N\to \infty}\max_{j \in [\![2,N]\!]} \mathbb{E}\left[\text{dist}(\widetilde{\mu}^{j,N}_t,\widehat{\mu}^{1,N}_t)\right]
				\overset{(\ref{stima_hat})}{\leq}\lim_{N\to \infty}\left\{ \frac{1}{N-1}+ \frac{1}{N-1}\sum_{l=2}^N \mathbb{P}\left(\widetilde{X}^{l,N}_t\neq \widehat{X}^{l,N}_t\right)\right\}					=0,
			\end{equation}
			because of the induction hypothesis. 
			Thence,  with the notation $\text{dist}(\widetilde{\mu}^{j,N}_t,\widehat{\mu}^{1,N}_t)=\delta_j^N$, for any $\varepsilon >0$, we have
			\[
				\begin{split}
					\max_{j \in [\![2,N]\!]} \mathbb{E}\left[w(\text{dist}(\widetilde{\mu}^{j,N}_t,\widehat{\mu}^{1,N}_t))\right]&
					=\max_{j \in [\![2,N]\!]} \mathbb{E}\left[w(\delta_j^N)\right]\\&
					\leq \max_{j \in [\![2,N]\!]} \bigg\{\mathbb{E}\left[w(\delta_j^N)\big| \delta_j^N \geq \varepsilon\right] \mathbb{P}( \delta_j^N \geq \varepsilon) +\mathbb{E}\left[w(\delta_j^N)\big| \delta_j^N < \varepsilon\right] \mathbb{P}( \delta_j^N < \varepsilon) \bigg\} \\&
					\leq \max_{j \in [\![2,N]\!]} \bigg\{||w||_{\infty} \mathbb{P}( \delta_j^N \geq \varepsilon)
 +  \mathbb{E}\left[w(\delta_j^N)\big| \delta_j^N < \varepsilon\right]  \bigg\}\\&
					\leq \max_{j \in [\![2,N]\!]} \bigg\{||w||_{\infty} \mathbb{P}( \delta_j^N \geq \varepsilon)
 + w(\varepsilon) \bigg\}
					\leq   w(\varepsilon)+ ||w||_{\infty} \max_{j \in [\![2,N]\!]} \frac{\mathbb{E}[ \delta_j^N]}{\varepsilon}\\&
					\leq   w(\varepsilon)+ \frac{||w||_{\infty}}{\varepsilon} \max_{j \in [\![2,N]\!]} {\mathbb{E}[ \text{dist}(\widetilde{\mu}^{j,N}_t,\widehat{\mu}^{1,N}_t)]} 								\overset{N\to \infty}{\longrightarrow} w(\varepsilon),
				\end{split}
			\]
			where we have made use of  disintegration, the fact that $w$ is bounded, Markov's inequality and the convergence result in  (\ref{dis_dist_mu}). 
		
			The fact that $w$ converges to $0$ as its argument goes to zero and the arbitrariness of $\varepsilon > 0 $ therefore implies
			\begin{equation}\label{Eq_conv_w}
				\lim_{N \to \infty}\left\{\max_{j \in [\![2,N]\!]} \mathbb{E}\left[w(\text{dist}(\widetilde{\mu}^{j,N}_t,\widehat{\mu}^{1,N}_t))\right]\right\}=0.
			\end{equation}
			Applying once more the induction hypothesis to the inequality in (\ref{to_be_con}), we get
			\begin{equation}
				\begin{split}
					\lim_{N\to \infty} \frac{1}{N-1} \mathbb{P}\left( \widetilde{X}^{j,N}_{t+1}\neq \widehat{X}^{j,N}_{t+1}  \right)&
					\leq \lim_{N\to \infty} \Bigg\{\frac{1}{N-1}  \mathbb{P}\left( \widetilde{X}^{j,N}_t\neq \widehat{X}^{j,N}_t  \right) +\max_{j \in [\![2,N]\!]} \mathbb{E}\left[w(\text{dist}(\widetilde{\mu}^{j,N}_t,\widehat{\mu}^{1,N}_t)\right]\Bigg\}= 0.
				\end{split}
			\end{equation}
			Thus, we have shown (\ref{tilde_X}), which, together with \eqref{stima_hat}, implies \eqref{conv_mu_N} and so our claim.
			\end{proof}
			
			Then, by the triangular inequality and the monotonicity of expectation, Equation \eqref{Eq_conv_mu_N} together with the statement in Claim \ref{Claim_conv_hat_mu_1N} yields
			\begin{align}\label{Eq_conv_tilde_mu_N}
					\mathbb{E}_{N}[\text{dist}_T(\widetilde \mu^{1,N},\mu)]
					\leq \mathbb{E}_{N}[\text{dist}_T(\widetilde \mu^{1,N},\widehat \mu^{1,N})]+\mathbb{E}_{N}[\text{dist}_T(\widehat \mu^{1,N},\mu^{1,N})]+\mathbb{E}_{N}[\text{dist}_T(\mu^{1,N},\mu)]	
					\overset{N \to \infty}{\longrightarrow}0.		
			\end{align}						
			
			Now, set 
			\begin{equation}
			\widetilde{J}_1^N(m_0^{\otimes N}, \gamma^N,\widetilde{\gamma}^N)
			:= \mathbb{E}_N\left[ \sum_{t=0}^T f(t,\widetilde X_t^{1,N},\mu_t,\widetilde u^{1,N}_t ) + F(\widetilde X_T^{1,N},\mu_T)\right],
			\end{equation}	
			and
			\begin{equation}
			\overline{J}_1^N(m_0^{\otimes N}, \gamma^N,\widetilde{\gamma}^N)
			:= \mathbb{E}_N\left[ \sum_{t=0}^T f(t,\overline X_t^{1,N},\mu_t, \widetilde u^{1,N}_t) + F(\overline X_T^{1,N},\mu_T)\right],
			\end{equation}
			with processes $\widetilde X^{1,N}$ and $\overline X^{1,N}$ defined in Equations \eqref{Eq_tilde_X_j} and \eqref{Eq_bar_X_j}.
			\\
			Now, consider a real valued sequence $\{f_n\}_{n \in \mathbb{N}}$ s.t., for any $n \in \mathbb{N}$, $f_n=h_n+g_n+h$ with $\lim_{n \to \infty}h_n=0$, $g_n \geq 0$, for all $n \in \mathbb{N}$. Then, 
		\begin{align}\label{Eq_lim_inf}
		\liminf_{n \to \infty} f_n \geq h.
		\end{align}			
			
			In order to prove Equation \eqref{2)}, we want to exploit the inequality in Equation \eqref{Eq_lim_inf} with $g_N=\overline{J}_1^N(m_0^{\otimes N},\gamma^N,\widetilde{\gamma}^N)-J(m_0,\rho,\iota)$, $h_N={J}_1^N(m_0^{\otimes N}, \gamma^N,\widetilde{\gamma}^N)-\overline{J}_1^N(m_0^{\otimes N},\gamma^N,\widetilde{\gamma}^N)$ and $h=J(m_0,\rho,\iota)$.
			First of all, \hypref{HypCosts} and the convergence in Equation \eqref{Eq_conv_tilde_mu_N} imply
			\begin{align}\label{Eq_J_tilde_j}
			&|{J}_1^N(m_0^{\otimes N}, \gamma^N,\widetilde{\gamma}^N)-\widetilde{J}_1^N(m_0^{\otimes N}, \gamma^N,\widetilde{\gamma}^N)|\\ \nonumber
			& \leq \mathbb{E}_N\Bigg[ \sum_{t=0}^T |f(t,\widetilde X_t^{1,N},\widetilde \mu^{1,N}_t, \widetilde u^{1,N}_t)-f(t,\widetilde X_t^{1,N},\mu_t, \widetilde u^{1,N}_t)| + |F(\widetilde X_T^{1,N},\widetilde\mu^{1,N}_T)-F(\widetilde X_T^{1,N},\mu_T)|\Bigg]\\\nonumber
			&  \leq \mathbb{E}_N\left[ \sum_{t=0}^T L\text{dist}(\widetilde\mu^{1,N}_t, \mu_t)	+ L\text{dist}(\widetilde\mu^{1,N}_T, \mu_T)	\right] = L \mathbb{E}\left[\text{dist}_T(\widetilde\mu^{1,N}, \mu)\right]\overset{N \to \infty}{\longrightarrow}0.
			\end{align}
			Furthermore, for all $t \in [\![0,T]\!]$,
			\begin{align}\label{Eq_conv_over_tilde_X_N}
			\lim_{N \to \infty}\mathbb{P}_N (\widetilde{X}^{1,N}_t \neq \overline{X}^{1,N}_t) =0.
			\end{align}
			We show this by induction on $t \in [\![0,T]\!]$.
			Indeed, for $t=0$, $\mathbb{P}_N (\widetilde{X}^{1,N}_0 \neq \overline{X}^{1,N}_0) =0$, being $\widetilde{X}^{1,N}_0 = \overline{X}^{1,N}_0=X^1_0$, $\mathbb{P}_N$-a.s., by construction.
			Now, suppose that  $\lim_{N \to \infty}\mathbb{P}_N (\widetilde{X}^{1,N}_t \neq \overline{X}^{1,N}_t)=0$, for some $t \in [\![0,T]\!]$. Then, exploiting Assumption \hypref{HypPsiContinuity}~1), we obtain
			\begin{align*}
				\mathbb{P}_N (\widetilde{X}^{1,N}_{t+1} \neq \overline{X}^{1,N}_{t+1})
				& \leq \mathbb{P}_N (\widetilde{X}^{1,N}_t \neq \overline{X}^{1,N}_t)+\mathbb{P}_N (\widetilde{X}^{1,N}_{t+1} \neq \overline{X}^{1,N}_{t+1},\widetilde{X}^{1,N}_t = \overline{X}^{1,N}_t)\\
				& \leq \mathbb{P}_N (\widetilde{X}^{1,N}_t \neq \overline{X}^{1,N}_t)\\
				& \quad  +\mathbb{P}_N \left(\Psi\left(t,\widetilde X^{1,N}_t, \widetilde \mu^{1,N}_t, \widetilde u^{1,N}_t,\xi^1_{t+1}\right) \neq \Psi\left(t,\overline X^{1,N}_t, \mu_t, \widetilde u^{1,N}_t,\xi^1_{t+1}\right),\widetilde{X}^{1,N}_t = \overline{X}^{1,N}_t\right)\\
				& \leq \mathbb{P}_N (\widetilde{X}^{1,N}_t \neq \overline{X}^{1,N}_t) +\mathbb{P}_N \left(\Psi\left(t,\widetilde X^{1,N}_t, \widetilde \mu^{1,N}_t, \widetilde u^{1,N}_t,\xi^1_{t+1}\right) \neq \Psi\left(t,\widetilde X^{1,N}_t, \mu_t, \widetilde u^{1,N}_t,\xi^1_{t+1}\right)\right)\\
				& \leq \mathbb{P}_N (\widetilde{X}^{1,N}_t \neq \overline{X}^{1,N}_t) +\mathbb{E}_N \left[w\left(\text{dist}(\widetilde \mu^{1,N}_t, \mu_t)\right)\right],
			\end{align*}
			and the last term on the right goes to zero as $N$ goes to infinity by the induction assumption and the convergence in Equation \eqref{Eq_conv_tilde_mu_N}, reasoning in a similar way as in the proof of Equation \eqref{Eq_conv_w}.
			As a consequence, we see
			\begin{align}\label{Eq_tilde_J_bar_J}
			&|\widetilde{J}_1^N(m_0^{\otimes N}, \gamma^N,\widetilde{\gamma}^N)-\overline{J}_1^N(m_0^{\otimes N}, \gamma^N,\widetilde{\gamma}^N)|\\ \nonumber
			& \leq \mathbb{E}_N\Bigg[ \sum_{t=0}^T |f(t,\widetilde X_t^{1,N},\mu_t, \widetilde u^{1,N}_t)-f(t,\overline X_t^{1,N},\mu_t, \widetilde u^{1,N}_t)| + |F(\widetilde X_T^{1,N},\mu_T)-F(\overline X_T^{1,N},\mu_T)|\Bigg]\\ \nonumber
			&  \leq 2||f||_{\infty} \sum_{t=0}^T  \mathbb{P}_N(\widetilde{X}^{1,N}_t \neq \overline{X}^{1,N}_t) + 2||F||_{\infty}\mathbb{P}_N (\widetilde{X}^{1,N}_T \neq \overline{X}^{1,N}_T)\overset{N \to \infty}{\longrightarrow}0,
			\end{align}
			where we have exploited the fact that $f$ and $F$ being $L$-Lipschitz continuous real-valued function on a compact domain are bounded.
			The convergences in Equations \eqref{Eq_J_tilde_j} and \eqref{Eq_tilde_J_bar_J} implies 
			\begin{align*}
				|h_N| 
				& = |{J}_1^N(m_0^{\otimes N}, \gamma^N,\widetilde{\gamma}^N)-\overline{J}_1^N(m_0^{\otimes N},\gamma^N,\widetilde{\gamma}^N)| \\
				& \leq |{J}_1^N(m_0^{\otimes N}, \gamma^N,\widetilde{\gamma}^N)-\overline{J}_1^N(m_0^{\otimes N},\gamma^N,\widetilde{\gamma}^N)| + |{J}_1^N(m_0^{\otimes N}, \gamma^N,\widetilde{\gamma}^N)-\overline{J}_1^N(m_0^{\otimes N},\gamma^N,\widetilde{\gamma}^N)| \overset{N \to \infty}{\longrightarrow}0.
			\end{align*}
			Thus, an application of the inequality in \eqref{Eq_lim_inf} with 
			$$
			g_N=\overline{J}_1^N(m_0^{\otimes N},\gamma^N,\widetilde{\gamma}^N)-J(m_0,\rho,\iota),
			$$
			$$
			h_N={J}_1^N(m_0^{\otimes N}, \gamma^N,\widetilde{\gamma}^N)-\overline{J}_1^N(m_0^{\otimes N},\gamma^N,\widetilde{\gamma}^N)
			$$ 
			and 
			$$
			h=J(m_0,\rho,\iota),
			$$
			yields \eqref{2)} provided that $g_N=\overline{J}_1^N(m_0^{\otimes N}, \gamma^N,\widetilde{\gamma}^N) - J(m_0, \rho, \iota) \geq 0$. 
			This is a consequence of the fact that $\overline{J}_1^N(m_0^{\otimes N}, \gamma^N,\widetilde{\gamma}^N)$ can be interpreted as the value of the MFG when the representative player implements the strategy $\widetilde u^{1,N}_t= \widetilde \Upsilon_1^N(t,\widetilde X_t^{1,N},\widetilde \mu_1^N)$, $t \in [\![0,T-1]\!]$. Indeed, the realization of the triple $(m_0^{\otimes N}, \gamma^N, \widetilde \gamma^N )$ for the first player on the previously defined complete probability space $( \Omega_N,\mathcal{F}_N,\mathbb{P}_N)$  can be seen as a tuple $\big((\Omega_N, \mathcal{F}_N,\{\mathcal{G}^N_t\}_{t=0}^{T-1} ,\mathbb{P}_N),\Phi_1, (\mu_t)_{t=0}^T, X_0,(\xi^1_t)_{t=1}^T$, $ (\widetilde u^{1,N}_t)_{t=0}^{T-1}$,  $(\overline X^{1,N}_t)_{t=0}^T \big)$ such that
			\begin{itemize}
			\item[\textbf{i)}]  $\mathbb{P}_N\circ (X^1_0)^{-1}=m_0$;
			\item[\textbf{ii)}] $\mathbb{P}_N\circ (\Phi_1,(\mu_t)_{t=0}^T)^{-1}=\rho$;
			\item[\textbf{iii)}] $(\xi^1_t)_{t=1}^T$,   $\mathcal{Z}$-valued random variables  i.i.d. all distributed according to $\nu;$

			\item[\textbf{iv)}] $X_0^1$,  $(\xi^1_t)_{t=1}^T$, $(\Phi_1,(\mu_t)_{t=0}^T)$ are independent;

			\item[\textbf{iv')}] For each $t \in [\![0,T-1]\!]$,
			\begin{itemize}
			\item $\xi^1_t$ is $\mathcal{G}^N_t$-measurable and $(\xi^1_{t+k})_{k=1}^T$ are jointly independent of $\mathcal{G}^N_t$,
			
			\item $\mathcal{G}^N_t=\mathcal{H}^N_t\lor \sigma(\mu^{(t)})\lor\sigma(\Phi_1)\lor\sigma(X^1_0)$, with $\mathcal{H}^N_t$ independent of $\sigma(\Phi_1,(\mu_t)_{t=0}^T,X^1_0)$,
			
			\item $\widetilde u^{1,N}_t$ is $\mathcal{G}^N_t$-measurable,
			\end{itemize} 
			
			\item[\textbf{v)}] Finally, for $t \in [\![0,T-1]\!]$, the state dynamics for the first player is given by
		\[
		\overline X^{1,N}_{t+1}= \Psi\left(t,\overline X_t^{1,N}, \mu_t, \widetilde u^{1,N}_t,\xi^1_{t+1}\right), \quad \mathbb{P}_N\text{-a.s.}
		\]
		\end{itemize}
		Above we have exploited the fact that, by definition, the sequence of control actions $(\widetilde u^{1,N}_t)_{t=0}^{T-1}$, 
		$$
		\widetilde u^{1,N}_t= \widetilde \Upsilon_1^N(t,\widetilde X^{1,N},\widetilde \mu^{1,N})=w_t^N(\vartheta_t,\Phi_1)((\widetilde X^{1,N})^{(t)},(\widetilde\mu^{1,N})^{(t)}),
		$$ 
		is adapted to the filtration $\{\mathcal{G}^N_t\}_{t=0}^{T-1}$, defined as
		$$
		\mathcal{G}^N_t
		:=\sigma((X^j_0)_{j=1}^N, (\xi^1_s, \dots, \xi^N_s)_{s=1}^t, \Phi_1, (\vartheta_s)_{s=0}^t, (Z_j)_{j=2}^N, \mu^{(t)} )= \mathcal{H}^N_t \lor \sigma(\mu^{(t)})\lor\sigma(\Phi_1)\lor\sigma(X^1_0),
		$$ 
		with $\mathcal{H}^N_t:=\sigma((X^j_0)_{j=2}^N, (Z_j)_{j=2}^N, (\xi^1_s, \dots, \xi^N_s)_{s=1}^{t},\vartheta^{(t)})$.
		Notice that, for all $t \in [\![1,T]\!]$, $\xi^1_t$ is $\mathcal{G}^N_t$-measurable and, in turn, $\mathcal{G}^N_t$ is jointly independent of $(\xi^1_{t+k})_{k=1}^{T-t}$.	Furthermore, for all $t \in [\![0,T]\!]$, $\mathcal{H}^N_t$, $\sigma(X^1_0)$ and $\sigma(\Phi_1,(\mu_t)_{t=0}^T)$ are independent.
		\\
		Hence, the tuple $((\Omega_N, \mathcal{F}_N, \{\mathcal{G}_t^N\}_{t=0}^{T-1}, \mathbb{P}_N), \Phi_1, (\mu_t)_{t=0}^T, X^1_0, (\xi^1_t)_{t=1}^T$, $(\widetilde u^{1,N}_t)_{t=0}^{T-1}$, $(\overline X^{1,N}_t)_{t=0}^{T})$ represents a realization of the triple $(m_0,\rho,(\widetilde u^{1,N}_t)_{t=0}^{T-1})$ for the open-loop MFG, with costs given by 
		$$
		\widehat{J}(m_0, \rho, (\widetilde u^{1,N}_t)_{t=0}^{T-1})=\overline{J}_1^N(m_0^{\otimes N}, \gamma^N,\widetilde{\gamma}^N).
		$$ 
		Now, $\rho$ is a solution of the correlated MFG according to Definition \ref{Def_CMFG_closed} and the values of the objective functionals at the equilibrium for the correlated MFGs in open-loop and closed-loop strategies are the same (see Proposition \ref{Prop_equiv_CMFGs}). 
		Thus, by the optimality condition in Definition \ref{Def_CMFG_open}, we get $\overline{J}_1^N(m_0^{\otimes N}, \gamma^N,\widetilde{\gamma}^N) \geq J(m_0, \rho, \iota) \geq 0$ and this ends our proof.		
		
	\end{proof}

\renewcommand\qedsymbol{$\square$}

\end{proof}

%%%%%%%%%%%%%%%%%%%%%%%%%%%%%%%%%%%%%%%%%%%%%%%%%%%%%%%%%%%
%%%%%%%%%%%%%%%%%%%%%%%%%%%%%%%%%%%%%%%%%%%%%%%%%%%%%%%%%%%
%%%%%%%%%%%%%%%%%%%%%%%%%%%%%%%%%%%%%%%%%%%%%%%%%%%%%%%%%%%

\section{A Toy Example} \label{SectMFGExample}

	In order to further motivate the definition of mean field game solution given in Section~\ref{SectMFG-closed}, we consider the two-state example introduced in \cite{CF2022} and show that it possesses correlated solutions with non-deterministic flow of measures also in the sense of Definition~\ref{Def_CMFG_closed}. Moreover, assumptions \hyprefall\ as well as conditions \Hyprefall\ on the correlated solution will be seen to hold. 
	
	Let us recall the setting. 
	Let  $T=2$, $\mathcal{X}=\{-1,1\}$, and $\Gamma=\{0,1\}$.
	Let the system function and the cost functional, respectively, be given by
	\begin{align}\label{Eq_ex_Psi}
		\Psi(x,\gamma,z)
		& = \Psi(t,x,\gamma,z)
		= x[\mathbf{1}_{\{0\}}(\gamma)( \mathbf{1}_{[0, \frac{1}{2}]} - \mathbf{1}_{(\frac{1}{2},1]} )(z) + \mathbf{1}_{\{1\}}(\gamma)( \mathbf{1}_{[0, \frac{3}{4}]} - \mathbf{1}_{(\frac{3}{4},1]} )(z)]\nonumber\\
		&= x[(1-\gamma)( \mathbf{1}_{[0, \frac{1}{2}]} - \mathbf{1}_{(\frac{1}{2},1]} )(z) +\gamma( \mathbf{1}_{[0, \frac{3}{4}]} - \mathbf{1}_{(\frac{3}{4},1]} )(z)],
	\end{align}
	and
	\begin{align}\label{Eq_ex_f_and_F}
		& f(t,x,\gamma,m)= c_0 (1-t)\gamma + t (c_1 \gamma - x \text{M}(m)),\nonumber
		\\
		& F(x,m)= -x \text{M}(m),
	\end{align}
	with $c_0,c_1 > 0$.
	
	\vspace{-1cm}
	\begin{figure*}[htb!]
        \hspace*{-6 cm}
        \begin{subfigure}{.2\textwidth}
            \begin{tikzpicture}[->,>=stealth',shorten >=1pt,auto,node distance=2.5cm,
                    semithick]
  \tikzstyle{every state}=[fill=white,draw=black,text=black]

  \node[state] 		   (R)                    {$1$};
  \node[state]         (A) [right of=R]       {$-1$};
  \draw (A) to [out=90,in=0,looseness=10] node {1/2} (A);
  \draw	(A) to [out=90,in=90,looseness=2] node {1/2} (R);
  \draw (R) to [out=270,in=270,looseness=2] node {1/2} (A);
  \draw	(R) to [out=270,in=180,looseness=10] node {1/2} (R);
\end{tikzpicture}
        \end{subfigure}
        \hspace*{4 cm}
        \begin{subfigure}{.2\textwidth}
            \begin{tikzpicture}[->,>=stealth',shorten >=1pt,auto,node distance=2.5cm,
                    semithick]
  \tikzstyle{every state}=[fill=white,draw=black,text=black]

  \node[state] 		   (R)                    {$1$};
  \node[state]         (A) [right of=R]       {$-1$};
  \draw (A) to [out=90,in=0,looseness=10] node {3/4} (A);
  \draw	(A) to [out=90,in=90,looseness=2] node {1/4} (R);
  \draw (R) to [out=270,in=270,looseness=2] node {1/4} (A);
  \draw	(R) to [out=270,in=180,looseness=10] node {3/4} (R);
\end{tikzpicture}
        \end{subfigure}
        \vspace*{-1 cm}
        \caption{States and corresponding transition probabilities for the action $\gamma=0$ (left) and $\gamma=1$ (right).}\label{fig:states}
\end{figure*}
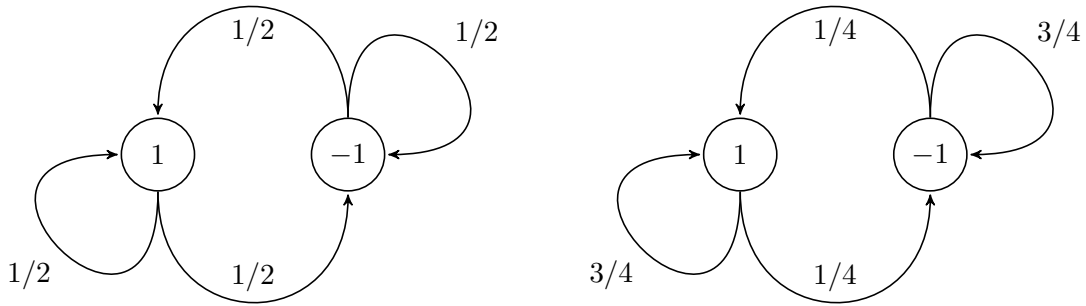

	Now, we consider the following \emph{candidate} correlated solution for the game
	\begin{align}\label{Eq_ex_rho}
		\rho = & + \beta_1(\delta_{(\varphi_+,m_+)}+\delta_{(\varphi_-,m_-)} )+\beta_2(\delta_{(\varphi_0,m_+)} + \delta_{(\varphi_0,m_-)})\nonumber \\
		& + \beta_3(\delta_{(\widehat \varphi_+,\widehat{m}_+)}+\delta_{(\widehat \varphi_-,\widehat{m}_-)})   + \beta_4(\delta_{(\varphi_0,\widehat{m}_+)} +\delta_{(\varphi_0,\widehat{m}_-)}),
	\end{align}
	where
	\begin{align}\label{Eq_ex_phis}
		&\varphi_0(t,x):=0, 
		\qquad
		\varphi_+(t,x):= \mathbf{1}_{\{1\}}(x)=\frac{1+x}{2},
		\qquad
		\varphi_-(t,x):= \mathbf{1}_{\{-1\}}(x)=\frac{1-x}{2},\nonumber \\
		& \widehat \varphi_+(t,x)= \mathbf{1}_{\{0\}}(t)\mathbf{1}_{\{1\}}(x)=\frac{(1-t)(1+x)}{2},
		\qquad
		\widehat  \varphi_-(t,x):= \mathbf{1}_{\{0\}}(t)\mathbf{1}_{\{-1\}}(x)=\frac{(1-t)(1-x)}{2}
	\end{align}
	and 
	\begin{align}
		m_+ :=(m_0,m_1^+,m_2^+),
		\quad
		m_+ :=(m_0,m_1^-,m_2^-),
		\quad
		\widehat m_+:=(m_0,m_1^+,m_0),
		\quad
		\widehat m_-:=(m_0,m_1^-,m_0),
	\end{align}
	with	
	\begin{align}\label{Eq_ex_ms}
		& m_0=\frac{1}{2}\delta_1+\frac{1}{2}\delta_{-1},
		& \nonumber \\
		& m_1^+=\frac{5\beta_1+4\beta_2}{8(\beta_1+\beta_2)}\delta_1+\frac{3\beta_1+4\beta_2}{8(\beta_1+\beta_2)}\delta_{-1},
		\qquad
		& m_1^-=\frac{3\beta_1+4\beta_2}{8(\beta_1+\beta_2)}\delta_1+\frac{5\beta_1+4\beta_2}{8(\beta_1+\beta_2)}\delta_{-1},\nonumber\\
		& m_2^+=\frac{21\beta_1+16\beta_2}{32(\beta_1+\beta_2)}\delta_1+\frac{11\beta_1+16\beta_2}{32(\beta_1+\beta_2)}\delta_{-1},
		\qquad
		& m_2^-=\frac{11\beta_1+16\beta_2}{32(\beta_1+\beta_2)}\delta_1+\frac{21\beta_1+16\beta_2}{32(\beta_1+\beta_2)}\delta_{-1},
	\end{align}	
	and $\beta_i > 0$, $i \in [\![1,4]\!]$, $\sum_{i=1}^4 \beta_i=\frac{1}{2}$.
	\\	
	
%%%%%%%%%%%%%%%%%%%%%%%%%%%%%%%%%%%%%%%%%%%%%%%%%%
Let $((\Omega,\mathcal{F},\mathbb{P}), \Phi, \iota, (X_0,X_1,X_2), (\mu_0,\mu_1,\mu_2), (\xi_1,\xi_2))$ be a realization of $(m_0, \rho, \iota)$.
First of all, let's check that this example satisfies the additional assumptions we have set for this extended framework.

\begin{itemize}

	\item[\hypref{HypPsiTransitions}] Fix $t \in \{0,1\}$, $x,y \in \{-1,1\}$ and $\gamma \in \{0,1\}$ and let $Z$ be a r.v. distributed according to $\nu$ defined on a probability space $(\Omega, \mathcal{F}, \mathbb{P})$.
	We have
	\begin{align*}
		\mathbb{P}(\Psi(x,\gamma,Z)=y)
		= \mathbb{P}(x[(1-\gamma)( \mathbf{1}_{[0, \frac{1}{2}]} - \mathbf{1}_{(\frac{1}{2},1]} )(Z) +\gamma( \mathbf{1}_{[0, \frac{3}{4}]} - \mathbf{1}_{(\frac{3}{4},1]} )(Z)]=y)
	\end{align*}
	and so
	\begin{itemize}
		\item for $x=y \in \{-1,1\}$ and $\gamma=0$:
		\begin{align*}
			\mathbb{P}\left(\Psi(x,\gamma,Z)=y\right)
			= \mathbb{P}\left(( \mathbf{1}_{[0, \frac{1}{2}]} - \mathbf{1}_{(\frac{1}{2},1]} )(Z)=1\right)
			=\mathbb{P}\left(Z \in \left[0,\frac{1}{2}\right]\right)
			=\frac{1}{2};
		\end{align*}
		
		\item for $x=y \in \{-1,1\}$ and $\gamma=1$:
		\begin{align*}
			\mathbb{P}\left(\Psi(x,\gamma,Z)=y\right)
			= \mathbb{P}\left(( \mathbf{1}_{[0, \frac{3}{4}]} - \mathbf{1}_{(\frac{3}{4},1]} )(Z)=1\right)
			=\mathbb{P}\left(Z \in \left[0,\frac{3}{4}\right]\right)
			=\frac{3}{4};
		\end{align*}
		
		\item for $x\neq y \in \{-1,1\}$ and $\gamma=0$:
		\begin{align*}
			\mathbb{P}\left(\Psi(x,\gamma,Z)=y\right)
			= \mathbb{P}\left(( \mathbf{1}_{[0, \frac{1}{2}]} - \mathbf{1}_{(\frac{1}{2},1]} )(Z)=-1\right)
			=\mathbb{P}\left(Z \in \left(\frac{1}{2},1\right]\right)
			=\frac{1}{2};
		\end{align*}
		
		\item for $x\neq y \in \{-1,1\}$ and $\gamma=1$:
		\begin{align*}
			\mathbb{P}\left(\Psi(x,\gamma,Z)=y\right)
			= \mathbb{P}\left(( \mathbf{1}_{[0, \frac{3}{4}]} - \mathbf{1}_{(\frac{3}{4},1]} )(Z)=-1\right)
			=\mathbb{P}\left(Z \in \left(\frac{3}{4}, 1\right]\right)
			=\frac{1}{4}.
		\end{align*}
	\end{itemize}
	Thus, for any $t \in \{0,1\}$, $x,y \in \{-1,1\}$ and $\gamma \in \{0,1\}$,
	\begin{align}
		\mathbb{P}\left(\Psi(x,\gamma,Z)=y\right) \geq \frac{1}{4} > 0.
	\end{align}
	
	\item[\Hypref{HypSolutionFlows}] Notice that in the example $\mathbb{P}(\Phi = \varphi) >0$ if and only if $\varphi \in \{ \varphi_0,  \varphi_+, \varphi_-, \widehat{\varphi}_+, \widehat{\varphi}_-  \} =: \mathfrak{F}$. 
	Thus, if $\varphi \in \mathfrak{F}\setminus \{\varphi_0\}$, the conditions in \Hypref{HypSolutionFlows} are obviously satisfied. Indeed, the corresponding set $\mathcal{P}_\varphi$ reduces to a singleton:  in particular, we have $\mathcal{P}_{\varphi_+}= \{m_+\}$, $\mathcal{P}_{\varphi_-}= \{m_-\}$, $\mathcal{P}_{\widehat\varphi_+}= \{\widehat m_+\}$ and $\mathcal{P}_{\widehat\varphi_-}= \{\widehat m_-\}$. 
	When $\{\Phi = \varphi_0 \}$, we have $\mathcal{P}_{\varphi_0} = \{m_+, m_-, \widehat{m}_+, \widehat{m}_- \}$ and:
	\begin{enumerate}
		\item $|\mathcal{P}_{\varphi_0}|=4$;
		
		\item $\mathbb{P}_{\varphi_0}(\mu \in \mathcal{P}_{\varphi_0})=1$;
		
		\item $\mathbb{P}_{\varphi_0}(\mu = m)\geq \min\{\frac{\beta_2}{2(\beta_2+\beta_4)}, \frac{\beta_4}{2(\beta_2+\beta_4)}\}$, for any $m \in \mathfrak{M}:= \{m_+, m_-, \widehat{m}_+, \widehat{m}_- \}$.
	\end{enumerate}
	Further notice that, in this case, we have
	\begin{align*}
		\mathcal{P}_{\varphi_0}^{(0)} = \{m_0\},
		\qquad
		\mathcal{P}_{\varphi_0}^{(1)} = \{m_+^{(1)}, m_-^{(1)} \}= \{ (m_0, m_1^+), (m_0, m_1^-) \}.
%		\qquad
%		\mathcal{P}_{\varphi_0}^{[2]} = \{m_0, m_2^+, m_2^-\}.
	\end{align*}
	
	\item[\Hypref{HypSolutionCondInd}] In order to guarantee the validity of this assumption, we have to set a new condition on the parameters of the model, that is $\beta_1=\beta_3=\beta$ and $\beta_2=\beta_4=\gamma$ (so that $\beta+\gamma= \frac{1}{4}$).
	It is sufficient to notice that, given a probability space $(\Omega, \mathcal{F}, \mathbb{P})$ endowed with a couple of independent random variables $\mu \sim \rho_2$, with $\rho_2= \rho \circ \pi_{\mathcal{P(X)}}^{-1}= \frac{1}{4}(\delta_{m_-}+\delta_{m_+}+\delta_{\widehat m_-}+\delta_{\widehat m_+})$, and $W \sim \nu$ and setting
	\begin{align}\label{Eq_ex_Phi_def_cond}
		\Phi 
		& = \mathbf{1}_{\{m_+\}}(\mu)(\mathbf{1}_{[0,4\beta]}(W)\varphi_+ + \mathbf{1}_{(4\beta, 1]}(W)\varphi_0 )
		+\mathbf{1}_{\{m_-\}}(\mu)(\mathbf{1}_{[0,4\beta]}(W)\varphi_- + \mathbf{1}_{(4\beta, 1]}(W)\varphi_0 )\nonumber \\
		& \qquad  +\mathbf{1}_{\{\widehat m_+\}}(\mu)(\mathbf{1}_{[0,4\beta]}(W)\widehat\varphi_+ + \mathbf{1}_{(4\beta, 1]}(W)\varphi_0 )
		+\mathbf{1}_{\{\widehat m_+\}}(\mu)(\mathbf{1}_{[0,4\beta]}(W)\widehat\varphi_+ + \mathbf{1}_{(4\beta, 1]}(W)\varphi_0 )\\ \nonumber
		&=:\alpha_1(W,\mu^{(2)}),
	\end{align}
	we have:
	\begin{itemize}
		\item $\mathbb{P}\circ (\Phi,\mu)^{-1}=\rho$.
		Indeed, exploiting the fact that $\mu$ is distributed according to $\rho_2$ and that $\Phi$ is defined via Equation \eqref{Eq_ex_Phi_def_cond}, for $(\varphi, \widetilde{m}) \in \mathfrak{F \times M}$, we have
		\begin{align*}
			\mathbb{P}((\Phi,\mu)=(\varphi,\widetilde{m}))
			 & = \sum_{m \in \mathfrak{M}} \mathbf{1}_{\{m\}} (\widetilde{m}) \mathbb{P}(\mu = m)\mathbb{P}(\Phi=\varphi | \mu= m)	\\
			%&= \frac{1}{4}\Bigg\{ \mathbf{1}_{\{m_+\}}(\widetilde m)\mathbb{P}(\mathbf{1}_{[0,4\beta]}(W)\varphi_+ + \mathbf{1}_{(4\beta, 1]}(W)\varphi_0 =\varphi)  
			%+\mathbf{1}_{\{m_-\}}(\widetilde m)\mathbb{P}(\mathbf{1}_{[0,4\beta]}(W)\varphi_- + \mathbf{1}_{(4\beta, 1]}(W)\varphi_0 =\varphi)\nonumber \\
			%& +\mathbf{1}_{\{\widehat m_+\}}(\widetilde m)\mathbb{P}(\mathbf{1}_{[0,4\beta]}(W)\widehat\varphi_+ + \mathbf{1}_{(4\beta, 1]}(W)\varphi_0 =\varphi)
			%+\mathbf{1}_{\{\widehat m_-\}}(\widetilde m)\mathbb{P}(\mathbf{1}_{[0,4\beta]}(W)\widehat\varphi_+ + \mathbf{1}_{(4\beta, 1]}(W)\varphi_0 =\varphi)\Bigg\}\\
			%%%%%%%%%%%%%%%%%
			&= \frac{1}{4}\Bigg\{ \mathbf{1}_{\{m_+\}}(\widetilde m)(4\beta \mathbf{1}_{\{\varphi_+\}}(\varphi) + 4\gamma\mathbf{1}_{\{\varphi_0\}}(\varphi))  
			+\mathbf{1}_{\{m_-\}}(\widetilde m)(4\beta \mathbf{1}_{\{\varphi_-\}}(\varphi) + 4\gamma\mathbf{1}_{\{\varphi_0\}}(\varphi))  \\
			& \qquad  +\mathbf{1}_{\{\widehat m_+\}}(\widetilde m)(4\beta \mathbf{1}_{\{\widehat \varphi_+\}}(\varphi) + 4\gamma\mathbf{1}_{\{\varphi_0\}}(\varphi)) 
			+\mathbf{1}_{\{\widehat m_-\}}(\widetilde m)(4\beta \mathbf{1}_{\{\widehat \varphi_-\}}(\varphi) + 4\gamma\mathbf{1}_{\{\varphi_0\}}(\varphi)) \Bigg\}\\
			 & = \rho(\varphi, \widetilde{m}).
		\end{align*}
		
		\item It holds that
		\begin{align*}
			\Phi(0,\cdot)
			& =  \mathbf{1}_{\{m_+\}}(\mu)\mathbf{1}_{[0,4\beta]}(W)\mathbf{1}_{\{1\}} 
		+\mathbf{1}_{\{m_-\}}(\mu)\mathbf{1}_{[0,4\beta]}(W)\mathbf{1}_{\{-1\}}\nonumber \\
			& \qquad +\mathbf{1}_{\{\widehat m_+\}}(\mu)\mathbf{1}_{[0,4\beta]}(W)\mathbf{1}_{\{1\}}  
		+\mathbf{1}_{\{\widehat m_-\}}(\mu)\mathbf{1}_{[0,4\beta]}(W)\mathbf{1}_{\{-1\}} \\
			% & =  (\mathbf{1}_{\{m_+\}} +\mathbf{1}_{\{\widehat m_+\}})(\mu)\mathbf{1}_{[0,4\beta]}(W)\mathbf{1}_{\{1\}} 
			%+(\mathbf{1}_{\{m_-\}}+\mathbf{1}_{\{\widehat m_-\}} )(\mu)\mathbf{1}_{[0,4\beta]}(W)\mathbf{1}_{\{-1\}}\\
			& = \mathbf{1}_{\{m_1^+\}}(\mu_1)\mathbf{1}_{[0,4\beta]}(W)\mathbf{1}_{\{1 \}}
		+\mathbf{1}_{\{m_1^-\}}(\mu_1)\mathbf{1}_{[0,4\beta]}(W)\mathbf{1}_{\{-1\}}\\
		& = \alpha_0(W, \mu_1),
		\end{align*}
		with $\alpha_0: \mathcal{Z}\times \mathcal{P(X)}^2 \to \mathcal{E}$, measurable function defined as
		\begin{displaymath}
			\alpha_0(w,m):=\mathbf{1}_{\{m^+_1\}}(m)\mathbf{1}_{[0,4\beta]}(w)\mathbf{1}_{\{1\}} +\mathbf{1}_{\{m^-_1\}}(m)\mathbf{1}_{[0,4\beta]}(w)\mathbf{1}_{\{-1\}},
		\end{displaymath}
				
		Hence,  the conditional independence property holds being equivalent to the existence of a $Z \sim \nu$ independent of $\mu$ s.t. $\Phi(0, \cdot)= u (Z, \mu^{(1)})$, with $u: \mathcal{Z}\times \mathcal{P(X)}^2 \to \mathcal{E}$, measurable function (see \cite[Proposition 6.13]{Kallenberg}).
	
	\end{itemize}
	
	\item[\hypref{HypPsiContinuity}] This is omitted being the same as in \cite{CF2022}.
	
	\item[\hypref{HypCosts}] Let us start by checking the Lipschitzianity of $f$.
	\begin{itemize}
		\item[-$t=0$:] for any $x_1,x_2 \in \mathcal{X}$, $\gamma_1, \gamma_2 \in \Gamma$ and $m_1, m_2 \in \mathcal{P(X)}$,
		\begin{align*}
			|f(0,x_1,\gamma_1,m_1)- f(0,x_2,\gamma_2,m_2)|	
			= c_0 |\gamma_1 - \gamma_2|
			= c_0 \text{d}(\gamma_1,\gamma_2);
		\end{align*}
		
		\item[-$t=1$:] for any $x_1,x_2 \in \mathcal{X}$, $\gamma_1, \gamma_2 \in \Gamma$ and $m_1, m_2 \in \mathcal{P(X)}$,
		\begin{align*}
			|f(1,x_1,\gamma_1,m_1)- f(1,x_2,\gamma_2,m_2)|	
			%& =  |c_1 \gamma_1 -x_1 \text{M}(m_1) - (c_1 \gamma_2  -x_2 \text{M}(m_2))|\\
			%& \leq c_1 |\gamma_1 - \gamma_2 | + |x_1|| \text{M}(m_1) - \text{M}(m_2))| + |\text{M}(m_2)| |x_1 - x_2|\\
			& \leq c_1 \text{d}(\gamma_1,\gamma_2) + 2 \text{dist}(m_1, m_2) + 2\text{d}(x_1, x_2). 
		\end{align*}
	\end{itemize}
	Now, for any $x_1,x_2 \in \mathcal{X}$ and $m_1, m_2 \in \mathcal{P(X)}$,
	\begin{align*}
		|F(x_1,m_1) - F(x_2, m_2)|
		%& = |x_1 \text{M}(m_1) -x_2 \text{M}(m_2)|
		%\leq  |x_1|| \text{M}(m_1) - \text{M}(m_2))| + |\text{M}(m_2)| |x_1 - x_2|\\
		& \leq  2\text{d}(x_1, x_2) + 2 \text{dist}(m_1, m_2).
	\end{align*}
	Hence, the validity of the last assumption follows from the choice $L= \max\{c_0, 4, c_1 + 4 \}= \max\{c_0, c_1 + 4 \}$.
\end{itemize}

%%%%%%%%%%%%%%%%%%%%%%%%%%%%%%%%%%%%%%%%%%%%%%%%%%

Now, let us write down some identities specific for the example that we are going to exploit in the following.
		Concerning the \emph{means} associated to the measure flows, we have 
		\begin{align}\label{Eq_ex_means_meas_flows}
			& \text{M}(m_0)=0, 
			\qquad
			\text{M}(m_1^+)
			= - \text{M}(m_1^-)
			= \frac{\beta_1}{4(\beta_1+\beta_2)} 
			= \beta, \nonumber\\
			& \text{M}( m_2^+) 
			= - \text{M}(m_2^-)
			= \frac{5\beta_1}{16(\beta_1+\beta_2)}
			= \frac{5}{8}\beta.		
		\end{align}	
		Then, set $\mathbb{P}_0(\cdot):= \mathbb{P}(\cdot|\Phi=\varphi_0)$ and, analogously, $\mathbb{E}_0[\cdot]:= \mathbb{E}[\cdot|\Phi=\varphi_0]$. 
		The distribution of the measure flow conditionally on the event $\{ \Phi = \varphi_0\}$ can be computed explicitly and it is given by
		\begin{align}\label{Eq_ex_distr_cond_mu}
			& \mathbb{P}_0\left(\mu^{(2)}=m_+\right)
			= \mathbb{P}_0\left(\mu^{(2)}=m_-\right)
			= \frac{\beta_2}{2(\beta_2+\beta_4)}
			= \frac{1}{4},\nonumber \\
			& \mathbb{P}_0\left(\mu^{(2)}=\widehat{m}_-\right)
			= \mathbb{P}_0\left(\mu^{(2)}=\widehat{m}_+\right)
			= \frac{\beta_4}{2(\beta_2+\beta_4)}
			= \frac{1}{4},
		\end{align}	
		and, setting $m_+^{(1)}:=(m_0,m_1^+)$ and $m_-^{(1)}:=(m_0,m_1^-)$, we have
		\begin{align}\label{Eq_ex_distr_cond_mu_1}
			& \mathbb{P}_0\left(\mu^{(1)}=m_+^{(1)}\right)
			= \mathbb{P}_0\left(\mu^{(1)}=m_-^{(1)}\right)
			= \frac{1}{2}.
		\end{align}
		Then, we compute the distribution of $\mu^{(2)}$ conditionally on $\mu^{(1)}$:
		\begin{align}\label{Eq_ex_distr_cond_mu_2_on_mu_1}
			& \mathbb{P}_0\left(\mu^{(2)}=m_+| \mu^{(1)}=m_+^{(1)}\right)
			= \mathbb{P}_0\left(\mu^{(2)}=m_-| \mu^{(1)}=m_-^{(1)}\right)
			= \frac{\beta_2}{\beta_2+\beta_4}
			= \frac{1}{2},\\
			& \mathbb{P}_0\left(\mu^{(2)}=\widehat{m}_+ | \mu^{(1)}=m_+^{(1)}\right)
			= \mathbb{P}_0\left(\mu^{(2)}=\widehat{m}_-| \mu^{(1)}=m_-^{(1)}\right)
			= \frac{\beta_4}{\beta_2+\beta_4}
			= \frac{1}{2}.
		\end{align}	

%%%%%%%%%%%%%%%%%%%%%%%%%%%%%%%%%%%%%%%%%%%%%%%%%%

	The conditions on parameters ensuring the optimality of $\rho$ are presented in the following result.

	\begin{proposition}
	Consider the MFG setting described above. Then, 
		\begin{align}\label{Eq_ex_def_rho_rest}
			\rho =  \beta (\delta_{(\varphi_+,m_+)}+\delta_{(\varphi_-,m_-)} + \delta_{(\widehat \varphi_+,\widehat{m}_+)}+\delta_{(\widehat \varphi_-,\widehat{m}_-)} )+\gamma(\delta_{(\varphi_0,m_+)} + \delta_{(\varphi_0,m_-)} + \delta_{(\varphi_0,\widehat{m}_+)} +\delta_{(\varphi_0,\widehat{m}_-)}),
		\end{align}	 
	is optimal provided that
		\begin{itemize}
	
			\item[i)] $\beta, \gamma \in [0,1]$  and $\beta+\gamma=\frac{1}{4}$,
	
			\item[ii)] 
			$
				0 < c_0 < \frac{\beta}{2},
			$
			\item[ii)] 
			$
				\frac{5}{32} \beta < c_1 < \frac{5}{16} \beta.
			$	
		\end{itemize}
	\end{proposition}
	
	\begin{remark}
		Under the assumption that $\beta_1 = \beta_3 = \beta$ and $\beta_2 = \beta_4 = \gamma$, which we have previously set to ensure the validity of \Hypref{HypSolutionCondInd}, the consistency property is automatically satisfied.
		Furthermore, under the stronger conditions in the Proposition above, there are still infinitely many correlated solutions but we loose a degree of freedom w.r.t. the result in \cite{CF2022}.
	\end{remark}
	
	\begin{proof}		
		In this simplified context the set of strategy modifications maps the set $\mathfrak{F}$ into
		\begin{align}\label{Eq_ex_W_T}
			\widehat{\mathcal{R}} = \{ \psi: \{0,1\}\times \mathcal{X}^3\times \mathfrak{M} \to \Gamma, \text{ progressively measurable}\},
		\end{align}
		that is, for any $w \in \widehat{\mathcal{D}}$ and for any $\varphi \in \mathfrak{F}$,  $w(\varphi)(0,(x_0,x_1,x_2), (m_0,m_1, m_2))= w(\varphi)(0,x_0, m_0)$ and  $w(\varphi)(1,(x_0,x_1,x_2), (m_0,m_1, m_2))$ $=$ $ w(\varphi)(1,(x_0,x_1), (m_0,m_1))$.
		In order to find the conditions on the parameters in the definition of $\rho$ in Equation \eqref{Eq_ex_def_rho_rest} ensuring that it is a solution in the MFG, we rewrite the cost functional exploiting desintegration over sets of the form $\{\Phi = \varphi\}$, with $\varphi \in \mathfrak{F} $,
		\begin{align*}
		J(m_0, \rho, w)
		& = \mathbb{E}\Big[c_0 w(\Phi)(0, X_0,m_0)+ c_1 w(\Phi)(1,(X_0,X_1),(m_0,\mu_1)) - X_1\text{M}(\mu_1)-X_2\text{M}(\mu_2)\Big]\nonumber\\
		%%%%%%%%%%%%%%%
		& = \beta\Bigg\{c_0 \mathbb{E}_+\big[w(\varphi_+)(0, X_0,m_0)\big]+ c_1 \mathbb{E}_+\big[w(\varphi_+)(1, (X_0,X_1),(m_0,m_1^+))\big] \nonumber\\
		& \quad - \mathbb{E}_+\big[X_1\big]\text{M}(m_1^+)-\mathbb{E}_+\big[X_2\big]\text{M}(m_2^+)\Bigg\} + \beta \Bigg\{c_0 \widehat{\mathbb{E}}_+\big[w(\widehat\varphi_+)(0, X_0,m_0)\big]\nonumber\\
		& \quad + c_1  \widehat{\mathbb{E}}_+\big[w(\widehat\varphi_+)(1, (X_0,X_1),(m_0,m_1^+))\big] -  \widehat{\mathbb{E}}_+\big[\widehat X_1\big]\text{M}(m_1^+)- \widehat{\mathbb{E}}_+\big[\widehat X_2\big]\text{M}(m_0)\Bigg\}\nonumber\\
		& \quad +\beta \Bigg\{c_0 \mathbb{E}_-\big[w(\varphi_-)(0, X_0,m_0)\big]+ c_1 \mathbb{E}_-\big[w(\varphi_-)(1,(X_0,X_1),(m_0,m_1^-))\big] \nonumber\\
		& \quad - \mathbb{E}_-\big[X_1\big]\text{M}(m_1^-)-\mathbb{E}_-\big[X_2\big]\text{M}(m_2^-)\Bigg\} + \beta \Bigg\{c_0  \widehat{\mathbb{E}}_-\big[w(\widehat\varphi_-)(0, X_0,m_0)\big] \nonumber\\
		& \quad + c_1 \widehat{\mathbb{E}}_-\big[w(\widehat\varphi_-)(1, (X_0,X_1),(m_0,m_1^-))\big] - \widehat{\mathbb{E}}_-\big[X_1\big]\text{M}(m_1^-)-\widehat{\mathbb{E}}_-\big[X_2\big]\text{M}(m_0)\Bigg\}\nonumber\\
		& \quad + 4 \gamma \Bigg\{c_0 \mathbb{E}_0\big[w(\varphi_0)(0, X_0,m_0)\big]+ c_1 \mathbb{E}_0\big[w(\varphi_0)(1, (X_0,X_1),(m_0,\mu_1))\big] \nonumber\\
		& \quad - \mathbb{E}_0\big[X_1\text{M}(\mu_1)\big]-\mathbb{E}_0\big[X_2\text{M}(\mu_2)\big]\Bigg\},
		\end{align*}
		where we have exploited the fact that the conditioning on $\{\Phi=\varphi\}$, with $\varphi \in \{ \varphi_+, \varphi_-, \widehat \varphi_+, \widehat \varphi_-\}$, completely determines the measure flow as well. Notice that the notation $\mathbb{E}_+$ (resp. $\mathbb{E}_-$, $\widehat{\mathbb{E}}_+$, $\widehat{\mathbb{E}}_-$ and $\mathbb{E}_0$) was introduced to denote conditional expectation w.r.t. the event $\{ \Phi=\varphi_+\}$ (resp. $\varphi_-, \widehat \varphi_+, \widehat \varphi_-$ and $\varphi_0$). 
		Before proceeding with the study of the different cases we make a useful remark.
		\\
		
		Consider the probability space $(\Omega, \mathcal{F}, \mathbb{P}_\varphi)$, where $\mathbb{P}_\varphi(\cdot)= \mathbb{P}(\cdot| \Phi = \varphi)$, with $\varphi \in \mathfrak{F}$. For any $t \in [\![0, T-1]\!]$, $X^{(t)}$ and $(\mu_{t+1}, \dots, \mu_T)$ are conditionally independent given $\mu^{(t)}$. Indeed, for any $m \in \mathcal{P(X)}^{T-t}$, $x \in \mathcal{X}^{t+1}$, exploiting in sequence the tower property, the measurability of $X^{(t)}$ w.r.t. $\sigma(X_0, \xi_1, \dots, \xi_t, \Phi, \mu^{(t)})$, the joint independence of $\mu$ from $X_0$ and $\xi_1, \dots, \xi_T$, and the measurability of conditional expectations, we have
				\begin{align*}
					\mathbb{P}_\varphi ((\mu_{t+1}, \dots, \mu_T) = m, X^{(t)}=x | \mu^{(t)})
					 & = \mathbb{E}_\varphi [\mathbf{1}_{\{m\}}(\mu_{t+1}, \dots, \mu_T) \mathbf{1}_{\{x\}}(X^{(t)}) | \mu^{(t)}]\\
				%	& = \mathbb{E}_\varphi [\mathbb{E}_\varphi [\mathbf{1}_{\{m\}}(\mu_{t+1}, \dots, \mu_T) \mathbf{1}_{\{x\}}(X^{(t)}) | \mu^{(t)}, X_0,\xi_1,\dots,\xi_t]| \mu^{(t)}]\\
					& = \mathbb{E}_\varphi [ \mathbf{1}_{\{x\}}(X^{(t)})\mathbb{E}_\varphi [ \mathbf{1}_{\{m\}}(\mu_{t+1}, \dots, \mu_T)| \mu^{(t)}, X_0,\xi_1,\dots,\xi_t]| \mu^{(t)}]\\
					& = \mathbb{E}_\varphi [ \mathbf{1}_{\{x\}}(X^{(t)})\mathbb{E}_\varphi [ \mathbf{1}_{\{m\}}(\mu_{t+1}, \dots, \mu_T)| \mu^{(t)}]| \mu^{(t)}]\\
					& = \mathbb{E}_\varphi [ \mathbf{1}_{\{m\}}(\mu_{t+1}, \dots, \mu_T)| \mu^{(t)}]\mathbb{E}_\varphi [ \mathbf{1}_{\{x\}}(X^{(t)})| \mu^{(t)}]\\
					& = \mathbb{P}_\varphi ((\mu_{t+1}, \dots, \mu_T) = m| \mu^{(t)}) \mathbb{P}_\varphi ( X^{(t)}=x | \mu^{(t)}).
				\end{align*}

		%%%%%%%%%%%%%%%%%%%%%%%%%%%%%%%%%%%%%%%%%%%%%%%%%%%%%%%		
		Now, let's start by discussing the first case, that is when the suggestion is  $\{\Phi = \varphi_+\}$.
		We proceed exploiting the DPP (Proposition \ref{Prop_cond_DPP_CMFG}).
		In the following we omit the dependency on the measure flow being it identically equal to a single element and we introduce the following simplified notations: $V_+:= V_{\varphi_+}$, $V_-:= V_{\varphi_-}$, $\widehat{V}_+:= V_{\widehat{\varphi}_+}$, $\widehat{V}_-:= V_{\widehat{\varphi}_-}$ and $V_0:= V_{\varphi_0}$.
		\begin{itemize}
			\item For $t=2$, $x \in \{-1, 1\}^3$,
			\begin{align*}
				& V_+(2, (x_0,x_1,1)) = F(1, m_2^+)= - \text{M}(m_2^+) = -\frac{5}{4}\beta,\\
				& V_+(2, (x_0,x_1,-1)) = F(-1, m_2^+)= + \text{M}(m_2^+) = \frac{5}{4}\beta,
			\end{align*}
			
			\item For $t=1$, $x \in \{-1, 1\}^2$,
			\begin{align*}
				V_+(1, (x_0,-1)) 
				& = \min_{\gamma \in \{0,1\}} \Bigg\{ c_1 \gamma + \text{M}(m_1^+) + \mathbb{E}_+\left[V_+\left(2, (x_0,-1,\Psi(-1,\gamma,\xi_2))\right)\right] \Bigg\} \\
				& = \min_{\gamma \in \{0,1\}} \Bigg\{ c_1 \gamma + \text{M}(m_1^+) + \text{M}(m_2^+)\left[\mathbb{P}_+\left(\Psi(-1,\gamma,\xi_2)=-1\right) - \mathbb{P}_+\left(\Psi(-1,\gamma,\xi_2)=1\right) \right] \Bigg\}\\
				& =  \text{M}(m_1^+) + \min \Bigg\{ \text{M}(m_2^+)\left(\frac{1}{2}-\frac{1}{2}\right), c_1 + \text{M}(m_2^+)\left( -\frac{1}{4}+ \frac{3}{4} \right) \Bigg\} \\
				& = \beta + \min \Bigg\{0, c_1 + \frac{5}{16}\beta \Bigg\}.		
			\end{align*}
			This implies that, at time $t=1$ in state $(x_0,-1)$, $\gamma=0$ is optimal which corresponds to $\varphi_+$ evaluated at $t=1, x=-1$.
			Analogously,
			\begin{align*}
				V_+(1, (x_0,1)) 
				& = \min_{\gamma \in \{0,1\}} \Bigg\{ c_1 \gamma - \text{M}(m_1^+) + \mathbb{E}_+\left[V_+\left(2, (x_0,1,\Psi(1,\gamma,\xi_2))\right)\right] \Bigg\} \\
				%& = \min_{\gamma \in \{0,1\}} \Bigg\{ c_1 \gamma - \text{M}(m_1^+) + \text{M}(m_2^+)\left[\mathbb{P}_+\left(\Psi(1,\gamma,\xi_2)=-1\right) - \mathbb{P}_+\left(\Psi(1,\gamma,\xi_2)=1\right) \right] \Bigg\}\\
				%& = - \text{M}(m_1^+) + \min \Bigg\{ \text{M}(m_2^+)\left(\frac{1}{2}-\frac{1}{2}\right), c_1 + \text{M}(m_2^+)\left( \frac{1}{4}- \frac{3}{4} \right) \Bigg\} \\
				& = - \beta + \min \Bigg\{0, c_1 - \frac{5}{16}\beta \Bigg\} .			
			\end{align*}
			This implies that $\gamma=1$ (and so $\varphi_+$) is optimal at time $t=1$ and state $(x_0,1)$ if and only if	 $c_1 - \frac{5}{16}\beta < 0$, that is
			\begin{align}\label{Eq_ex_u_cond_on_c1}
				0 < c_1 < \frac{5}{16}\beta.
			\end{align}
			
			\item For $t=0$, $x \in \{-1, 1\}$,
			\begin{align*}
				V_+(0, -1) 
				& = \min_{\gamma \in \{0,1\}} \Bigg\{ c_0 \gamma + \mathbb{E}_+\left[V_+\left(1, (-1,\Psi(-1,\gamma,\xi_2))\right)\right] \Bigg\} \\
				& = \min_{\gamma \in \{0,1\}} \Bigg\{ c_0 \gamma  + \left( - \beta  + c_1 - \frac{5}{16}\beta \right)\mathbb{P}_+\left(\Psi(-1,\gamma,\xi_2)=1\right) + \beta \mathbb{P}_+\left(\Psi(-1,\gamma,\xi_2)=-1\right)  \Bigg\}\\
				& = \min  \Bigg\{0 + \frac{1}{2}\left( - \beta  + c_1 - \frac{5}{16}\beta \right) + \frac{1}{2}\beta,  c_0  + \left( - \beta  + c_1 - \frac{5}{16}\beta \right) \frac{1}{4} + \beta  \frac{3}{4}  \Bigg\} \\
				& = \min  \Bigg\{ \frac{1}{2}\left(  c_1 - \frac{5}{16}\beta \right),  c_0  + \left( c_1 - \frac{5}{16}\beta \right) \frac{1}{4} + \beta  \frac{1}{2}  \Bigg\}.		
			\end{align*}
			Since $c_1 - \frac{5}{16}\beta < 0$ and all the parameters are positive, at time $t=0$ in state $x_0=-1$, $\gamma=0$ is optimal which corresponds to $\varphi_+$ evaluated at $t=0, x=-1$.
			Analogously,
			\begin{align*}
				V_+(0, 1) 
				& = \min_{\gamma \in \{0,1\}} \Bigg\{ c_0 \gamma + \mathbb{E}_+\left[V_+\left(1, (1,\Psi(1,\gamma,\xi_2))\right)\right] \Bigg\} \\
		%		& = \min_{\gamma \in \{0,1\}} \Bigg\{ c_0 \gamma  + \left( - \beta  + c_1 - \frac{5}{16}\beta \right)\mathbb{P}_+\left(\Psi(1,\gamma,\xi_2)=1\right) + \beta \mathbb{P}_+\left(\Psi(1,\gamma,\xi_2)=-1\right)  \Bigg\}\\
		%		& = \min  \Bigg\{0 + \frac{1}{2}\left( - \beta  + c_1 - \frac{5}{16}\beta \right) + \frac{1}{2}\beta,  c_0  + \left( - \beta  + c_1 - \frac{5}{16}\beta \right) \frac{3}{4} + \beta  \frac{1}{4}  \Bigg\} \\
				& = \min  \Bigg\{ \frac{1}{2}\left(  c_1 - \frac{5}{16}\beta \right),  c_0  + \left( c_1 - \frac{5}{16}\beta \right) \frac{3}{4} - \beta  \frac{1}{2}  \Bigg\}.		
			\end{align*}
			This implies that $\gamma=1$ (and so $\varphi_+$) is optimal at time $t=0$ and state $1$ if and only if	 $\frac{1}{2}\left(  c_1 - \frac{5}{16}\beta \right) > c_0  + \left( c_1 - \frac{5}{16}\beta \right) \frac{3}{4} - \beta  \frac{1}{2}$. Since we have already set $c_1 < \frac{5}{16}\beta$, we set the following stronger condition that guarantees the validity of the inequality above
			\begin{align}\label{Eq_ex_u_cond_on_c0}
				0 < c_0 < \frac{1}{2}\beta.
			\end{align}
		\end{itemize}
		Hence, we have shown that, conditionally on the event $\{\Phi = \varphi_+\}$, $\varphi_+$ is optimal.
		
		%%%%%%%%%%%%%%%%%%%%%%%%%%%%%%%%%%%%%%%%%%%%%%%
		
		The case $\{\Phi = \varphi_-\}$ is completely analogous and leads to the same constraints on the coefficients.
		
		%%%%%%%%%%%%%%%%%%%%%%%%%%%%%%%%%%%%%%%%%%%%%%%%%%%
		
		 Now, let's  discuss in details the case in which the suggestion is  $\{\Phi = \widehat \varphi_+\}$.
		\begin{itemize}
			\item For $t=2$, $x \in \{-1, 1\}^3$,
			\begin{align*}
				& \widehat V_+(2, (x_0,x_1,x_2)) = F(x_2, m_0)= -x_2 \text{M}(m_0) = 0,
			\end{align*}
			
			\item For $t=1$, $x \in \{-1, 1\}^2$,
			\begin{align*}
				\widehat V_+(1, (x_0,x_1)) 
				& = \min_{\gamma \in \{0,1\}} \Bigg\{ c_1 \gamma - x_1 \text{M}(m_1^+) + \widehat{\mathbb{E}}_+\left[\widehat V_+\left(2, (x_0,x_1,\Psi(x_1,\gamma,\xi_2))\right)\right] \Bigg\} \\
				& = \min_{\gamma \in \{0,1\}} \Bigg\{ c_1 \gamma - x_1 \text{M}(m_1^+) \Bigg\}\\
				& =  - x_1 \text{M}(m_1^+) + c_1\min \Big\{0, \gamma\Big\} 
				= - x_1 \text{M}(m_1^+).
			\end{align*}
			This implies that at time $t=1$, in any state $(x_0,x_1)$, $\gamma=0$ is optimal which corresponds to $\widehat \varphi_+$ evaluated at $t=1$.
					
			\item For $t=0$, $x \in \{-1, 1\}$,
			\begin{align*}
				\widehat V_+(0, -1) 
				& = \min_{\gamma \in \{0,1\}} \Bigg\{ c_0 \gamma + \widehat{\mathbb{E}}_+\left[\widehat V_+\left(1, (-1,\Psi(-1,\gamma,\xi_2))\right)\right] \Bigg\} \\
				& = \min_{\gamma \in \{0,1\}} \Bigg\{ c_0 \gamma   +( -\beta) \widehat{\mathbb{E}}_+\left[\Psi(-1,\gamma,\xi_2)\right]  \Bigg\}\\
				&
				%= \min  \Bigg\{-\beta\left(\frac{1}{2} - \frac{1}{2}\right),  c_0  -\beta  \left(\frac{1}{4}- \frac{3}{4}\right)  \Bigg\}\\
				= \min  \Bigg\{0,  c_0  + \frac{\beta}{2} \Bigg\}.		
			\end{align*}
			At time $t=0$ in state $x_0=-1$, $\gamma=0$ is optimal which corresponds to $\widehat \varphi_+$ evaluated at $t=0, x=-1$.
			Analogously,
			\begin{align*}
				\widehat V_+(0, 1) 
				& = \min_{\gamma \in \{0,1\}} \Bigg\{ c_0 \gamma + \widehat{\mathbb{E}}_+\left[\widehat V_+\left(1, (1,\Psi(1,\gamma,\xi_2))\right)\right] \Bigg\} 
				%& = \min_{\gamma \in \{0,1\}} \Bigg\{ c_0 \gamma   +( -\beta) \widehat{\mathbb{E}}_+\left[\Psi(1,\gamma,\xi_2)\right]  \Bigg\}
				%= \min  \Bigg\{-\beta\left(\frac{1}{2} - \frac{1}{2}\right),  c_0  -\beta  \left(-\frac{1}{4}+ \frac{3}{4}\right)  \Bigg\}\\
				= \min  \Bigg\{0,  c_0  - \frac{\beta}{2} \Bigg\}.		
			\end{align*}
			The condition that we have set in Equation \eqref{Eq_ex_u_cond_on_c0} yields that $\gamma=1$ (and so $\widehat \varphi_+$) is optimal at time $t=0$ and state $1$. 
		Hence, we have checked that, conditionally on the event $\{\Phi = \widehat \varphi_+\}$, $\widehat \varphi_+$ is optimal.
		\end{itemize}
		
		%%%%%%%%%%%%%%%%%%%%%%%%%%%%%%%%%%%%%%%%%%%%%
		
		The computations for the case $\{ \Phi = \widehat{\varphi}_- \}$ are analogous and lead to the same constraints.
		
		%%%%%%%%%%%%%%%%%%%%%%%%%%%%%%%%%%%%%%%%%%%%%%
			
		The last case, namely $\{\Phi=\varphi_0\}$, is the most complicated. Indeed, in this case we have to handle a random measure flow and consequently different flows of measure and different outcomes when evaluating the strategies of the representative player.
		This is done exploiting again the dynamic programming principle.
		\begin{itemize}
			\item For $t=2$, $x \in \{1,-1\}^3$, $(m_0,m_1,m_2) \in \{ m_+, \widehat{m}_+, m_-, \widehat{m}_- \} = \mathcal{D}_{\varphi_0}$,
			\begin{align}
				V_0(2,(x_0,x_1,x_2),(m_0,m_1,m_2))= -x_2 \text{M}(m_2).
			\end{align}
			In particular, we have 
			\begin{align*}
				& V_0(2,(x_0,x_1,1),m_+)
				= V_0(2,(x_0,x_1,-1),m_-)
				= -\text{M}(m_2^+)
				= -\frac{5}{8}\beta,\\
				& V_0(2,(x_0,x_1,-1),m_+)
				= V_0(2,(x_0,x_1,1),m_-)
				= \text{M}(m_2^+)
				= \frac{5}{8}\beta,\\
				&V_0(2,(x_0,x_1,x_2),(m_0,m_1,m_0))
				= 0.
			\end{align*}
			
			\item For $t=1$, $x \in \{1,-1\}^2$, $(m_0,m_1) \in \{m_+^{(1)}, m_-^{(1)} \} = \mathcal{D}_{\varphi_0}^{(1)} $,
			\begin{align}
				V_0 &(1,(x_0,x_1),(m_0,m_1))
				 = \min_{\gamma \in \{0,1\}} \Big\{ c_1 \gamma -x \text{M}(m_1) \nonumber \\
				& + \mathbb{E}_0 \left[ V_0(2, (x_0,x_1,\Psi(x_1,\gamma,\xi_2),(m_0,m_1,\mu_2))) | X^{(1)}=(x_0,x_1), \mu^{(1)}= (m_0,m_1)\right]\Big\}.
			\end{align}
			Exploiting the computations at the previous step, the fact that $\xi_2$ and $(\Phi, \mu, X_0, \xi_1)$ are independent, and the fact that, on the probability space $(\Omega, \mathcal{F}, \mathbb{P}_{\varphi_0})$, $X^{(1)}$ and $\mu_2$ are conditionally independent given $\mu^{(1)}$, we have
		
		\begin{align*}
			V_0& (1,(x_0,1),m_+^{(1)})
			 = \min_{\gamma \in \{0,1\}} \bigg\{ c_1 \gamma - \text{M}(m_1^+) \\
			 & \qquad + \mathbb{E}_0 \left[ V_0(2,(x_0,1,\Psi(1,\gamma,\xi_2)), (m_+^{(1)},\mu^{(2)}) ) | X^{(1)}=(x_0,1),\mu^{(1)}= m_+^{(1)}\right]\bigg\} \nonumber \\
			%%%%%%%%%%%%%%%%%%%%%%%%%
			& = - \text{M}(m_1^+) + \min_{\gamma \in \{0,1\}} \Big\{ c_1 \gamma + \text{M}(m_2^+)\bigg[
			\mathbb{P}_0(\Psi(1,\gamma,\xi_2)=-1,\mu=m_2^+  | X^{(1)}=(x_0,1),\mu^{(1)}= m_+^{(1)})\nonumber \\
			& \qquad - \mathbb{P}_0(\Psi(1,\gamma,\xi_2)=1,\mu=m_2^+  | X^{(1)}=(x_0,1),\mu^{(1)}= m_+^{(1)})\bigg]	\Big\}\nonumber \\
			%%%%%%%%%%%%%%%%%%%%%%%%%
			& = - \text{M}(m_1^+) + \min_{\gamma \in \{0,1\}} \Big\{ c_1 \gamma + \text{M}(m_2^+)\bigg[
			\mathbb{P}_0(\Psi(1,\gamma,\xi_2)=-1)\mathbb{P}_0(\mu=m_2^+  | X^{(1)}=(x_0,1),\mu^{(1)}= m_+^{(1)})\nonumber \\
			& \qquad - \mathbb{P}_0(\Psi(1,\gamma,\xi_2)=1)\mathbb{P}_0(\mu=m_2^+  | X^{(1)}=(x_0,1),\mu^{(1)}= m_+^{(1)})\bigg]	\Big\}\nonumber \\
			%%%%%%%%%%%%%%%%%%%%%%%%%
			& = - \beta + \min_{\gamma \in \{0,1\}} \Big\{ c_1 \gamma + \frac{5}{8}\beta\bigg[
			\mathbb{P}_0(\Psi(1,\gamma,\xi_2)=-1)\mathbb{P}_0(\mu=m_2^+  | \mu^{(1)}= m_+^{(1)})\nonumber \\
			& \qquad - \mathbb{P}_0(\Psi(1,\gamma,\xi_2)=1)\mathbb{P}_0(\mu=m_2^+  | \mu^{(1)}= m_+^{(1)})\bigg]	\Big\}\nonumber \\
			%%%%%%%%%%%%%%%%%%%%%%%%%
			& = - \beta + \min \Big\{0 + \frac{5}{16}\beta \left[\frac{1}{2}-\frac{1}{2}\right], c_1 + \frac{5}{16}\beta \left[\frac{1}{4}-\frac{3}{4}\right]	\Big\}
			= - \beta + \min \Big\{0, c_1 -\frac{5}{32}\beta 	\Big\}
		\end{align*}
		and, similarly,
		
		\begin{align*}
			V_0& (1,(x_0,-1),m_-^{(1)})
			 = \min_{\gamma \in \{0,1\}} \bigg\{ c_1 \gamma + \text{M}(m_1^-) \\
			 & \qquad + \mathbb{E}_0 \left[ V_0(2,(x_0,-1,\Psi(-1,\gamma,\xi_2)), (m_-^{(1)},\mu^{(2)}) ) | X^{(1)}=(x_0,-1),\mu^{(1)}= m_-^{(1)}\right]\bigg\} \nonumber \\
			%%%%%%%%%%%%%%%%%%%%%%%%%
		%	& = - \text{M}(m_1^+) + \min_{\gamma \in \{0,1\}} \Big\{ c_1 \gamma + \text{M}(m_2^+)\bigg[ \mathbb{P}_0(\Psi(-1,\gamma,\xi_2)=1,\mu=m_2^-  | X^{(1)}=(x_0,-1),\mu^{(1)}= m_-^{(1)})\nonumber \\
		%	& \qquad  - \mathbb{P}_0(\Psi(-1,\gamma,\xi_2)=-1,\mu=m_2^-  | X^{(1)}=(x_0,-1),\mu^{(1)}= m_-^{(1)})\bigg]	\Big\}\nonumber \\
			%%%%%%%%%%%%%%%%%%%%%%%%%
		%	& = - \text{M}(m_1^+) + \min_{\gamma \in \{0,1\}} \Big\{ c_1 \gamma + \text{M}(m_2^+)\bigg[ \mathbb{P}_0(\Psi(-1,\gamma,\xi_2)=1)\mathbb{P}_0(\mu=m_2^-  | X^{(1)}=(x_0,-1),\mu^{(1)}= m_-^{(1)})\nonumber \\
		%	& \qquad  - \mathbb{P}_0(\Psi(-1,\gamma,\xi_2)=-1)\mathbb{P}_0(\mu=m_2^-  | X^{(1)}=(x_0,- 1),\mu^{(1)}= m_-^{(1)})\bigg]	\Big\}\nonumber \\
			%%%%%%%%%%%%%%%%%%%%%%%%%
		%	& = - \beta + \min_{\gamma \in \{0,1\}} \Big\{ c_1 \gamma + \frac{5}{8}\beta\bigg[ \mathbb{P}_0(\Psi(-1,\gamma,\xi_2)=1)\mathbb{P}_0(\mu=m_2^-  | \mu^{(1)}= m_-^{(1)})\nonumber \\
		%	& \qquad  - \mathbb{P}_0(\Psi(-1,\gamma,\xi_2)=-1)\mathbb{P}_0(\mu=m_2^-  | \mu^{(1)}= m_-^{(1)})\bigg]	\Big\}\nonumber \\
			%%%%%%%%%%%%%%%%%%%%%%%%%
			& = - \beta + \min \Big\{0 + \frac{5}{16}\beta \left[\frac{1}{2}-\frac{1}{2}\right], c_1 + \frac{5}{16}\beta \left[\frac{1}{4}-\frac{3}{4}\right]	\Big\}
			= - \beta + \min \Big\{0, c_1 -\frac{5}{32}\beta 	\Big\}.
		\end{align*}
		This yields that $\gamma=0$ (and so $\varphi_0$) is optimal at time $t=0$ when $(x,m)\in \{((x_0,1),(m_0,m_1^+)),$ $((x_0,-1),(m_0,m_1^-))\}$ if and only if $c_1 -\frac{5}{32}\beta  > 0$.
		Thus, we set the condition
		\begin{align}\label{Eq_ex_l_cond_c1}
			\frac{5}{32}\beta < c_1.
		\end{align}
		
		Analogously, we compute 
		\begin{align*}
			V_0& (1,(x_0,-1),m_+^{(1)}) = \min_{\gamma \in \{0,1\}} \bigg\{ c_1 \gamma + \text{M}(m_1^+)  \\
			 & \qquad + \mathbb{E}_0 \left[ V_0(2,(x_0,-1,\Psi(-1,\gamma,\xi_2)), (m_+^{(1)},\mu^{(2)}) ) | X^{(1)}=(x_0,-1),\mu^{(1)}= m_+^{(1)}\right]\bigg\} \nonumber \\
			%%%%%%%%%%%%%%%%%%%%%%%%%
		%	& =  \text{M}(m_1^+) + \min_{\gamma \in \{0,1\}} \Big\{ c_1 \gamma + \text{M}(m_2^+)\bigg[ \mathbb{P}_0(\Psi(-1,\gamma,\xi_2)=-1,\mu=m_2^+  | X^{(1)}=(x_0,-1),\mu^{(1)}= m_+^{(1)})\nonumber \\
		%	&  - \mathbb{P}_0(\Psi(-1,\gamma,\xi_2)=1,\mu=m_2^+  | X^{(1)}=(x_0,-1),\mu^{(1)}= m_+^{(1)})\bigg]	\Big\}\nonumber \\
			%%%%%%%%%%%%%%%%%%%%%%%%%
		%	& = \text{M}(m_1^+) + \min_{\gamma \in \{0,1\}} \Big\{ c_1 \gamma + \text{M}(m_2^+)\bigg[ \mathbb{P}_0(\Psi(-1,\gamma,\xi_2)=-1)\mathbb{P}_0(\mu=m_2^+  | X^{(1)}=(x_0,-1),\mu^{(1)}= m_+^{(1)})\nonumber \\
		%	&  - \mathbb{P}_0(\Psi(-1,\gamma,\xi_2)=1)\mathbb{P}_0(\mu=m_2^+  | X^{(1)}=(x_0,-1),\mu^{(1)}= m_+^{(1)})\bigg]	\Big\}\nonumber \\
			%%%%%%%%%%%%%%%%%%%%%%%%%
		%	& = \beta + \min_{\gamma \in \{0,1\}} \Big\{ c_1 \gamma + \frac{5}{8}\beta\bigg[\mathbb{P}_0(\Psi(-1,\gamma,\xi_2)=-1)\mathbb{P}_0(\mu=m_2^+  | \mu^{(1)}= m_+^{(1)})\nonumber \\
		%	&  - \mathbb{P}_0(\Psi(-1,\gamma,\xi_2)=1)\mathbb{P}_0(\mu=m_2^+  | \mu^{(1)}= m_+^{(1)})\bigg]	\Big\}\nonumber \\
			%%%%%%%%%%%%%%%%%%%%%%%%%
			& = \beta + \min \Big\{0 + \frac{5}{16}\beta \left[\frac{1}{2}-\frac{1}{2}\right], c_1 + \frac{5}{16}\beta \left[\frac{3}{4}-\frac{1}{4}\right]	\Big\} = \beta + \min \Big\{0, c_1 +\frac{5}{32}\beta 	\Big\},
		\end{align*}
		and 
		\begin{align*}
			V_0& (1,(x_0,1),m_-^{(1)})
			 = \min_{\gamma \in \{0,1\}} \bigg\{ c_1 \gamma + \text{M}(m_1^-) \\
			 & \qquad + \mathbb{E}_0 \left[ V_0(2,(x_0,1,\Psi(1,\gamma,\xi_2)), (m_-^{(1)},\mu^{(2)}) ) | X^{(1)}=(x_0,1),\mu^{(1)}= m_-^{(1)}\right]\bigg\} \nonumber \\
			%%%%%%%%%%%%%%%%%%%%%%%%%
		%	& = \text{M}(m_1^+) + \min_{\gamma \in \{0,1\}} \Big\{ c_1 \gamma + \text{M}(m_2^+)\bigg[ \mathbb{P}_0(\Psi(1,\gamma,\xi_2)=1,\mu=m_2^-  | X^{(1)}=(x_0,1),\mu^{(1)}= m_-^{(1)})\nonumber \\
		%	&  - \mathbb{P}_0(\Psi(1,\gamma,\xi_2)=-1,\mu=m_2^-  | X^{(1)}=(x_0,1),\mu^{(1)}= m_-^{(1)})\bigg]	\Big\}\nonumber \\
			%%%%%%%%%%%%%%%%%%%%%%%%%
		%	& = \text{M}(m_1^+) + \min_{\gamma \in \{0,1\}} \Big\{ c_1 \gamma + \text{M}(m_2^+)\bigg[ \mathbb{P}_0(\Psi(1,\gamma,\xi_2)=1)\mathbb{P}_0(\mu=m_2^-  | X^{(1)}=(x_0,1),\mu^{(1)}= m_-^{(1)})\nonumber \\
		%	&  - \mathbb{P}_0(\Psi(1,\gamma,\xi_2)=-1)\mathbb{P}_0(\mu=m_2^-  | X^{(1)}=(x_0, 1),\mu^{(1)}= m_-^{(1)})\bigg]	\Big\}\nonumber \\
			%%%%%%%%%%%%%%%%%%%%%%%%%
		%	& = \beta + \min_{\gamma \in \{0,1\}} \Big\{ c_1 \gamma + \frac{5}{8}\beta\bigg[ \mathbb{P}_0(\Psi(1,\gamma,\xi_2)=1)\mathbb{P}_0(\mu=m_2^-  | \mu^{(1)}= m_-^{(1)})\nonumber \\
		%	&  - \mathbb{P}_0(\Psi(1,\gamma,\xi_2)=-1)\mathbb{P}_0(\mu=m_2^-  | \mu^{(1)}= m_-^{(1)})\bigg]	\Big\}\nonumber \\
			%%%%%%%%%%%%%%%%%%%%%%%%%
			& = \beta + \min \Big\{0 + \frac{5}{16}\beta \left[\frac{1}{2}-\frac{1}{2}\right], c_1 + \frac{5}{16}\beta \left[\frac{3}{4}-\frac{1}{4}\right]	\Big\}
			= \beta + \min \Big\{0, c_1 +\frac{5}{32}\beta 	\Big\}
		\end{align*}
		Thus, $\gamma=0$ (and so $\varphi_0$) is optimal at time $t=0$ when $(x,m)\in \{((x_0,-1),(m_0,m_1^+)),$ $((x_0,1),(m_0,m_1^-))\}$, without the need of any further constraint.	
			
			\item For $t=0$, $x_0 \in \{1,-1\}$,  
			\begin{align}
				V_0(0,x_0)
				& = V_0(0,x_0,m_0)\nonumber \\
				& = \min_{\gamma \in \{0,1\}} \left\{ c_0 \gamma  + \mathbb{E}_0 \left[ V_0(1, (x_0,\Psi(x_0,\gamma,\xi_1),(m_0,\mu_1)) | X_0=x_0, \mu_0=m_0 \right]\right\}.
			\end{align}
			
			Finally, we study the initial time step in detail, exploiting the fact that $X_0$, $\xi_1$ and $\mu_1$ are independent on the probability space $(\Omega, \mathcal{F},\mathbb{P}_{0})$:
			\begin{align*}
				V_0(0,1)
				& = \min_{\gamma \in \{0,1\}} \left\{ c_0 \gamma  + \mathbb{E}_0 \left[ V_0(1, (1,\Psi(1,\gamma,\xi_1),(m_0,\mu_1)) | X_0=1, \mu_0=m_0 \right]\right\} \nonumber \\
				%%%%%%%%%%%%%%%%%%%%%%%%
				& = \min_{\gamma \in \{0,1\}} \left\{ c_0 \gamma  + \mathbb{E}_0 \left[ V_0(1, (1,\Psi(1,\gamma,\xi_1),(m_0,\mu_1)) | X_0=1 \right]\right\} \nonumber \\
				%%%%%%%%%%%%%%%%%%%%%%%%
				% & = \min_{\gamma \in \{0,1\}} \bigg\{ c_0 \gamma  + \beta \Big[\mathbb{P}_0 \left( \Psi(1,\gamma,\xi_1)=-1 \right)\mathbb{P}_0(\mu_1=m_1^+)\nonumber \\
				% & +\mathbb{P}_0 \left( \Psi(1,\gamma,\xi_1)=1 \right)\mathbb{P}_0(\mu_1=m_1^-) \Big] 
				% - \beta \Big[\mathbb{P}_0 \left( \Psi(1,\gamma,\xi_1)=1 \right)\mathbb{P}_0(\mu_1=m_1^+)\nonumber\\
				% &+\mathbb{P}_0 \left( \Psi(1,\gamma,\xi_1)=-1 \right)\mathbb{P}_0(\mu_1=m_1^-) \Big]\bigg\} \nonumber \\
				%%%%%%%%%%%%%%%%%%%%%%%%
				& = \min_{\gamma \in \{0,1\}} \bigg\{ c_0 \gamma  + \frac{\beta}{2} \Big[\mathbb{P}_0 \left( \Psi(1,\gamma,\xi_1)=-1 \right)\nonumber +\mathbb{P}_0 \left( \Psi(1,\gamma,\xi_1)=1 \right) \Big] \\
				& \qquad -  \frac{\beta}{2}  \Big[\mathbb{P}_0 \left( \Psi(1,\gamma,\xi_1)=1 \right)
				+ \mathbb{P}_0 \left( \Psi(1,\gamma,\xi_1)=-1 \right) \Big]\bigg\} \nonumber \\
				%%%%%%%%%%%%%%%%%%%%%%%%
				& = \min \Bigg\{ 0 + \frac{\beta}{2} \left[\left(\frac{1}{2} + \frac{1}{2}\right) - \left(\frac{1}{2} + \frac{1}{2}\right)\right], c_0 + \frac{\beta}{2} \left[\left(\frac{1}{4} + \frac{3}{4}\right) - \left(\frac{1}{4} + \frac{3}{4}\right)\right] \Bigg\}\nonumber \\
				&= \min  \{ 0, c_0 \}
			\end{align*}
			and, similarly,
			\begin{align*}
				V_0(0,-1)
				& = \min_{\gamma \in \{0,1\}} \left\{ c_0 \gamma  + \mathbb{E}_0 \left[ V_0(1, (-1,\Psi(-1,\gamma,\xi_1),(m_0,\mu_1)) | X_0=-1, \mu_0=m_0 \right]\right\} \nonumber \\
				%%%%%%%%%%%%%%%%%%%%%%%%
			%	& = \min_{\gamma \in \{0,1\}} \left\{ c_0 \gamma  + \mathbb{E}_0 \left[ V_0(1, (-1,\Psi(-1,\gamma,\xi_1),(m_0,\mu_1)) | X_0=-1 \right]\right\} \nonumber \\
				%%%%%%%%%%%%%%%%%%%%%%%%
			%	& = \min_{\gamma \in \{0,1\}} \bigg\{ c_0 \gamma  + \beta \Big[\mathbb{P}_0 \left( \Psi(-1,\gamma,\xi_1)=-1 \right)\mathbb{P}_0(\mu_1=m_1^+)\nonumber \\
			%	& +\mathbb{P}_0 \left( \Psi(-1,\gamma,\xi_1)=1 \right)\mathbb{P}_0(\mu_1=m_1^-) \Big] - \beta \Big[\mathbb{P}_0 \left( \Psi(-1,\gamma,\xi_1)=1 \right)\mathbb{P}_0(\mu_1=m_1^+)\nonumber\\
			%	&+\mathbb{P}_0 \left( \Psi(-1,\gamma,\xi_1)=-1 \right)\mathbb{P}_0(\mu_1=m_1^-) \Big]\bigg\} \nonumber \\
				%%%%%%%%%%%%%%%%%%%%%%%%
			%	& = \min_{\gamma \in \{0,1\}} \bigg\{ c_0 \gamma  + \frac{\beta}{2} \Big[\mathbb{P}_0 \left( \Psi(-1,\gamma,\xi_1)=-1 \right)\nonumber +\mathbb{P}_0 \left( \Psi(-1,\gamma,\xi_1)=1 \right) \Big] \\
			%	& -  \frac{\beta}{2}  \Big[\mathbb{P}_0 \left( \Psi(-1,\gamma,\xi_1)=1 \right) + \mathbb{P}_0 \left( \Psi(-1,\gamma,\xi_1)=-1 \right) \Big]\bigg\} \nonumber \\
				%%%%%%%%%%%%%%%%%%%%%%%%
				& = \min \Bigg\{ 0 + \frac{\beta}{2} \left[\left(\frac{1}{2} + \frac{1}{2}\right) - \left(\frac{1}{2} + \frac{1}{2}\right)\right], c_0 + \frac{\beta}{2} \left[\left(\frac{1}{4} + \frac{3}{4}\right) - \left(\frac{1}{4} + \frac{3}{4}\right)\right] \Bigg\}\nonumber \\
				&= \min  \{ 0, c_0 \}.
			\end{align*}
			Hence, at time $t=0$, $\gamma=0$ (and so $\varphi_0$) is optimal at any state.
		\end{itemize}
		Thus, we have proved that, conditionally on the event $\{\Phi=\varphi_0\}$, the strategy $\varphi_0$ is optimal, completing the analysis of the various cases.
		Now, putting together the conditions in Equations \eqref{Eq_ex_u_cond_on_c1}, \eqref{Eq_ex_u_cond_on_c0} and \eqref{Eq_ex_l_cond_c1}, we obtain the statement of the theorem.

	\end{proof}

%%%%%%%%%%%%%%%%%%%%%%%%%%%%%%%%%%%%%%%%%%%%%%%%%%%%%%%%%%%
%%%%%%%%%%%%%%%%%%%%%%%%%%%%%%%%%%%%%%%%%%%%%%%%%%%%%%%%%%%
%%%%%%%%%%%%%%%%%%%%%%%%%%%%%%%%%%%%%%%%%%%%%%%%%%%%%%%%%%%

%%%%%%%%%%%%%%%%%%%%%%%%%%%%%%%%%%%%%%%%%%%%%%%%%%%%%%%%%%%%%%%%%%%%%%%%%%%%%%%%%%%%%%%%%%%%%%%%%%%%%%%%%%%%%%%%%%%%%%%%%%%%%%%%%%%%%%%%%%%%%%%%%%%%%%%%%%%%%%%%%%%%%%%%%%%%%%%%%%%%%%%%%%%%%%%%%%%%%%%%%%%%%%%%%%%%%%%%%%%%%%%%%%%%%%%%%%%%%%%%%%%%%%%%%%%%%%%%%%%%%%%%%%%%%%%%%%%%%%%%%%%%%%%%%%%%%%%

\appendix

\section{Propagation of chaos} \label{SectAppendix}

	First of all, let us recall some basic definitions, for which we refer to \cite{Gottlieb}.
		We denote with $\Pi_n$ the set of permutations over $n$ elements, namely over $[\![1, n]\!].$  
		Consider a probability measure $p \in \mathcal {P(X)}$ and a sequence of symmetric probability measures $\{ p_n\}_{n \in \mathbb{N}}$,  with $p_n \in \mathcal{P}(\mathcal X^n)$, for each $n \in \mathbb{N}$.
			We call the sequence of probability measures $( p_n)_{n \in \mathbb{N}}$  \emph{$p$-chaotic} if for any choice of $k \in \mathbb N$ continuous and bounded functions on $\mathcal{X}$, $g_1, \dots,g_k$, we have
			\begin{equation}
				\lim_{n \to \infty} \int_{\mathcal{X}^n}g_1(s_1) \dots g_k(s_k) p_n(ds_1, \ldots, ds_n)=\prod_{j=1}^k \int_{\mathcal{X}}g_j(s) p(ds).
			\end{equation}
			Then, we call a sequence of symmetric probability measures $( p_n)_{n \in \mathbb{N}}$ \emph{chaotic}, if there exists a probability measure $p \in \mathcal{P(X)}$ s.t. $( p_n)_{n \in \mathbb{N}}$ is $p$-\emph{chaotic}.
			Let $(\beta_n(\cdot,\cdot))_{n \in \mathbb N}$ be a sequence of probability kernels such that, for any $n \in \mathbb N$, \hbox{$\beta_n: \mathcal X^n \times \mathcal{B(X)}^n \to [0,1]$} satisfies the following (symmetry) condition:
			\[
				\beta_n(x,B)
				=\beta_n(\pi x,\pi B), \quad \text{ for any } \pi \in \Pi_n.
			\]
			We say that \emph{propagation of chaos} holds for the sequence $(\beta_n(\cdot,\cdot))_{n \in \mathbb N}$ if $(Up_n)_{n \in \mathbb N}$ is chaotic for any chaotic sequence $(p_n)_{n \in \mathbb N}$ , where, for any $n \in \mathbb N$,
			\[
				Up_n(B):=\int_{\mathcal X^n}\beta_n(x,B)p_n(dx), \quad \text{ for all } B \in \mathcal{B(X)}^n.
			\]
		
		We are going to show that propagation of chaos holds in our case via the following equivalent characterization.
		
		\begin{theorem}[Theorem 4.2, in \cite{Gottlieb}]\label{Gott}
			Consider a couple of  complete and separable metric spaces, $(\mathcal{X}, d_{\mathcal{X}})$ and $(\mathcal{Y}, d_{\mathcal{Y}})$. 
			For each $n \in \mathbb N$, let $\Pi_n$ denote the set of permutations over $[\![1,n]\!]$. 
	Let $\beta_n: \mathcal{X}^n \times \mathcal{B}(\mathcal{Y}^n)\to[0,1]$  be a sequence of Markovian transition functions (probability kernels), i.e. for $x_n \in \mathcal{X}^n$ and $B \in  \mathcal{B}(\mathcal{Y}^n)$, $\beta_N(x_n, B)$ is the probability that the state of the n-particle system  lies in  $B$, given that the initial  state was  $x_n$. 
	Suppose that the transition functions satisfy the following condition:
	\begin{equation}\label{ipot}
		\beta_n(x_n,B)
		=\beta_n(\pi x_n,\pi B), \quad \text{ for all } \pi \in \Pi_n, \text{ for all } x_n \in \mathcal{X}^n \text{ and } \text{ for all } B \in  \mathcal{B}(\mathcal{Y}^n).
	\end{equation}
	Then, $\{ \beta_n \}_{n \in \mathbb{N}}$ propagates chaos if and only if, whenever $\mu_n(x_n):= \frac{1}{n}\sum_{j=1}^{n}\delta_{(x_n)^{j}}\to p$ in $\mathcal{P(X)}$ with $x_n \in \mathcal{X}^n$, then  $\{\widetilde{ \beta}_n(x_n, \cdot) \}_{n \in \mathbb{N}}$ is $F(p)$-chaotic, where $F: \mathcal{P(X)} \to \mathcal{P(Y)} $, is  a continuous function w.r.t. weak topologies and $\widetilde{ \beta}_n$ is defined as
	\[
		\widetilde{\beta}_n(x_n, B)= \frac{1}{n!}\sum_{\pi \in \Pi_n} \beta_n(x_n, \pi B). 
	\]
\end{theorem}

		Now, we should reframe the general definitions above in our context.
		Consider $x^N \in \mathcal{X}^{N}$ (initial conditions) and $B \in \mathcal{B}(\mathcal{X}^{N})$. 
		In our case, for an arbitrary fixed $N \in \mathbb N$, the probability kernel is given by
		\begin{equation}\label{Def:Eq_ker_chaos_p}
			\begin{split}
				\beta_N(x^N,B)&
				=\mathbb{P}_{N,m} \circ ({X}^{1,N,m}_1, \dots, {X}^{N,N,m}_1)^{-1}(B)\\& 
				=\mathbb{P}_{N,m} \bigg(\Big(\Psi(0, x^N_j, \frac{1}{N-1}\sum_{k\neq j}\delta_{x^N_k}, \Phi^{N,m}_j(0,x^N_j), \xi^{j,N,m}_1)\Big)_{j = 1}^N \in B\bigg),
			\end{split}
		\end{equation}
		where, in the second line, we have exploited the fact that $\Phi_1^{N,m}=\widetilde \Phi_1^{N,m}$, $\mathbb{P}_{N,m}$-a.s., and   that, since $\gamma_m^N=\rho_1(\cdot|m)^{\otimes N}$, $\Phi^{N,m}_j$ takes values in $\mathcal{R},$ for each $j \in [\![1, N]\!]$.
		
		We have the following propagation of chaos result:
		\begin{claim}\label{Claim:chaos_propagation}
			Propagation of chaos holds for the first time step of our model, i.e. $(\beta_N(\cdot, \cdot))_{N \in \mathbb{N}}$, as defined in Equation \eqref{Def:Eq_ker_chaos_p}, propagates chaos.
		\end{claim}

		\begin{proof}[Proof of Claim \ref{Claim:chaos_propagation}]
		First of all, we need to prove that condition (\ref{ipot}) in Theorem \ref{Gott} holds. 
		We denote with $\pi$ a generic permutation of $[\![1, N]\!]$.
		For any $x^N \in \mathcal{X}^{N}$ and $B=B_1 \times \ldots \times B_N \in \mathcal{B}(\mathcal{X}^{N})$, with $\pi B= B_{\pi(1)}\times \ldots \times B_{\pi (N)}$, we have
		\[
			\begin{split}
				&\beta_N(\pi x^N,\pi B)\\&
				=\mathbb{P}_{N,m} \bigg(\Big(\Psi(0, x^N_{\pi(j)}, \frac{1}{N-1}\sum_{k\neq j}\delta_{x^N_{\pi(k)}}, \Phi^{N,m}_j(0,x^N_{\pi(j)}), \xi^{j,N,m}_1)\Big)_{j = 1}^N \in \pi B\bigg)=\star.
			\end{split}
		\]
		Since $(\Phi^{N,m}_j)_{j=1}^N\overset{d}{\sim} \rho_1(\cdot|m)^{\otimes N}$ and $(\xi_1^{j,N,m})_{j=1}^N \overset{d}{\sim} \nu^{\otimes N}$ are independent, we reorder the terms to get
		\[
			\begin{split}
				\star &
				=\mathbb{P}_{N,m}\bigg(\Big(\Psi(0, x^N_{\pi(j)}, \frac{1}{N-1}\sum_{k\neq j}\delta_{x^N_{\pi(k)}}, \Phi^{N,m}_{\pi(j)}(0,x^N_{\pi(j)}), \xi^{\pi(j),N,m}_{1})\Big)_{j = 1}^N \in \pi B\bigg)\\&
				=\mathbb{P}_{N,m} \bigg(\Big(\Psi(0, x^N_j, \frac{1}{N-1}\sum_{k\neq j}\delta_{x^N_k}, \Phi^{N,m}_j(0,x^n_j), \xi^{j,N,m}_1)\Big)_{j = 1}^N \in B\bigg)=\beta_N(x^N,B).
			\end{split}
		\]
		Thus, we have shown that condition (\ref{ipot}) holds.
		Now, to conclude that $(\beta_N(\cdot,\cdot))_{N \in \mathbb{N}}$ propagates chaos we need to prove that, for any given sequence $x^N  \in \mathcal X^N$, $N \in \mathbb N$, such that $\mu_N(x^N):=\frac{1}{N}\sum_{j=1}^N \delta_{x^N_j} \to p$ in $\mathcal{P(X)}$, the sequence $(\widetilde{\beta}_N(x^N, \cdot))_{N=1}^\infty$, with   $\widetilde{ \beta}_N$  defined as
		\[
			\widetilde{\beta}_N(x^N, B)
			= \frac{1}{N!}\sum_{\pi \in \Pi_N} \beta_N(x_N, \pi B),  \qquad x^N \in \mathcal X^N, B \in \mathcal{B(X)}^N,
		\]
		is $F(p)$-chaotic, where $F: \mathcal{P(X)} \to 		\mathcal{P(X)}$ is a suitable continuous function.
		\\
		Suppose that $\mu_N(x^N)=\frac{1}{N}\sum_{j=1}^N \delta_{x^N_j} \to p$ in $\mathcal{P(X)}$, let us consider $g_1, \dots, g_l \in C_b(\mathcal{X})$, $l \in \mathbb{N}$, exploiting property (\ref{ipot}) we have
		\begin{displaymath}
			\begin{split}
				\int_{\mathcal{X}^N}g_1(y_1) \ldots g_l(y_l)\widetilde{\beta}_N(x^N,dy_1 \ldots dy_N)&
				= \frac{1}{N!}\sum_{\pi \in \Pi_N}		\int_{\mathcal{X}^N}g_1(y_1) \ldots g_l(y_l)	\beta_N(x^N,dy_{\pi(1)} \ldots dy_{\pi(N)})\\&
				= \frac{1}{N!}\sum_{\pi \in \Pi_N}\int_{\mathcal{X}^N}g_1(y_1) \ldots g_l(y_l)\beta_N(\pi x^N,dy_{1} \ldots dy_{N})=: \star
			\end{split}
		\end{displaymath}
		Now, we exploit the definition of $\beta_N(\cdot, \cdot)$ to gather terms together in order to get
		\begin{displaymath}
			\begin{split}
				\star&
				= \frac{1}{N!}\sum_{\pi \in \Pi_N}\int_{\mathcal{R}^N}\int_{\mathcal{Z}^N} \prod_{j=1}^l g_j(\Psi(0, x_{\pi(j)}^N, \frac{1}{N-1}\sum_{k \neq j}\delta_{x_{\pi(k)}^N}, \phi_j(0,x^N_{\pi(j)}),z_j)) \nu^{\otimes N}(dz_1,\dots, dz_N)\gamma^N_m(d\phi)
				\\&
				= \frac{1}{N!}\sum_{\pi \in \Pi_N} \prod_{j=1}^l \int_{\mathcal{R}}\int_{\mathcal{Z}} g_j(\Psi(0, x_{\pi(j)}^N, \frac{1}{N-1}\sum_{k \neq j}\delta_{x_{\pi(k)}^N}, \phi(0,x^N_{\pi(j)}),z))\nu(dz)\rho_1(d\phi|m)\\&
				= \frac{1}{N!}\sum_{\pi \in \Pi_N} \prod_{j=1}^l \int_{\mathcal{R}}\int_{\mathcal{Z}} g_j(\Psi(0, x_{\pi(j)}^N, \frac{N}{N-1}\mu_N(x^N)-\frac{1}{N-1}\delta_{x_{\pi(j)}^N},\phi(0,x^N_{\pi(j)}),z))\nu(dz)\rho_1(d\phi|m)\\&
				= \frac{(N-l)!}{N!}\sum_{\lambda \in \mathcal{I}_{N:l}} \prod_{j=1}^l \int_{\mathcal{R}}\int_{\mathcal{Z}} g_j(\Psi(0, x_{\lambda(j)}^N, \frac{N}{\Nnn}\mu_N(x^N)-\frac{1}{\Nnn}\delta_{x_{\lambda(j)}^N},\phi(0,x^N_{\lambda(j)}),z))\nu(dz)\rho_1(d\phi|m)\\&
				=:\diamond,
			\end{split}
		\end{displaymath}
		where $\mathcal{I}_{N:l}$ denotes the set of injections from $[\![1,l]\!]$ to $[\![1,N]\!]$.
		\\
		Set $\mu_{N:l}$ to be, for a vector $x^N \in \mathcal{X}^N$, the symmetric probability measure given by
		\begin{equation}
			\mu_{N:l}(x^N)=\frac{(N-l)!}{N!}\sum_{\lambda \in \mathcal{I}_{N:l}}\delta_{(x^N_{\lambda(1)},\dots,x^N_{\lambda(l)})}.
		\end{equation}
		It is possible to show, see \cite{Gottlieb} pg. 29, that $\mu_{N}(x^N)\underset{N \to \infty}{\longrightarrow}p$ implies $\mu_{N:l}(x^N)\underset{N \to \infty}{\longrightarrow}p^{\otimes l}$.
		We have
		\begin{displaymath}
			\begin{split}
				\diamond&
				= \frac{(N-l)!}{N!}\sum_{\lambda \in \mathcal{I}_{N:l}} \prod_{j=1}^l \int_{\mathcal{R}}\int_{\mathcal{Z}} g_j(\Psi(0, x_{\lambda(j)}^N, \frac{N}{\Nnn}\mu_N(x^N)-\frac{1}{\Nnn}\delta_{x_{\lambda(j)}^N},\phi(0,x^N_{\lambda(j)}),z))\nu(dz)\rho_1(d\phi|m)\\&
				=\int_{\mathcal{X}^l} \prod_{j=1}^l \int_{\mathcal{R}}\int_{\mathcal{Z}} g_j(\Psi(0, y_j, \frac{N}{\Nnn}\mu_N(x^N)-\frac{1}{\Nnn}\delta_{y_j}, \phi(0, y_j),z))\nu(dz)\rho_1(d\phi|m) \mu_{N:l}(x^N)(dy_1,\ldots, dy_l)\\&
				=\int_{\mathcal{X}^l}\mu_{N:l}(x^N)(dy) \bigg\{ \prod_{j=1}^l \int_{\mathcal{R}} \rho_1(d\phi|m) \int_{\mathcal{Z}}\nu(dz)  g_j(\Psi(0, y_j, \frac{N}{\Nnn}\mu_N(x^N)-\frac{1}{\Nnn}\delta_{y_j}, \phi(0, y_j),z))  \bigg\}\\&
				\underset{N \to \infty}{\longrightarrow} 
\int_{\mathcal{X}^l}p^{\otimes l}(dy) \bigg\{ \prod_{j=1}^l \int_{\mathcal{R}} \rho_1(d\phi|m) \int_{\mathcal{Z}}\nu(dz)  g_j(\Psi(0, y_j, p , \phi(0, y_j),z))  \bigg\}\\&
				\qquad = \prod_{j=1}^l \int_{\mathcal{X}}p(dy) \int_{\mathcal{R}} \rho_1(d\phi|m) \int_{\mathcal{Z}}\nu(dz)   g_j(\Psi(0, y, p , \phi(0, y),z)) 
				= \prod_{j=1}^l \int_{\mathcal{X}}g_j(x) q(p)(dx),
			\end{split}
		\end{displaymath}
		where $q(p)$ is the image of $(p,\rho_1(\cdot|m),\nu)$ via the mapping $(y,\phi,z) \mapsto \Psi(0,y,p,\phi(0,y),z)$.
		In particular, the convergence in the fourth line is proved as follows, exploiting a generalization of the continuous mapping theorem, namely \cite[Theorem I.5.5]{Billingsley}.
		In the notation of \cite[Theorem I.5.5]{Billingsley}, we have $\mathbf{P}_N = \mu_{N:l}(x^N) \overset{N \to \infty}{\longrightarrow}p^{\otimes l}$, by assumption.
		Furthermore, we consider the following functions      
$h_N: \mathcal X^l \to [-\prod_{j=1}^l\|g_j\|_{\infty},\prod_{j=1}^l\|g_j\|_{\infty}]$, for all $N \in \mathbb N$, and 
$h: \mathcal X^l \to [-\prod_{j=1}^l\|g_j\|_{\infty},\prod_{j=1}^l\|g_j\|_{\infty}]$, 
defined, for $y \in \mathcal{X}$, by
		\[
			\begin{split}
				&h_N(y)
				:=\prod_{j=1}^l \int_{\mathcal{R}}			\int_{\mathcal{Z}} g_j(\Psi(0, y_j, \frac{N}{\Nnn}\mu_N(x^N)-		\frac{1}{\Nnn}\delta_{y_j}, \phi(0, y_j),z))\nu(dz)\rho_1(d\phi|m), 
				\\&
				h(y)
				:=\prod_{j=1}^l \int_{\mathcal{R}}		\int_{\mathcal{Z}} g_j(\Psi(0, y_j, p, \phi(0, y_j),z))\nu(dz)	\rho_1(d\phi|m).
			\end{split}
		\] 
		We have that both previous functions are measurable, since the finite set $\mathcal X^l$ is equipped with the discrete metric.
		Finally, we show that, for any $y \in \mathcal X^l$, $h_N(y)\to h(y)$, as $N \to \infty$. 
		We prove this for $l=2$, but the result can be extended to any $l \in \mathbb N$. 
		In the following we exploit the notation $\bar{\epsilon}_{N,j}:=\frac{N}{N-1}\mu_N(x^N)-\frac{1}{N-1}\delta_{y_j}$. 
		Exploiting the fact that $g_1, g_2 \in \mathcal C_b(\mathcal X)$ and \hypref{HypPsiContinuity}, we have
		\begin{align*}
				&|h_N(y)-h(y)| \\
				&=\Big|\prod_{j=1}^2 \int_{\mathcal{R}}\int_{\mathcal{Z}} g_j(\Psi(0, y_j, \bar{\epsilon}_{N,j}, \phi(0, y_j),z))\nu(dz)\rho_1(d\phi|m)-\prod_{j=1}^2 \int_{\mathcal{R}}\int_{\mathcal{Z}} g_j(\Psi(0, y_j, p, \phi(0, y_j),z))\nu(dz)\rho_1(d\phi|m)\Big|
				\\&
				\leq \Big|\prod_{j=1}^2 \int_{\mathcal{R}}\int_{\mathcal{Z}} g_j(\Psi(0, y_j, \bar{\epsilon}_{N,j}, \phi(0, y_j),z))\nu(dz)\rho_1(d\phi|m)
				\\&
				\qquad- 
				\int_{\mathcal{R}}\int_{\mathcal{Z}} g_1(\Psi(0, y_1, \bar{\epsilon}_{N,1}, \phi(0, y_1),z))\nu(dz)\rho_1(d\phi|m)\int_{\mathcal{R}}\int_{\mathcal{Z}} g_2(\Psi(0, y_2, p, \phi(0, y_2),z))\nu(dz)\rho_1(d\phi|m)
				\Big|\\&
				\quad +\Big|
				\int_{\mathcal{R}}\int_{\mathcal{Z}} g_1(\Psi(0, y_1, \bar{\epsilon}_{N,1}, \phi(0, y_1),z))\nu(dz)\rho_1(d\phi|m)\int_{\mathcal{R}}\int_{\mathcal{Z}} g_2(\Psi(0, y_2, p, \phi(0, y_2),z))\nu(dz)\rho_1(d\phi|m)
				\\&
				\qquad-
				\prod_{j=1}^2 \int_{\mathcal{R}}\int_{\mathcal{Z}} \phi_j(\Psi(0, y_j, p, \phi(0, y_j),z))\nu(dz)\rho_1(d\phi|m)
				\Big|\\&
				\leq \|g_1\|_{\infty} \Big|\int_{\mathcal{R}}\int_{\mathcal{Z}} g_2(\Psi(0, y_2, \bar{\epsilon}_{N,2}, \phi(0, y_2),z))- g_2(\Psi(0, y_2, p, \phi(0, y_2),z))\nu(dz)\rho_1(d\phi|m)\Big|
				\\&
				\qquad
				+ \|g_2\|_{\infty}
				\Big|\int_{\mathcal{R}}\int_{\mathcal{Z}} 		g_1(\Psi(0, y_1, \bar{\epsilon}_{N,1}, \phi(0, y_1),z))- 	g_1(\Psi(0, y_1, p, \phi(0, y_1),z))\nu(dz)\rho_1(d\phi|m)\Big|
				\\&
				\leq 2\|g_1\|_{\infty}\|g_2\|_{\infty}
				\Big|\int_{\mathcal{R}}\int_{\mathcal{Z}} \bold{1}_{\Psi(0, y_2, \bar{\epsilon}_{N,2}, \phi(0, y_2),z)\neq \Psi(0, y_2, p, \phi(0, y_2),z)}\nu(dz)\rho_1(d\phi|m)\Big|
				\\&
				\qquad
				+2\|g_1\|_{\infty}\|g_2\|_{\infty}
				\Big|\int_{\mathcal{R}}\int_{\mathcal{Z}} 		\bold{1}_{\Psi(0, y_1, \bar{\epsilon}_{N,1}, \phi(0, y_1),z)\neq 				\Psi(0, y_1, p, \phi(0, y_1),z)}\nu(dz)\rho_1(d\phi|m)\Big|
				\\&
				\leq2\|g_1\|_{\infty}\|g_2\|_{\infty}
				(w(dist( \bar{\epsilon}_{N,1},p))+w(dist( \bar{\epsilon}_{N,2},p)))
\overset{N \to \infty}{\longrightarrow}0.
			\end{align*}
		Indeed, $\lim_{s\to 0^+}w(s)=0$ and, for any 
		$
		j \in \{1,2\},
		\text{dist}( \bar{\epsilon}_{N,j},p)
		\leq \text{dist}( \bar{\epsilon}_{N,j}, \mu_N(x^N))
		+ \text{dist}(\mu_N(x^N),p).
		$
		\\
		The second term on the right vanishes as $N \to \infty$ by assumption and
		\[
			\begin{split}
				\text{dist}\left( \frac{N}{\Nnn}\mu_N(x^N)-\frac{1}{\Nnn}\delta_{y_j}, \mu_N(x^N)\right)&
				=\frac{1}{2} \sum_{z \in \mathcal X}| \frac{N}{\Nnn}\mu_N(x^N)(z)-\frac{1}{\Nnn}\delta_{y_j}(z)-\mu_N(x^N)(z)|\\&
				=\frac{1}{2(N-1)} \sum_{z \in \mathcal X}| \mu_N(x^N)(z)-\delta_{y_j}(z)|
				\leq \frac{1}{N-1}\to 0.
			\end{split}
		\]

		Thence, an application of \cite[Theorem I.5.5]{Billingsley} yields the desired convergence.
		To conclude we need to show that the function $q: \mathcal{P(X)} \to \mathcal{P(X)}$, defined,  for $p \in \mathcal{P(X)}$, as the image of $(p,\rho_1(\cdot|m),\nu)$ via the mapping $(y,\phi,z) \mapsto \Psi(0,y,p,\phi(0,y),z)$, is a continuous function of $p$. 
		This function $q(p)$ corresponds to the function $F(p)$ in the statement of  Theorem \ref{Gott}. 
		Let's consider a sequence $\{ p_n \}_{n \in \mathbb{N}}\subseteq \mathcal{P(X)}$, such that $p_n \underset{N \to \infty}{\longrightarrow} p$ weakly and let $B \in \mathcal{B(X)}$. 
		Exploiting hypothesis \hypref{HypPsiContinuity}, we are able to  deduce 
		\begin{align*}
				&|q(p_n)(B)-q(p)(B)|
				\\&
				\leq \bigg|
				\int_{\mathcal{X}}p_n(dy) \int_{\mathcal{R}} \rho_1(d\phi|m) \int_{\mathcal{Z}}\nu(dz)  \bold{1}_{B}(\Psi(0, y, p_n , \phi(0, y),z))
				\\&
				\qquad -
				\int_{\mathcal{X}}p(dy) \int_{\mathcal{R}} \rho_1(d\phi|m) \int_{\mathcal{Z}}\nu(dz)  \bold{1}_{B}(\Psi(0, y, p , \phi(0, y),z))
				\bigg|\\&
				\leq \bigg|
				\int_{\mathcal{X}}p_n(dy) \int_{\mathcal{R}} \rho_1(d\phi|m) \int_{\mathcal{Z}}\nu(dz) \Big\{ \bold{1}_{B}(\Psi(0, y, p_n , \phi(0, y),z))
- \bold{1}_{B}(\Psi(0, y, p , \phi(0, y),z))\Big\} 
\bigg| \\&
				\qquad + \bigg|\int_{\mathcal{X}}(p_n-p)(dy) \int_{\mathcal{R}} \rho_1(d\phi|m) \int_{\mathcal{Z}}\nu(dz)   \bold{1}_{B}(\Psi(0, y, p , \phi(0, y),z)   
				\bigg|\\&
				\leq 
				\int_{\mathcal{X}}p_n(dy) \int_{\mathcal{R}} \rho_1(d\phi|m) \int_{\mathcal{Z}}\nu(dz)   \bold{1}_{\{ \Psi(0, y, p_n , \phi(0, y),z) \neq \Psi(0, y, p , \phi(0, y),z)\}}
 				 \\&
 				\qquad + 
				\bigg|\int_{\mathcal{X}}(p_n-p)(dy) 			\int_{\mathcal{R}} \rho_1(d\phi|m) \int_{\mathcal{Z}}\nu(dz)  	\bigg|
				\\&
				\leq \int_{\mathcal{X}}p_n(dy) 		\int_{\mathcal{R}} \rho_1(d\phi|m) w(\text{dist}(p_n,p)) 
				 + \int_{\mathcal{X}}|p_n-p|(dy)  
				\\&
				\leq w(\text{dist}(p_n,p))+\text{dist}(p_n,p).
		\end{align*}
		This fact, in particular, implies 
		\[
			\begin{split}
				\text{dist}(q(p_n),q(p))&
				= d_{TV}(q(p_n),q(p))=\sup_{B \in \mathcal{B(X)}} |q(p_n)(B)-q(p)(B)| 
				\\& 
				\leq  w(\text{dist}(p_n,p))+\text{dist}(p_n,p)
				\underset{N \to \infty}{\longrightarrow}0,
			\end{split}
		\]
		where $d_{TV}$ denotes the distance in total variation, 
that coincides with the distance $\text{dist}(\cdot,\cdot)$, 
compatible with weak topology, because the set $\mathcal{X}$ is finite.
		So, we get the continuity of $q$ and conclude the proof of chaos propagation.
		\end{proof}

%%%%%%%%%%%%%%%%%%%%%%%%%%%%%%%%%%%%%%%%%%
\newpage
\bibliographystyle{plain}
\bibliography{references}
%%%%%%%%%%%%%%%%%%%%%%%%%%%%%%%%%%%%%%%%%%

\end{document}